\documentclass{amsart}
\usepackage[latin1]{inputenc}
\usepackage[T1]{fontenc}
\usepackage[french]{babel}
\usepackage{amsfonts}
\usepackage{amsmath}
\usepackage{amsthm}
\usepackage{amssymb}
\usepackage{amsopn}
\usepackage{upref}
\usepackage{mathrsfs}
\usepackage{ae, aecompl}
\usepackage{graphicx}
\usepackage{bbm}

\usepackage{afterpage}

\usepackage{fancyhdr}
\usepackage{lmodern}
\usepackage{babelbib}
\usepackage{enumerate}
\usepackage{url}
\selectbiblanguage{french}
\usepackage[all,cmtip]{xy}
\usepackage{dsfont}

\DeclareMathOperator{\Rep}{Rep}

\DeclareMathOperator{\cl}{cl}
\DeclareMathOperator{\prim}{prim}

\DeclareMathOperator{\GL}{GL}
\DeclareMathOperator{\CHM}{CHM}
\DeclareMathOperator{\DM}{DM}

\DeclareMathOperator{\CH}{CH}
\DeclareMathOperator{\HW}{H}

\DeclareMathOperator{\HS}{HS}

\DeclareMathOperator{\GSp}{GSp}
\DeclareMathOperator{\NUM}{NUM}
\DeclareMathOperator{\Gal}{Gal}

\DeclareMathOperator{\Immaginario}{Im}

\DeclareMathOperator{\interiore}{int}

\DeclareMathOperator{\End}{End}
\DeclareMathOperator{\Hom}{Hom}
\DeclareMathOperator{\num}{num}
\DeclareMathOperator{\Id}{Id}

\DeclareMathOperator{\Sym}{Sym}

\DeclareMathOperator{\Spec}{Spec}

\DeclareMathOperator{\ab}{ab}
\DeclareMathOperator{\cris}{cris}

\renewcommand{\Im}{\Immaginario}

\renewcommand{\int}{\interiore}
\renewcommand{\tilde}{\widetilde}

\newtheoremstyle{theorem}{11pt}{11pt}{\itshape}{}{\bfseries}{.}{.5em}{}
\newtheoremstyle{note}{11pt}{11pt}{}{}{\bfseries}{.}{.5em}{}

\theoremstyle{note}
\newtheorem{defin}{Définition}[section]
\newtheorem{exemple}[defin]{Exemple}
\newtheorem{rem}[defin]{Remarque}
\newtheorem{paragrafo_numerato}[defin]{}
\newtheorem{paragrafo_numerato_nome}[defin]{}

\theoremstyle{theorem}
\newtheorem{thm}[defin]{Théorème}
\newtheorem{cor}[defin]{Corollaire}
\newtheorem{lem}[defin]{Lemme}
\newtheorem{prop}[defin]{Proposition}
\newtheorem{conj}[defin]{Conjecture}

\newcommand{\Q}{\mathbb{Q}}
\newcommand{\Z}{\mathbb{Z}}
\newcommand{\R}{\mathbb{R}}

\newcommand{\C}{\mathbb{C}}

\newcommand{\Hdg}{\textrm{Hdg}}
\newcommand{\Sm}{\textrm{Sm}}
\newcommand{\PSh}{\textrm{PSh}}

\newcommand{\cM}{\mathcal M}

\newcommand{\bC}{\mathbb{C}}

\newcommand{\Mot}{{\cM ot}}
\newcommand{\Mat}{{\cM at}}

\newcommand{\GrVect}{\mathrm{GrVect}}
\newcommand{\Fil}{\mathrm{Fil}}

\newcommand{\dR}{\mathrm{dR}}

\newcommand{\eval}{\mathrm{eval}}
\newcommand{\degtr}{\mathrm{degtr}}

\makeatletter
\newcommand{\oset}[3][0.2ex]{%
	\mathrel{\mathop{#3}\limits^{
			\vbox to#1{\kern-2\ex@
				\hbox{$\scriptstyle#2$}\vss}}}}
\makeatother

\newenvironment{dedication}
{
   \cleardoublepage
   \thispagestyle{empty}
   \vspace*{\stretch{1}}
   \hfill\begin{minipage}[t]{0.66\textwidth}
   \raggedright
}%
{
   \end{minipage}
   \vspace*{\stretch{3}}
   \clearpage
}

\DeclareMathOperator{\Aut}{Aut}

\DeclareMathOperator{\Isom}{Isom}
\DeclareMathOperator{\Emb}{Emb} 
\DeclareMathOperator{\Nat}{Nat}

\DeclareMathOperator{\id}{\mathsf{id}}

\DeclareMathOperator{\tr}{tr}

\newcommand{\isocan}{\xrightarrow{\hspace{1.85pt}\sim \hspace{1.85pt}}}
\numberwithin{equation}{section}

\author{Giuseppe Ancona}

\address{IRMA\\
Université de Strasbourg
7, rue René Descartes  \\
67000 Strasbourg}
\email{ancona@math.unistra.fr}
\urladdr{http://irma.math.unistra.fr/~ancona/}
 
\begin{document}

 \title[Mémoire HdR]{Quelques aspects arithmétiques et géométriques des cycles algébriques et des motifs}

\maketitle

\setcounter{tocdepth}{1}

\begin{center}
{Habilitation à diriger les recherches}
 \end{center}

\vspace{1cm}

Soutenue le 17 Novembre 2022 devant le jury composé de :

\vspace{1cm}

Yves André, examinateur

Joseph Ayoub, rapporteur

  Jean-Benoît Bost, examinateur
  
Anna Cadoret, rapporteur

François Charles, rapporteur

  Carlo Gasbarri, examinateur
  
  Rutger Noot, garant

\tableofcontents

\setcounter{section}{-1}
\setcounter{tocdepth}{1}

\newpage
  \null
    \thispagestyle{empty}
    \newpage
\section*{Remerciements} 

On m'avait souvent décrit l'Habilitation comme une formalité par laquelle il fallait passer. Je me rends compte qu'en fait d'une part j'ai pris plaisir à rédiger ce mémoire et d'autre part c'est avec un peu d'émotion que je m'approche à cette soutenance. Je serai entouré par des collègues et des amis qui me sont chers et j'aimerais par ces quelques lignes les remercier.

\

Merci aux trois rapporteurs d'avoir accepté de consacrer leur temps à la lecture de mon mémoire. Joseph Ayoub a été mon encadrant de postdoc et c'est une figure centrale de la théorie des motifs. J'ai bénéficié de nombreuses discussions avec lui et c'est un grand honneur pour moi l'attention qu'il a porté à mes travaux. 

J'ai rencontré Anna Cadoret pendant ma thèse à l'occasion d'un groupe de travail organisé par son ANR. J'ai le souvenir d'une atmosphère détendue et ouverte vers les jeunes, comme il arrive rarement dans notre métier (atmosphère que l'on a essayé de reproduire avec la RéGA, puis le JAVA). J'ai toujours aimé l'originalité des questions qu'elle sait se poser et je me réjouis de ses appréciations à ce mémoire. 

François Charles était 2 ans avant moi à l'ENS, il a représenté pour plusieurs d'entre nous une  référence. J'ai toujours été impressionné par ses fortes intuitions et je le remercie pour les échanges que l'on a eu et  pour son intérêt   pour mes résultats.

\ 

Rutger Noot a été rapporteur de ma thèse et m'a beaucoup encouragé après : je lui en suis reconnaissant. Depuis que je suis à Strasbourg, j'ai pu profiter de sa culture mathématique à plusieurs reprises. C'était pour moi le choix naturel de garant et je le remercie d'avoir accepté. 

Il n'y a probablement aucune section de ce mémoire où le nom d'Yves André n'apparaisse pas. Bien que j'essaie constamment d'élargir mes intérêts et m'ouvrir à de nouveaux domaines, je finis toujours par y découvrir un théorème fondamental qui lui est dû. Sa présence dans mon jury est un honneur pour moi. 

J'ai suivi un cours de Jean-Benoît Bost en M2 : sa clarté et sa vision sont devenues    une référence pour moi. Je le remercie pour son soutien depuis cette époque ; c'est un plaisir pour moi de l'avoir dans mon jury d'Habilitation.  

Les déjeuners à la cantine sont souvent enrichis par les digressions mathématiques (toujours passionnées) de Carlo Gasbarri. Je le remercie pour celles que l'on a eu, celles que l'on aura et  d'avoir accepté de faire partie du jury.

\

Je remercie Emiliano Ambrosi,  Pierre Baumann, Mattia Cavicchi, Frédéric Chapoton, Dragos Fratila,   Florence Lecomte et Rutger Noot pour leurs relectures aux premières versions de ce texte.

\

Certains des résultats présentés ici sont issus de collaborations avec des mathématiciens qui m'ont beaucoup appris. Je les remercie tous ici, à partir de ma première collaboratrice, Annette Huber, avec laquelle nous continuons à beaucoup échanger, jusqu'au plus jeune, Mattia Cavicchi, mathématicien doué à qui je souhaite tout le meilleur pour ses débuts de carrière.

Mes recherches mathématiques n'auraient pas pu avoir lieu sans les échanges, les conseils et les remarques de mes chers amis mathématiciens Olivier Benoist, Yohan Brunebarbe, Javier Fres\'an et Marco Maculan. 

Quand je repense à mon passé mathématique je dois remercier mon directeur de thèse Jörg Wildeshaus, ainsi que Frédéric Déglise qui était chercheur à Paris 13 à l'époque de ma thèse et qui m'a beaucoup appris.

\

Pendant ces années à l'IRMA j'ai eu le plaisir de discuter de mathématiques avec de nombreux collègues, tant dans mon équipe que dans les autres. Je pense en particulier à Emilano Ambrosi, Dragos Fratila, Robert Laterveer, Yohann Le Floch, Adriano Marmora, Pierre Py et Ana Rechtman. 

L'ambiance au laboratoire est bonne, j'ai échangé avec à peu près tout le monde à l'occasion d'enseignements, conseils ou autre et j'ai souvent trouvé de l'empathie et de la bienveillance. Chers collègues qui lisez ces lignes, je vous en remercie ! 

Le travail d'enseignant-chercheur a besoin d'un soutien administratif et technique considérable. Nous sommes gâté de ce point de vue-là à Strasbourg : je remercie en  particulier  Alexandra Carminati, Sandrine Cerdan, Pascale Igot, Delphine Karleskind, Jessica Maurer-Spoerk, Alexis Palaticky et Alain Sartout.

\

Je suis ému par la présence de plusieurs amis et collègues qui n'habitent pas à Strasbourg et qui ont fait de la route pour venir voir ma soutenance : Javier, Mattia, Olivier, Quentin et Yohan, merci ! Je suis également touché par la présence d'amis qui ont décidé de passer une heure de leur vie à me voir délirer devant des tableaux bleus : Joanne, Cédric, François... cela me touche beaucoup.

\

Enfin, merci à tout ce qui est non-mathématique dans ma vie et qui la rend si heureuse. Je pense aux amis du tango de la belle association Hermosa : merci pour la bonne ambiance des milongas. Je pense aux amis du foot du terrain 5 : merci pour les matchs intenses et pour les après-matchs encore plus intenses. Je pense aux jogging avec  Gianluca (accompagnés toujours de questionnements et conseils réciproques). Je pense aux jeunes parents avec lesquels on partage les apéros où l'on ne peut jamais terminer nos phrases   ainsi que les fêtes regrettées le lendemain. Et évidemment je pense à ma famille pétillante. Merci à toutes les générations de la Nonna Michelina qui me chantait l'opéra pendant qu'elle préparait ses pâtes  jusqu'à mes filles qui sautent toujours à mon cou. Merci à Anna qui a toujours envie de   me voir donner un exposé et toujours la joie de le célébrer. Merci à ma mère qui nous a préparé un super pot (vous verrez...), à mon père qui me parlait de sciences en voiture pendant qu'il m'accompagnait au judo, à mes frères et à notre chambre à trois où on pratiquait tout sport possible, à  la famille de Maglie qui nous a toujours accompagné avec amour et à ma belle famille avec qui on partage voyages, bonnes bouffes et discussions sur le monde.

\newpage

\begin{dedication}
A Anna, compagna di viaggio.

A Camilla e Gaia, il nostro viaggio.
\end{dedication}

\newpage
  \null
    \thispagestyle{empty}
    \newpage

\section{Introduction}

 Le protagoniste de ce mémoire est un morphisme d'anneaux gradués\footnote{La graduation dans $\CH(X)$ est induite par la codimension des sous-variétés et la multiplication est appelée produit d'intersection.} appelé application classe de cycle
 \[\cl_X : \CH(X) \longrightarrow \HW(X).\]
 Ici $X$ est une variété projective et lisse définie sur un corps $k$, $\HW$ est une cohomologie de Weil\footnote{Par exemple si $k=\bC$ on pourra choisir la cohomologie singulière ou  pour $k$ de caractéristique $p\geq 0$ et $\ell$ un nombre premier tel que $\ell\neq p$ on pourra choisir la cohomologie $\ell$-adique.} et $ \CH(X)$ est l'anneau de Chow à coefficients rationnels.
 Il faut penser à $\CH(X)$ comme un invariant de $X$  de nature algébrique et à $\HW(X)$ comme un invariant   de nature topologique : l'application $\cl_X$ compare ces différentes natures.
 
 Les questions autour de $\cl_X$ peuvent se diviser grossièrement en trois classes :
 \begin{enumerate}
 \item Décrire l'image de $\cl_X$,
 \item Décrire le noyau de $\cl_X$,
 \item Décrire la structure multiplicative de $\cl_X$.
 \end{enumerate}
 
 Par exemple, des conjectures qui entrent dans la première classe sont Hodge, Tate, la conjecture des périodes de Grothendieck, ou encore les conjectures standard de type Künneth et de type Lefschetz. 
 Dans la deuxième classe on trouve la conjecture de Bloch--Beilinson ou la conjecture de nilpotence.
 Certaines conjectures sont à cheval entre la première et la deuxième classe, par exemple la conjecture de conservativité ou la conjecture sur la dimension finie de Kimura. 
 Rentrent  dans la troisième classe les conjectures standard de type Hodge et de type  \og$\hom=\num$\fg.
 
 \
 
 Pour décomposer l'étude dans ces trois classes on peut d'abord factoriser $\cl_X$ (comme morphisme d'anneaux) et obtenir le diagramme  suivant :
  \[
\xymatrix{
    \CH(X)  \ar@{->>}[d]^{p} \ar[rd]^{\cl_X} &   \\
      \CH(X)/ \ker \cl_X \ar@{^{(}->}[r]^{\hspace{0.5cm}i}  & \HW(X) .
   }
    \]
    L'anneau quotient $ \CH(X)/ \ker \cl_X $ est également noté
    $ \CH(X)/ \hom$ et appelé anneaux des cycles modulo l'équivalence homologique.

Les questions appartenant à la classe (1) ci-dessus reviennent alors à l'étude de l'injection $i$ et celles de la classe (2) à la projection $p$. 
Pour exprimer la  classe (3) on considère l'équivalence numérique\footnote{Cette équivalence rend tous les points de la variété équivalents et les cycles de codimension complémentaire sont mis en dualité par le produit d'intersection.} et on complète le diagramme ainsi :

 \begin{align}\label{diagramme chow}
\xymatrix{
    \CH(X)  \ar@{->>}[d]^{p} \ar[rd]^{\cl_X} &   \\
      \CH(X)/ \hom \ar@{^{(}->}[r]^i \ar@{->>}[d]^{q} & \HW(X). \\
    \CH(X)/ \num & 
  }
    \end{align}
    La compréhension de  l'algèbre $\CH(X)/ \num$ et de la projection $q$ encode une bonne partie des questions de la  classe (3) ci-dessus.

\

Les conjectures principales sur les cycles algébriques s'expriment mieux  si on passe aux motifs. Cela revient, grosso-modo, à considérer le diagramme \eqref{diagramme chow} pour toutes les variétés $X$ à la fois : 
     
\begin{align}\label{diagramme motifs}
\xymatrix{
    \CHM(k)  \ar@{->>}[d]^{\pi} \ar[rd]^{R} &   \\
      \Mot(k)  \ar@{^{(}->}[r]^I \ar@{->>}[d]^{\pi'} & \GrVect. \\
    \NUM(k) & 
  }
    \end{align}
 
Ici $R$ indique la réalisation des motifs de Chow vers les espaces vectoriels gradués : c'est une collection d'applications classe de cycle.
On a de plus des foncteurs pleins $\pi,\pi'$ de projection vers les motifs homologiques et les motifs numériques et un foncteur fidèle $I$ des motifs homologiques  vers les espaces vectoriels gradués, également appelé réalisation.

\ 
    
    \
    
    Le théorie des motifs est utile non seulement à la formulation de questions sur les cycles algébriques mais  aussi à la démonstration de résultats sur ces derniers, voir la Section \ref{pourquoi motifs} pour une discussion de ce point. 
    En guise d'exemple, voici une petite liste de résultats - certains issus de mes travaux - où l'on remarquera que les motifs n'apparaissent pas dans les énoncés et pourtant  sont cruciaux dans leurs preuves.
\begin{thm}\label{thm intro}
 \begin{enumerate}
 \item (Kahn \cite{Kahn}) L'application classe de cycle est injective pour les produits de courbes elliptiques sur un corps fini.
\item (Kimura \cite{Kim}) Soit $S$ un surface complexe projective et lisse dominée par un produit de courbes. Si l'application classe de cycle est surjective alors elle est injective.
\item \cite{Anc21} Soit  $A$ une variété abélienne de dimension quatre, alors le produit d'intersection
\[\CH^2(A)/{\num} \times \CH^2(A)/{\num}  \longrightarrow \Q\]
est de signature $(\rho_2 - \rho_1 +1; \rho_1 - 1)$, où
$\rho_n=\dim_\Q (\CH^n(A)/{\num}).$ 
 \item \cite{AHP}  Soient $S$ un schéma de base régulier et $G$ un $S$-schéma en groupes commutatifs. Alors l'action du morphisme
\[n_G: G \longrightarrow G\]
de  multiplication par $n$
décompose l'espace $\CH(G) $ en une somme finie de sous-espaces propres (de plus les valeurs propres sont des  puissances explicites  de $n$ et la décomposition ne dépend pas de l'entier $n\geq 2$ choisi).
 \end{enumerate}
 \end{thm}
 
 Par sa nature même, la théorie des motifs se mélange aux différentes cohomologies de Weil que l'on peut considérer. Ainsi, en fonction du corps de base $k$, on se retrouve à utiliser  la théorie de Hodge ($k=\bC$), les représentations galoisiennes ($k$ corps de nombres), l'arithmétique des nombres de Weil ($k$ fini) ou encore de la théorie de Hodge $p$-adique (en caractéristique mixte). 
 Cela dégage les aspects arithmétiques et géométriques de la théorie. 
 
 Ces différentes théories cohomologiques ont des analogies, que l'on retrouve dans les motifs par des théorèmes ou des conjectures,  ainsi que des différences, que la théorie des motifs vise à réparer, voir la Section \ref{coho Weil}.
 
 \
 
 On peut distinguer les différents résultats dans le domaine des motifs d'une part par la partie du diagramme \ref{diagramme motifs} que l'on étudie et d'autre part par la nature géométrique ou arithmétique du corps de base $k$  qui est concerné.
 Pour aider la lecture du texte qui suivra, voici une répartition des travaux présentés.
 
 \vspace{0.5cm}                                 
 
 \noindent\begin{tabular}[b]{|c|c|c|c|c|}

       \hline 
     &  Arithmétique   &   Géométrie \\
   \hline
   $\CHM(k)$    & \S\ref{section conserv}   \cite{Ascona} &  \S\ref{AHP} \cite{AEH,AHP}     \\

   \hline
  $ \Mot(k) $  & \S\ref{section andre}   \cite{AF}   &   \S\ref{section lef} \cite{ACLS}  \\
   
   \hline 
  $\NUM(k)$     & \S\ref{positivite}   \cite{Anc21}  &    \\

  \hline
  \end{tabular}

  \subsection*{Organisation du texte}
  Les premières sections du texte fournissent une introduction partielle à la théorie.
 Nous avons déjà mentionné que les motifs sont utiles à l'étude des cycles algébrique et donné le Théorème \ref{thm intro} comme exemple, mais nous n'avons pas dit pourquoi ils sont utiles : c'est le but de la Section \ref{pourquoi motifs}. La Section \ref{coho Weil} présente la théorie de Hodge et ses pendants arithmétiques ; on insiste sur les analogies mais surtout sur les différences. La Section \ref{conjectures standard} donne des conjectures sur les cycles algébriques et la Section \ref{exemple} des exemples de motifs. Ces questions et ces exemples sont repris dans les sections successives où l'on présente différents résultats organisés selon le diagramme ci-dessus. Dans la Section \ref{section conserv} on démontre la conjecture de conservativité pour les motifs provenant de variétés abéliennes définies sur un corps fini. La Section  \S\ref{positivite} concerne la conjecture standard de type Hodge et montre  le Théorème \ref{thm intro}(3). Dans la Section  \S\ref{section andre} on introduit une nouvelle classe de périodes $p$-adiques qui surgit dans l'étude des classes algébriques en caractéristique mixte. Dans la Section \S\ref{AHP} on montre  le Théorème \ref{thm intro}(4) qui nécessite l'utilisation de techniques motiviques modernes. La Section \S\ref{section lef} étudie les classes algébriques  de certaines variétés hyper-kähler  qui admettent une fibration lagrangienne.

  \section{A quoi servent les motifs ?}\label{pourquoi motifs}
Dans un article du même titre \cite{DelMot}, Deligne expliquait  que les motifs n'ont qu'une utilité essentiellement philosophique  permettant de transférer des idées d'une cohomologie à l'autre, grosso-modo en appliquant le diagramme \ref{diagramme motifs} à différentes cohomologies. 
Aujourd'hui  on comprend que les motifs sont aussi un vrai outil technique, comme l'avait envisagé Grothendieck. 
Deligne même revoit sa position dans son Bourbaki sur les multizétas et explique que les travaux de Brown sont \og  un des cas où la philosophie des motifs est non seulement
un guide précieux, mais permet des démonstrations\fg{} \cite{DelBrown}.

 J'aimerais expliquer ici l'utilité des motifs notamment dans la théorie des cycles algébriques  :  une liste d'énoncés où leur utilisation est cruciale a déjà été donnée dans le Théorème \ref{thm intro}.
 Plusieurs avertissements sont tout de même nécessaires. Premièrement, je ne prétends pas que cet outil soit l'unique possible, beaucoup de résultats intéressants sur les anneaux de Chow ont été obtenus par d'autres méthodes.
 Deuxièmement, à l'heure actuelle les applications les plus impressionnantes des motifs apparaissent plutôt dans la théorie des périodes \cite{Brown,AyoubKZ} - on peut espérer que les applications majeures de la théorie aux anneaux de Chow sont encore à venir.
 
 \
 
  Dans sa construction de la théorie des motifs purs (i.e. pour les variétés propres et lisses) Grothendieck avait imaginé un pont entre les cycles algébriques et, par exemple, la cohomologie   $\ell$-adique. 
  Son idée était que la compréhension des premiers aurait impliqué ainsi des résultats sur la deuxième, par exemple les conjectures standard auraient impliqué les conjectures de Weil.  
  En un sens cela semblait la direction raisonnable : les cycles étaient là depuis plus longtemps (on pourrait dire depuis le théorème de Bézout) et leur définition pouvait les faire paraître comme plus accessibles.
  
  On comprend aujourd'hui qu'ils sont plus mystérieux que ce que l'on aurait pu imaginer. Ceci est devenu flagrant probablement avec le théorème de Mumford \cite{Mumzero} qui montre que  le groupe des $0$-cycles ne peut pas, en général, être paramétré par une variété de type fini. En revanche beaucoup de progrès   ont été faits sur la cohomologie.
  Ce pont maintient donc toute son utilité mais il faut plutôt le parcourir dans l'autre direction : on essaiera d'exploiter des informations cohomologiques et de les transposer sur les cycles via les motifs.

 La question naturelle qui se pose est alors pourquoi l'application classe de cycle ne serait pas elle-même suffisante pour un tel pont entre anneaux de Chow et cohomologie ?
 Une réponse courte est que la cohomologie a des propriétés agréables (Künneth, Poincaré,$\ldots$) que les anneaux de Chow n'ont pas. Les motifs sont une façon de réorganiser les applications classe de cycle de sorte à ce que l'on ait encore ces propriétés agréables.
 
 \subsection{Motifs purs}\label{motifs purs} Pour expliquer l'idée derrière la construction, prenons la situation suivante (inspirée par le formalisme tannakien). Soit  $\phi:H \rightarrow G$ un  morphisme de groupes et étudions le foncteur induit
 \begin{align}\label{analogia1}  f=\phi^* :  \Rep_F(G) \longrightarrow \Rep_F(H)\end{align}
 sur les représentations $F$-linéaires pour $F$ un corps fixé.
 Supposons avoir à notre disposition pour cette étude uniquement la collection d'applications 
 \begin{align}\label{analogia2}  c_V : V^G   \longrightarrow V^H \subset f(V) \end{align}
 pour chaque $V \in  \Rep(G)$.
 La donnée de cette collection est   certainement moins agréable que la donnée de $f$.
 On peut tout de même retrouver $f$ en remarquant que son action sur les morphismes est donnée par
  \begin{align}\label{invequi} c_{W\otimes V^{\vee}} : \Hom_G(V,W)=  (W\otimes V^{\vee})^G   \longrightarrow   (W\otimes V^{\vee})^H  =  \Hom_H(V,W) . \end{align}
  
  Cette idée de passer de l'invariant à l'équivariant est l'étape essentielle dans la construction de Grothendieck des motifs. 
Dans ce cas pour chaque variété $X$, $V$ sera son motif $\mathfrak{h}(X)$ (objet abstrait de la catégorie en construction), $V^G$ sera l'anneau de Chow $\CH(X)$,   $f(V)$ sera la  cohomologie singulière (ou $\ell$-adique,$\ldots$) $\HW(X)$, $V^H$ seront les classes de Hodge $\Hdg(X)$ (ou les classes Galois invariantes, $\ldots$) et $c_V $ sera l'application classe de cycle $\cl_X $. En résumant :
\begin{align}\label{Eq:realization functors}
\begin{array}{rcl}
         V &\rightsquigarrow&\mathfrak{h}(X), \\
         V^G& \rightsquigarrow & \CH(X), \\
         f(V) &\rightsquigarrow & \HW(X), \\
         V^H & \rightsquigarrow & \Hdg(X) ,\\
         c_V & \rightsquigarrow &\cl_X.
	\end{array}
\end{align}
Pour que la construction dans \eqref{invequi} soit applicable il faut donner un sens au dual d'un motif et au produit tensoriel de deux motifs. 
C'est ici qu'il est nécessaire de considérer des variétés propres et lisses.
En effet la formule de Künneth et la dualité de Poincaré suggèrent $\mathfrak{h}(X)\otimes \mathfrak{h}(Y) = \mathfrak{h}(X \times Y)$ et $\mathfrak{h}(X)^{\vee}=\mathfrak{h}(X)$ (à un twist de Tate près). L'espace $\Hom(\mathfrak{h}(X) , \mathfrak{h}(Y) )$ sera alors contrôlé\footnote{Pour les variétés générales qui ne sont pas projectives et lisses l'argument ci-dessus suggère qu'il n'est pas raisonnable d'espérer un lien entre  $\Hom(\mathfrak{h}(X) , \mathfrak{h}(Y) )$  et  des anneaux de Chow. Effectivement dans les motifs de Voevodsky ces $\Hom$ n'ont pas d'interprétation en terme d'invariants classiques, voir aussi la Remarque \ref{garufa}(4).} par $\CH(X \times Y)$. 

Une fois que la construction de cette catégorie est faite on pourra imaginer - et essayer de démontrer - des analogies entre $\HW(X)$ et $\mathfrak{h}(X)$ qui ne seraient pas raisonnables avec $\CH(X)$ pour ensuite déduire des informations sur les anneaux de Chow en passant aux $\Hom$ dans la catégorie.  Par exemple on pourra imaginer que $\HW(X)$ et $\mathfrak{h}(X)$ ont \og la même dimension \fg, ce qui ne peut pas avoir lieu avec $\CH(X)$ - voir l'analogie \eqref{analogia1}. Cette idée est à la base de la notion de dimension dans les motifs due à Kimura et O'Sullivan qui est l'ingrédient essentiel dans la preuve  du  Théorème \ref{thm intro}(1) et joue également un rôle dans les parties (2) et (4) du même théorème.


\subsection{Motifs mixtes}\label{motifs mixtes} La théorie a beaucoup évoluée depuis ses fondations.
Comme la théorie de Hodge, qui  a évolué d'abord avec les structures de Hodge mixtes associées à des  variétés qui ne sont pas forcément projectives ou lisses, pour arriver jusqu'aux modules de Hodge mixtes qui visent à étudier des familles de variétés sur des bases générales, également la théorie des motifs a eu une accélération significative avec la catégorie triangulée des motifs mixtes de Voevodsky jusqu'aux motifs relatifs sur une base générale \cite{Ayoub_these_1,CD}.
Ces   catégories   sont liées par le formalisme des six foncteurs, tout comme    les modules de Hodge mixtes. 
De plus, certains $\Hom$ dans ces catégories permettent de retrouver les anneaux de Chow.
On dispose également de foncteurs de réalisation, par exemple vers les faisceaux constructibles, qui permettent  de retrouver les applications classes de cycles.

Tout comme dans le cas   des variétés projectives et lisses expliqué plus haut, ces  catégories ont l'avantage d'avoir des analogies avec les catégories cohomologiques.
Par exemple on dispose d'une filtration de poids sur les motifs tout comme en théorie de Hodge mixte \cite{Bon}. (Une telle structure  n'a pas de bons analogues dans les anneaux de Chow : par exemple un ouvert d'un espace affine a un anneau de Chow trivial alors que la filtration de poids en cohomologie peut être   non triviale.) Un autre exemple est le résultat suivant d'Ayoub \cite[Proposition 3.24]{Ayoubet} : une application $f$ entre motifs au-dessus d'une base $S$ est un isomorphisme si et seulement si la restriction de $f$ en tout point de $S$ l'est. Cet énoncé est bien entendu inspiré de son pendant pour les faisceaux constructibles ou étales.

Ces nouvelles catégories présentent un deuxième avantage :
on dispose maintenant de beaucoup plus de flexibilité, analogue à celle permise par les modules de Hodge mixtes.
Par exemple, si on veut  étudier l'anneau de Chow d'une variété projective et lisse $X$, on pourrait avoir envie de   stratifier $X$ et d'étudier chaque strate, ou de   fibrer $X$ au-dessus d'une base et d'étudier comment les fibres varient. 
Ce genre de construction mène très souvent à des variétés qui ne sont pas lisses et pour lesquelles les anneaux de Chow et leur fonctorialité ne sont pas définis : ces catégories de motifs permettent, entre autre, de contourner ce problème.

Ces techniques ont permis la construction de certains cycles \og concrets \fg{}, par exemple certains cycles prédits par la conjecture de Hodge joint au théorème de décomposition, qui étaient inaccessible par des méthodes directes. 
Notamment cela a été appliqué  aux variétés de Shimura \cite{WildShi}, aux fibrés en quadriques \cite{CDN} et aux variétés hyper-kähler qui admettent une fibration lagrangienne \cite{ACLS}.

\

\subsection{Complexes motiviques}\label{complexes motiviques} Une troisième raison pour laquelle les motifs sont utiles à l'étude des cycles algébriques vient de leur définition moderne (depuis Voevodsky).
Dans la théorie de Grothendieck les motifs sont des symboles  formels et leur lien avec les cycles algébriques a lieu par construction.
Dans les catégories modernes les motifs sont des complexes de faisceaux et il est possible d'en construire un certain nombre explicitement. 
Leur lien avec les cycles algébriques est loin d'être une tautologie et c'est en fait un des résultats plus profond de la théorie.
On peut alors espérer que certaines questions   délicates sur les cycles algébriques puissent devenir concrètes dans leur pendant faisceautique.
C'est ce qui se passe notamment dans la construction de la réalisation de Betti \cite{Ayoubbetti} ou dans la preuve du   Théorème \ref{thm intro}(4), voir Section \ref{AHP}.

  \section{Cohomologies de Weil}\label{coho Weil}
  
  Dans cette section on rappelle la définition classique de structure de Hodge (Définition \ref{def HS}) ainsi qu'une formulation équivalente qui se prête  à mieux décrire les  propriétés de positivité (Définition \ref{def HSR}). Le but principal de la section est de rappeler ces propriétés de positivité ainsi que des propriétés d'autodualité des structures de Hodge puis de  montrer que leurs analogues en cohomologie $\ell$-adique sont faux en général (Remarque \ref{rem non autodual} et Exemple \ref{non autodual}). La définition de polarisation est cruciale, on essaie de la justifier dans la Remarque \ref{rem esempio pol}.
  
  D'autres différences entre la théorie de Hodge et ses analogues arithmétiques sont éparpillées un peu partout dans le texte, voir par exemple la Conjecture \ref{conj nonso} ou la Remarque   \ref{rem faltings}.
  \begin{defin}\label{def HS} (Structure de Hodge, définition classique.)
Une structure de Hodge pure de poids $n\in\Z$ est la donnée  d'un $\Q$-espace vectoriel de dimension finie $V$ muni d'une décomposition 
   \begin{align}\label{decomposition HodgeC} V\otimes_{\Q} \C=\bigoplus_{\substack{p+q=n \\ p, q \in \Z  }} V^{p,q}  \end{align}
  vérifiant $V^{p,q}= \overline{V^{q,p}}$.
     \end{defin}
  \hspace{1cm}
    
        Si on considère les espaces $ V^{\{p,q\}} = V\otimes_{\Q} \R\cap (V^{p,q} + V^{q,p})$
        on trouve la définition équivalente suivante.
        
\begin{defin}\label{def HSR} (Structure de Hodge, définition équivalente.)
Une structure de Hodge pure de poids $n\in\Z$ est  un triplet formé  d'un $\Q$-espace vectoriel de dimension finie $V$, d'une décomposition 
  \begin{align}\label{decomposition Hodge} V\otimes_{\Q} \R=\bigoplus_{\substack{ p\leq q \\ p+q=n \\ p, q \in \Z  }} V^{\{p,q\}}\end{align}
  en sous-espace réels  et d'une structure complexe sur   les espaces  $V^{\{p,q\}}$  pour $p \neq q$.
  
  Les paires $(p,q)$ qui apparaissent dans la décomposition sont appelées les types de $V$.
    \end{defin}
    \begin{rem}\label{lien classique} (Lien entre les définitions.)
    Les deux présentations ci-dessus sont bien équivalentes. On remarque que, pour $p < q$,   l' espace $V^{\{p,q\}}\otimes \C$ est muni de deux structures complexes, par conséquent on a une décomposition 
    \[V^{\{p,q\}}\otimes \C = V^{p,q} \oplus V^{q,p},\] où $V^{p,q}$ est l'espace  où les structures coïncident  et $V^{q,p}$ est l'espace où elles sont conjuguées\footnote{On pourrait inverser le rôle  de ces deux espaces, dans ce cas il faudrait changer les signes dans la Définition \ref{def pol imp} pour avoir les mêmes conventions de signe classiques.}.
    
     La Définition \ref{def HSR}  est plus pratique pour exprimer des propriétés de positivité\footnote{En revanche la structure tensorielle s'exprime mieux avec la Définition \ref{def HS}, par la règle
    \[(V\otimes W)^{p,q}=\bigoplus_{\substack{a+c=p \\ b+d=q  }}V^{a,b}\otimes  W^{c,d}. \]
    Dans le cas réel on remarque que $V^{\{a,b\}}\otimes  W^{\{c,d\}}$ a deux structures complexes quand $a<b$ et $c<d$. Ceci induit une décomposition en deux espaces
    comme précédemment, celui où les structures coïncident contribue à $(V\otimes W)^{a+c,b+d}$ et l'autre à $(V\otimes W)^{m,M}$ où $m$ et $M$ sont le minimum et le maximum de la paire $\{  a+d, b+c\}$.}, voir par exemple Définition \ref{def pol pair}. 
    \end{rem}
    \begin{exemple}\label{exemple HS}
    \begin{enumerate}
       \item (Cohomologie singulière.) Si $X$ est une variété projective et lisse sur les nombres complexes sa cohomologie singulière est munie d'une structure de Hodge fonctorielle en $X$, plus précisément $H^n(X)$ est pure de poids $n$. 
    Tous les théorèmes que l'on peut imaginer (Poincaré, Künneth, Lefschetz,$\ldots$) respectent cette structure supplémentaire.
    
    \item\label{exemple HS A} (Variétés abéliennes.) Dans le cas particulier d'une variété abélienne complexe $A$ dont la variété analytique sous-jacente est $\C^g/\Lambda$ on a $H_1(A)\otimes \R = \Lambda \otimes \R = \C^g$, ce qui fait apparaître explicitement la structure complexe dans la Définition \ref{def HSR} de l'espace $H_1(A)\otimes \R $.
    
      \item\label{model tate}  (Twist de Tate.) En poids  $2n$ il existe une seule structure de Hodge de dimension $1$ à isomorphisme non unique près. On fixe une telle structure de    Hodge et on la note $\Q(-n)$. Pour une structure de Hodge $V$ on note $V(-n)=V\otimes \Q(-n)$ et on l'appelle twist de Tate; on choisit les  $\Q(-n)$ de sorte à avoir l'identification $\Q(a)\otimes \Q(b)=\Q(a+b).$
      
      Traditionnellement on choisit $\Q(n)=(2i\pi)^n \Q \subset \C$, puisque c'est la structure de Hodge qui apparaît naturellement dans la cohomologie singulière de degré maximale. Ce choix particulier ne joue pas de rôle dans les constructions qui suivront, sauf pour la notion de polarisation, voir Remarque \ref{rem HR}\eqref{choix Deligne}.
    \end{enumerate}
\end{exemple}
\begin{rem} (Autodualité.)
Une structure de Hodge n'est pas, en général, autoduale dans le sens où $V$ et $V^{\vee}$  ne sont pas en général isomorphes, même à un twist près (le seul twist pour lequel cette autodualité est raisonnable étant le poids).

Par exemple prenons une structure de Hodge de poids $0$ et  supposons que la décomposition \eqref{decomposition Hodge} ait deux facteurs non nuls
$V\otimes_{\Q} \R= V^{\{0,0\}} \oplus V^{\{-1,+1\}}.$
Supposons de plus que $V^{\{0,0\}} $ soit compatible avec la structure rationnelle et $ V^{\{-1,+1\}}$ ne le soit pas, c'est-à-dire que l'on ait l'égalité 
$\dim_\Q (V^{\{0,0\}}\cap V)=\dim_\R V^{\{0,0\}}$ mais  $\dim_\Q (V^{\{-1,+1\}}\cap V)<\dim_\R V^{\{-1,+1\}}$.

Si $V$ était autodual on aurait en particulier une forme bilinéaire sur $V$   pour laquelle  $V^{\{0,0\}}$ et  $ V^{\{-1,+1\}}$ seraient l'orthogonal l'un de l'autre. Or comme une telle forme bilinéaire serait définie sur $\Q$, l'espace $V^{\{-1,+1\}}$ serait forcé à être également compatible avec la structure rationnelle.
    \end{rem}
  

Il se trouve que toutes les structures de Hodge provenant de la géométrie algébrique sont autoduales. En fait même plus, elles sont munies d'une forme bilinéaire \og aussi définie   que possible \fg, c'est l'objet de la définition de polarisation, rappelée ci-dessous. 
  \begin{defin}\label{def pol pair} (Polarisation - poids pair.)
  Une polarisation sur une structure de Hodge $V$ pure de poids $2n$ est une forme bilinéaire symétrique  $b$
  sur le $\Q$-espace vectoriel $V$ telle que :
  \begin{enumerate}
  \item  Les espaces $V^{\{p,q\}}$ sont orthogonaux deux à deux par rapport à $b$,
  \item L'adjointe par rapport à $b$ de la structure complexe sur $V^{\{p,q\}}$ est sa conjuguée,
  \item  La restriction de $b$ aux facteurs  $V^{\{n-a,n+a\}}$ est définie $(-1)^a$-positive\footnote{C'est-à-dire définie positive si $a$ est pair et définie négative si $a$ est impair.}.
  \end{enumerate}
  \end{defin}
  \begin{rem}\label{rem HR}
  \begin{enumerate}

  \item\label{choix Deligne}
     Les deux premières propriétés de la définition sont équivalentes à l'existence d'un  morphisme de structures de Hodge
  \[b:V\otimes V \longrightarrow \Q(-2n).\]
    
  Inversement, étant donné  un  tel $b$ 
il est nécessaire de  choisir une identification $\Q(-2n)\otimes_\Q \R\cong \R$ pour pouvoir exprimer la dernière propriété (de positivité).
 L'Exemple \ref{exemple HS}\eqref{model tate} fixe un tel choix.
  \item
  Dans le cas $V=H^{2n}(X)$ on déduit  qu'une polarisation est définie positive sur les classes algébriques, car elle l'est sur  $V^{\{n,n\}}$.
C'est   utile de connaître les signes sur les autres facteurs même pour la compréhension des classes algébriques, voir l'Exemple \ref{K3 fermat} et la Section \ref{positivite}.
  
 \item Une polarisation est automatiquement non dégénérée donc les structures de Hodge qui la possèdent sont autoduales. D'autre part il y a des structures de Hodge autoduales, même simples, qui n'admettent pas de polarisation, voir l'Exemple \ref{es dragos}. Autrement dit la polarisabilité est une notion plus forte que l'autodualité.

\end{enumerate}
    \end{rem}
    \begin{exemple}(Autodualité vs polarisation.)\label{es dragos}
    Construisons une structure de Hodge simple et autoduale qui n'admet pas de polarisation. Soit $A$ une variété abélienne complexe de dimension $4$ et très générale. Alors la partie primitive $V=\HW^{4,\prim}(A)$ est  une structure de Hodge simple,  de type $(0,4),(1,3)$ et $(2,2)$ et polarisable (par le Corollaire \ref{gianluca}).  Définissons une nouvelle structure de Hodge $W$  où le $\Q$-espace vectoriel est le même que $V$, ainsi que la décomposition, mais on inverse le rôle de la partie $(0,4)$ avec la partie $(1,3)$, c'est-à-dire
    \[W^{\{0,4\}}  =  V^{\{1,3\}},  \hspace{1cm}   W^{\{1,3\}}  =  V^{\{0,4\}}   \hspace{1cm} \textrm{et}   \hspace{1cm}  W^{\{2,2\}}  =  V^{\{2,2\}}. \]
    On prétend que $W$ satisfait aux propriétés requises. 
     Premièrement remarquons qu'une polarisation sur $V$ induit un accouplement $b$ sur $W$ qui rend $W$ autoduale. D'autre part ce $b$ ne peut pas être une polarisation car les signes de la Définition \ref{def pol pair} ne sont pas respectés. Il reste à montrer que $W$ n'admet pas de polarisation et pour cela il suffit de voir que $b$ est l'unique accouplement sur $W$ à scalaire près et donc  que $\End_{\HS}(W)=\Q \cdot \Id.$ Or, par construction, on a  $\End_{\HS}(W)=  \End_{\HS}(V)$. Pour comprendre les endomorphismes de $V$, on prend le point de vue tannakien. Le groupe tannakien associé à $\HW^1(A) $ est  $\GSp_8$ et la représentation sous-jacente à  $\HW^{4,\prim}(A) \subset \HW^1(A)^{\otimes 4} $ est géométriquement irréductible, donc $\End_{\HS}(V)=\Q \cdot \Id.$

        \end{exemple}
    Nous rappelons ci-dessous la définition de polarisation dans le cas de poids impair, elle est un peu moins agréable.
Même si on s'intéresse seulement aux classes algébriques les groupes de cohomologie de degré impair peuvent être utiles (par exemple la cohomologie d'une variété abélienne $A$ est contrôlée par son $H^1$) mais il est souvent suffisant de se rappeler uniquement que la définition de polarisation est stable par produit tensoriel (et donc une polarisation sur $H^1(A)$ en induira une sur $H^{n}(A)$ via $H^n(A) = \Lambda^n H^1(A)$).   
    \begin{defin}\label{def pol imp} (Polarisation - poids impair.)
  Une polarisation sur une structure de Hodge $V$ pure de poids $2n+1$ est une forme bilinéaire alternée  $b$
  sur le $\Q$-espace vectoriel $V$ telle que :
  \begin{enumerate}
  \item  Les espaces $V^{\{p,q\}}$ soient orthogonaux entre eux par rapport à $b$,
  \item L'adjointe par rapport à $b$ de la structure complexe sur $V^{\{p,q\}}$ est sa conjuguée,
  \item  La forme bilinéaire symétrique $b(\cdot , i\cdot)$   est définie $(-1)^a$-positive  sur les facteurs $V^{\{n-a,n+1+a  \}}$.
\end{enumerate}
 \end{defin}
     \begin{rem}
  \begin{enumerate}
  \item
La Remarque \ref{rem HR} s'applique également dans ce cas.
\item
On pourrait vouloir travailler avec $b(i\cdot , \cdot)$ à la place de $b(\cdot , i\cdot)$, dans ce cas les signes s'inverserait.
On trouve plus agréable  d'avoir de la positivité sur la \og partie centrale \fg{} $V^{\{n,n+1  \}}$, c'est analogue à ce qui se passe dans la Définition \ref{def pol pair} avec la partie centrale $V^{\{n,n  \}}$.
\end{enumerate}
    \end{rem}
  \begin{exemple}\label{esempio pol} (Polarisations issues de la géométrie.)
 Soit $X$ une variété algébrique complexe de dimension $d_X$. D'une part la dualité de Poincaré fournit une identification $H^n(X)=H^{2d_X-n}(X)^\vee(-d_X)$. D'autre part, le théorème de Lefschetz difficile donne un isomorphisme $H^n(X)\cong H^{2d_X-n}(X) (d_X-n),$ au moyen  du choix d'une section hyperplane $L$. En combinant les deux on obtient une forme bilinéaire non-dégénérée
   \[b_L : H^{n}(X)\otimes H^{n}(X) \longrightarrow \Q(-2n).\]
   Celle-ci est bien un morphisme de structure de Hodge, mais elle n'a pas les signes demandés par une polarisation. Pour les obtenir, il faut modifier les signes de certains facteurs de la décomposition de Lefschetz. 
   
   Par exemple écrivons la décomposition de Lefschetz en degré six
    \[H^6=H^{6,\prim}\oplus H^{4,\prim}(-1)\oplus  H^{2,\prim}(-2)\oplus  H^{0,\prim}(-3), \] elle est orthogonale par rapport à l'accouplement $b_L$ ci-dessus. La décomposition de Hodge de chacune de ces quatre structures de Hodge est encore orthogonale ; la signature de $b_L$ est la suivante    (où les cases vides sont pour les sous-espaces qui sont toujours réduits à zéro). 
\vspace{0.5cm}                                 
 
  \hspace{2cm}\begin{tabular}[b]{|c|c|c|c|c|}

       \hline 
     signes $b_L$  &  $H^{6,\prim} $ &  $H^{4,\prim}(-1) $ &   $H^{2,\prim}(-2)$ & $H^{0,\prim} (-3)$\\
   \hline
   $(3,3)$  &-  & +   &  - & +  \\

   \hline
  $ (2,4) $ &+ & -  &   +   &  \\
   
   \hline 
  $(1,5)$  &-   &  + &   &   \\
   \hline
  $ (0,6) $ &+ &   &      &   \\
  \hline
  \end{tabular}

  Pour obtenir une polarisation il faudra donc changer le signe sur les facteurs $H^{6,\prim}$ et  $H^{2,\prim}(-2).$
   \end{exemple}

    \begin{rem}\label{rem esempio pol}(Pourquoi la définition de polarisation ?)
   Considérons le changement de signe expliqué dans l'exemple ci-dessus.
     Il serait tentant, à première vue,   de changer encore de signe, cette fois-ci par rapport à la décomposition en $V^{\{p,q\}}$, de sorte à avoir une forme quadratique définie positive dans  la Définition \ref{def pol pair}. 
     Le problème est que cette deuxième décomposition n'est pas en général définie sur $\Q$ et donc ce changement de signes donnerait une forme bilinéaire qui n'est pas définie sur $\Q$. 
     
     De manière générale, pour les structures de Hodge provenant  de la géométrie, on ne peut pas espérer avoir une forme quadratique qui soit à la fois définie sur $\Q$, compatible avec la décomposition de Hodge et définie positive, voir l'Exemple \ref{exemple non positif}.
     Il faut alors imaginer la   Définition \ref{def pol pair} comme la meilleure approximation d'un produit scalaire qui puisse exister pour les structures de Hodge issues des variétés algébriques. D'ailleurs la proposition ci-dessous montre que la polarisation a toutes les conséquences que l'on aimerait déduire d'un produit scalaire. 
     \end{rem}
     
     \begin{prop}\label{HS semisimple}
Soient $V$ une structure de Hodge, $b$ une polarisation sur $V$ et $W\subset V$ une sous-structure de Hodge. Alors la restriction de $b$ à $W$   est une polarisation.
De plus l'orthogonal $W^\perp \subset V$ de $W$ par rapport à $b$ est une sous-structure de Hodge et on a l'égalité de structures de Hodge $V=W \oplus W^\perp$.
\end{prop}
\begin{cor}\label{gianluca}
Soit $X$ une variété complexe projective et lisse. Alors tout sous-structure de Hodge de la cohomologie singulière de $X$ est polarisable.
\end{cor}
     
      \begin{exemple}\label{exemple non positif} (Polarisations vs produits scalaires.)
      Construisons une structure de Hodge $V$ de poids zéro et d'origine géométrique telle que  $\Hom_{\HS}(V\otimes V, \Q(0))$ ne contient pas une forme quadratique définie positive.

Soit $E$ une courbe elliptique non $CM$ et considérons  $H^1(E)$. 
On a  une décomposition de structures de Hodge $H^1(E)\otimes H^1(E)^\vee =  \Q(0) \oplus V$, où $V$ a dimension $3$ et types $(0,0)$ et $(-1,1)$.

On prétend que $\Hom_{\HS}(V,V^\vee)= \Hom_{\HS}(V\otimes V, \Q(0))$ est un $\Q$-espace vectoriel de dimension $1$.  
Pour le montrer on prend le point de vue tannakien. On peut voir que le groupe tannakien associé   à $H^1(E)$ est $\GL_2$ et que $H^1(E)$ est la représentation standard. 
Le groupe $\GL_2$ agit sur $V$ et on a l'identification  $\Hom_{\HS}(V,V^\vee) = \Hom_{\GL_2}(V,V^\vee)$.
D'autre part, la théorie classique des représentations nous dit que $V$ est géométriquement irréductible et donc que $\Hom_{\GL_2}(V,V^\vee)$ est un $\Q$-espace vectoriel de dimension (au plus) $1$. 

D'autre part  l'espace $\Hom_{\HS}(V\otimes V, \Q(0))$ contient une polarisation, par le Corollaire \ref{gianluca}, et donc cet espace est formé uniquement de multiples d'une polarisation. En particulier $\Hom_{\HS}(V\otimes V, \Q(0))$ ne peut pas contenir une forme quadratique définie positive.

On peut aussi en déduire que $H^2(E\times E)$ n'admet pas une forme quadratique définie positive qui respecte la structure de Hodge car $V(-1)$ en est un facteur direct.
\end{exemple}

Le reste de la section insiste sur les différences entre théorie de Hodge et cohomologie $\ell$-adique, au regard notamment de la notion de polarisation et des propriétés d'autodualité.

    \begin{rem}\label{rem non autodual} (Polarisations en cohomologie $\ell$-adique ?)
La proposition ci-dessus dit en particulier que les structures de Hodge  issues de la géométrie algébrique forment une catégorie semi-simple et que chaque objet est autodual à un twist près. 
C'est une grande différence avec la cohomologie $\ell$-adique : la semi-simplicité est seulement conjecturale et l'autodualité est fausse en général, voir l'Exemple \ref{non autodual}.

La notion même de polarisation n'a pas d'analogue : on ne peut même pas formuler des propriétés de positivité analogues à celles de la Définition \ref{def pol pair} pour la simple raison que la notion de positif n'a pas de sens dans   $\Q_\ell$.

On peut construire des accouplements sur la cohomologie $\ell$-adique de la même façon qu'en théorie de Hodge, comme dans l'Exemple \ref{esempio pol}. On ne connait pas de formulation, même conjecturale, qui décrirait cette $\Q_\ell$-forme quadratique.  Par ailleurs les invariants d'une telle forme quadratique, comme son symbole  de Hilbert, ne contrôlent pas ceux des sous-formes quadratiques (contrairement à ce qui se passe avec la signature). Ceci suggère que  même si on trouvait une propriété analogue à la Définition \ref{def pol pair}  pour les groupes de cohomologies $H^n_\ell(X)$ elle pourrait ne pas être     valable pour les facteurs directs de $H^n_\ell(X)$.
\end{rem}

     \begin{exemple}\label{non autodual} (Non autodualité en cohomologie $\ell$-adique.)
Soient $k$ un corps  de type fini et $E$ une courbe elliptique définie sur   $k$ telle que $\End (E)\otimes \Q$ soit un corps quadratique imaginaire. Prenons un nombre premier $\ell$ différent de la caractéristique de $k$ et  tel que $\End (E)\otimes \Q_\ell =\Q_\ell  \oplus \Q_\ell.$ Alors l'action de $\End (E)\otimes \Q_\ell $ donne une décomposition de représentations galoisiennes
$H^1_\ell(E) =  V \oplus W.$
Le cup-produit induit une autodualité sur le  $H^1_\ell(E)$, elle réalise $W$ comme dual de $V$    par la positivité de Rosati.

D'autre part on prétend  que $V$ et $W$ ne sont pas isomorphes comme représentation de Galois, ce qui  impliquera en particulier que $V$ n'est pas autodual.
En effet s'ils étaient isomorphes on aurait $\End_{\Gal}H^1_\ell(E) = M_{2 \times 2}(\Q_\ell)$, or on a $\End_{\Gal}H^1_\ell(E)\cong\End(E)\otimes \Q_\ell$ comme prédit  par la conjecture de Tate, montrée dans ce context par Tate, Faltings et Zarhin.
\end{exemple}
\begin{rem}
L'exemple ci-dessus dépend du nombre premier $\ell$ choisi. On s'attend à ce que l'on ne puisse pas trouver une représentation d'origine géométrique,  \og indépendante de $\ell$ \fg{} et non autoduale. Cette idée est rendue précise par les motifs, voir
la Conjecture  \ref{conj autodual} et la remarque qui la suit.
\end{rem}

  \section{Conjectures standard et moins standard}\label{conjectures standard}
  Soient $k$ un corps de base et $\HW^*$ une cohomologie de Weil. Considérons le diagramme \eqref{diagramme motifs} qui représente des catégories de motifs : 
   \[
\xymatrix{
    \CHM(k)  \ar@{->>}[d]^{\pi} \ar[rd]^{R} &   \\
      \Mot(k)  \ar@{^{(}->}[r]^I \ar@{->>}[d]^{\pi'} & \GrVect. \\
    \NUM(k) & 
  }
    \]
    on les considère toujours à coefficients rationnels. Pour chaque variété projective et lisse $X$, de dimension $d_X$, on considère son motif $\mathfrak{h}(X)$, on utilisera ce même symbole dans les différentes catégories, on précisera la catégorie concernée si cela est important.
  
  Le but de cette section est de donner une collection de conjectures  qui permettent d'avoir une intuition sur les motifs. Ces  conjectures sont essentiellement classiques, hormis la Conjecture \ref{conj autodual} d'autodualité et la Conjecture \ref{Conj pos} de positivité qui sont nouvelles à notre connaissance. 
  
  La section est organisée en trois sous-sections. La première présente les conjectures principales de la théorie  (qui essentiellement entrainent les autres conjectures). La deuxième sous-section présente les conjectures dite standard de Grothendieck. La dernière porte sur la notion de dimension finie de Kimura.

  Avec les techniques actuelles ces conjectures sont hors de portée en toute généralité, mais il est possible en démontrer des cas particuliers, comme on le verra dans les sections successives.
  
  \subsection*{Conjectures principales}
  \begin{conj}\label{conj CK}(Chow--Künneth.) Il existe une décomposition  (non unique)
  \[\mathfrak{h}(X)=\bigoplus_{n=0}^{2d_X} \mathfrak{h}^n(X) \]
dans la catégorie $\CHM(k)$ telle que $R(\mathfrak{h}^n(X)) =\HW^n(X)$.
\end{conj}
\begin{rem}\label{rem kunneth} 
\begin{enumerate} 
\item  (Non unicité.) Le facteur $\mathfrak{h}^0(X)$ existe toujours. Pour le construire il suffit de considérer une application constante de $X$ vers $X$ : elle sera bien   l'identité sur le $\HW^0(X)$ et nulle sur les autres groupes. Le facteur ainsi défini dépend de l'image de cette application constante, ou plus précisément de sa classe modulo équivalence rationnelle. En particulier une telle décomposition ne peut pas être unique en général.
\item  (Autodualité.) On peut conjecturer l'existence d'une   décomposition de Chow--Künneth qui ait la propriété supplémentaire d'être  autoduale, c'est-à-dire que si on considère la dualité de Poincaré $\mathfrak{h}(X)^\vee= \mathfrak{h}(X)(d_X)$ alors le facteur $\mathfrak{h}^n(X)^\vee$ correspond à $\mathfrak{h}^{2d_X-n}(X)(d_X)$.
D'un point de vue de cycles algébriques cela veut dire que les projecteurs $p_n$ sont donnés par une collection de cycles dans $X\times X$ telle que $\sigma^*p_n=p_{2d-n}$, où $\sigma : X\times X  \rightarrow X\times X$ est l'inversion des deux facteurs.
Ce n'est pas automatique de construire une décomposition autoduale à partir d'une décomposition de Chow--Künneth : même si l'on pose   $p_{2d-n} = \sigma^*p_n$ on pourrait avoir que les projecteurs  $p_n$ et  $  \sigma^*p_n$ ne sont pas orthogonaux.
\end{enumerate}
\end{rem}
\begin{conj}\label{conj conserv} (Conservativité.) Tous les foncteurs du  diagramme \eqref{diagramme motifs} sont conservatifs, c'est--à-dire un morphisme entre motifs est en fait un isomorphisme si son image via  un de ces foncteurs du diagramme l'est.
 \end{conj}

\begin{defin}\label{def poids motif}
On dit qu'un motif est de poids $n$ s'il est facteur direct d'un motif de la forme  $\mathfrak{h}^n(X)$.
\end{defin}
  
\begin{conj}\label{conj autodual} (Autodualité.)
Si $M$ est un motif de poids $n$ (toujours à coefficients rationnels) alors il existe un isomorphisme (non unique) 
\[M^\vee \cong M(n).\]
\end{conj}
\begin{rem} \begin{enumerate}
\item (Autodualité des motifs vs autodualité en cohomologie.) Les groupes de cohomologies $\HW^n(X)$ jouissent de ce type d'autodualité via la dualité de Poincaré et l'isomorphisme de Lefschetz difficile (qui dépend du choix d'une section hyperplane), mais en général les facteurs directs de $\HW^n(X)$ n'ont pas cette propriété d'autodualité, voir l'Exemple \ref{non autodual}.
Ceci n'est pas en contradiction avec la conjecture ci-dessus : les exemples construits ne sont pas la réalisation d'un motif à coefficients rationnels.

\item (Lien avec les conjectures classiques.) Cette conjecture est nouvelle à notre connaissance. Nous trouvons sa formulation naturelle et elle nous a guidé dans l'étude de la conjecture de conservativité, voir la Section \ref{section conserv}.  
Elle suit par ailleurs de conjectures classiques  de positivité   (notamment la Conjecture \ref{CSTH} que l'on verra plus loin), de la même manière que l'autodualité pour les structures de Hodge suit des propriétés de positivité des polarisations, voir la Proposition \ref{HS semisimple} et la remarque qui la suit.
\end{enumerate}
\end{rem}

 \begin{conj}\label{Conj pos} (Positivité.) Supposons que le corps de base $k$ soit de caractéristique $p$. Soient $M$ un motif homologique sur $k$ et 
 \[q: \Sym^2 M \rightarrow \mathbbm{1}\]
 un morphisme dans $\Mot(k)$. 
 Supposons que $q$   soit la réduction modulo $p$ d'une application
 $\tilde{q}: \Sym^2 \tilde{M} \rightarrow \mathbbm{1}$
 définie en caractéristique zéro.
 Définissons  $q_Z$ comme la restriction de $q$ à toutes les classes algébriques $Z(M)=\Hom(\mathbbm{1}, M)$ de $M$ et  $q_B$ comme  la réalisation singulière de $\tilde{q}$.
 
Supposons  que $q_B$  soit une polarisation. Alors $q_Z$ est définie positive.
 \end{conj}

 \begin{rem}\begin{enumerate}
   \item (Sur la relevabilité.) L'hypothèse de relevabilité  à la caractéristique zéro est vérifiée dans des cas intéressants, par exemple les variétés abéliennes. Dans ce cas, les classes algébriques qui se relèvent à la caractéristique zéro vérifient automatiquement la conjecture, voir la Remarque   \ref{rem HR}.
Soulignons tout de même qu'il y a en général des classes algébriques qui ne sont pas relevables, même dans le cas des variétés abéliennes.
 \item (Lien avec les conjecture classique.)
 Cette conjecture est nouvelle et nous ne savons pas si elle peut se déduire de conjectures classiques. Nous trouvons sa formulation naturelle et elle nous a guidé dans l'étude de la conjecture standard de type Hodge, voir la Section \ref{positivite}.  
 \end{enumerate}
 \end{rem}
 
 \subsection*{Conjectures standard}
  \begin{conj}\label{kunneth}(Künneth.) Dans  la catégorie $\Mot(k)$ des motifs homologiques, il existe une décomposition  
  \[\mathfrak{h}(X)=\bigoplus_{n=0}^{2d_X} \mathfrak{h}^n(X) \]
 telle que la réalisation de $\mathfrak{h}^n(X)$ soit $\HW^n(X)$.
\end{conj}
 \begin{rem}\label{poincare motif} (Lien avec Chow--Künneth.)
   Cette conjecture est bien sûr une conséquence de la Conjecture \ref{conj CK}. 
   
 Remarquons qu'une  décomposition de Künneth est  automatiquement unique et autoduale par définition d'équivalence homologique : il s'agit de la graduation de $\GrVect$ et de  la dualité de Poincaré en cohomologie. C'est une différence avec la décomposition de Chow--Künneth, voir la Remarque \ref{rem kunneth}.
 \end{rem}

\begin{conj}\label{CSTL}(Lefschetz.) Soient $d_X$ la dimension de $X$ et $L$ une section hyperplane de $X$. Alors  pour tout $n\leq d_X$ il existe une correspondance $\gamma_n$ dans $X\times X$ dont la réalisation en degré $d_X + n$ induit un isomorphisme
\[R(\gamma_n) :  \HW^{d_X + n}(X)   \isocan \HW^{d_X - n}(X) (-n)\]
qui est l'inverse du cup produit  par $L^n$.
\end{conj}

\begin{rem}  (Lien avec l'autodualité.) La Conjecture standard de type Leschetz \ref{CSTL} implique la Conjecture standard de type Künneth \ref{kunneth}, c'est un argument classique de Kleiman \cite{GKL}.

Inversement, si $X$ vérifie Künneth alors la conjecture de type Lefschetz est équivalente à l'autodualité $\mathfrak{h}^n(X)^{\vee}\cong \mathfrak{h}^n(X)(n)$ du motif homologique $\mathfrak{h}^n(X)$. Cette équivalence se déduit de la  Proposition \ref{prop conserv}. Elle montre en particulier que la conjecture de type Lefschetz ne dépend pas de la section hyperplane $L$ choisie et elle suit de la Conjecture d'autodualité \ref{conj autodual}.
 \end{rem}
 \begin{conj}\label{conj hom=num} ($\hom = \num$.)  
 Le foncteur $\pi '$ est  une équivalence.
\end{conj}
La quatrième et dernière des conjectures standard est celle de type Hodge. Pour la formuler il est nécessaire d'introduire la proposition suivante.
\begin{prop}\label{pol mot} (cf. \cite[\S 3]{Anc21}) Supposons que $X$ vérifie la conjecture standard de type Lefschetz (Conjecture \ref{CSTL}). Choisissons une section hyperplane $L$. 

Alors le motif homologique $\mathfrak{h}^n(X)$ admet une décomposition en parties primitives et il est possible de construire un accouplement
\[q_{X,n,L}: \mathfrak{h}^n(X) \otimes \mathfrak{h}^n(X) \rightarrow  \Q(-n)\]
de façon analogue à la construction d'une polarisation sur la cohomologie singulière d'une variété algébrique complexe\footnote{En particulier, si $k=\C$ et $\HW^*$ est la cohomologie singulière, $R(q_{X,n,L})$ est la polarisation classique induite par $L$.}, voir l'Exemple  \ref{esempio pol}.

De plus, si on restreint l'accouplement 
\[q_{X,2n,L}(2n): \mathfrak{h}^{2n}(X)(n) \otimes \mathfrak{h}^{2n}(X)(n) \rightarrow  \Q\]
aux classes algébriques $\mathcal{Z}^n(X)_{/ \hom}= \Hom_{\Mot(k)}(\mathbbm{1}, \mathfrak{h}^{2n}(X)(n))$ on obtient une $\Q$-forme quadratique
\[q_{Z, \hom}: \mathcal{Z}^n(X)_{/ \hom} \otimes \mathcal{Z}^n(X)_{/ \hom} \rightarrow \Q \]
dont le noyau est formé exactement par les cycles numériquement triviaux\footnote{Autrement dit, la conjecture $\hom = \num$ pour $X$ est équivalente à dire que $q_{Z, \hom}$ est non dégénérée.}. En particulier $q_{Z, \hom}$ induit une  forme quadratique   sur les cycles modulo équivalence numérique
\[q_Z : \mathcal{Z}^n(X)_{/ \num} \otimes \mathcal{Z}^n(X)_{/ \num} \rightarrow \Q \]
qui est non dégénérée.
\end{prop}
\begin{conj}\label{CSTH} (Conjecture standard de type Hodge.) La forme quadratique 
\[q_Z : \mathcal{Z}^n(X)_{/ \num} \otimes \mathcal{Z}^n(X)_{/ \num} \rightarrow \Q \]
introduite ci-dessus est définie positive.
 \end{conj}

 \begin{rem}  (Lien avec la conjecture de positivité.)
 La conjecture standard de type Hodge (Conjecture \ref{CSTH}) ne demande pas la relevabilité de $X$ à la caractéristique zéro et en ce sens elle est plus générale que la Conjecture  de positivité \ref{Conj pos}.
 D'autre part pour les variétés relevables  la Conjecture \ref{Conj pos} est plus générale que la conjecture standard de type Hodge, car elle s'applique à des polarisations abstraites qui ne seraient pas forcément celles provenant d'une section hyperplane.
   \end{rem}
   
    \begin{rem}\label{rem HR2} (Positivité en caractéristique zéro.)
En caractéristique zéro, la forme quadratique introduite ci-dessus $q_Z : \mathcal{Z}^n(X)_{/ \hom} \otimes \mathcal{Z}^n(X)_{/ \hom} \rightarrow \Q$ est définie positive : c'est une conséquence des propriétés de positivité d'une polarisation (dites relations de Hodge--Riemann), voir Remarque \ref{rem HR}. En particulier, en caractéristique zéro, la conjecture standard de type Lefschetz implique les autres conjectures standard.
\end{rem}

Une autre différence entre la caractéristique zéro et la caractéristique positive se trouve dans la conjecture suivante.
\begin{conj}\label{conj nonso}
  Considérons l'application classe de cycle $\ell$-adique $\cl_X : \CH(X) \rightarrow \HW_\ell(X).$
  Alors le $\Q$-espace vectoriel $\Im \cl_X $ est de dimension finie. Plus précisément l'application canonique $\Im \cl_X \otimes_\Q \Q_\ell \rightarrow  \HW_\ell(X)$ est injective.
\end{conj}
\begin{rem}

Cette conjecture est une conséquence de la conjecture $\hom=\num$ car l'équivalence numérique commute à l'extension des scalaires. En caractéristique zéro elle est connue : on utilise les théorèmes de comparaison pour se reporter à la cohomologie singulière.
 \end{rem}

\subsection*{Motifs de dimension finie}
\begin{conj}\label{conj Kimura} (Dimension finie.) Tout motif de Chow $M$  admet une décomposition (non unique)
\[M= M_+ \oplus M_-\]
vérifiant $\Lambda^N M_+=0$ et $\Sym^N M_-=0$ pour un naturel $N$ assez grand.
 \end{conj}
\begin{rem}\label{rem kimura}\begin{enumerate}
\item\label{koszul} (Lien avec les autres conjectures.)
Supposons que $M=\mathfrak{h}(X)$ admet une décomposition de Chow--Künneth (Conjecture \ref{conj CK}). Alors 
\[M_+=\bigoplus\mathfrak{h}^{2n}(X) \hspace{0.5cm} \textrm{et}   \hspace{0.5cm}   M_-=\bigoplus\mathfrak{h}^{2n+1}(X) \]
devraient vérifier la conjecture ci-dessus. 
En effet la conservativité (Conjecture \ref{conj conserv}) prédit qu'il suffit de vérifier ces relations en cohomologie, or la réalisation de $M_+$ est un espace vectoriel gradué de dimension finie concentré en degrés pairs, par conséquence toute puissance $N$-ème extérieure l'annule dès que $N$ est plus grand que la dimension totale.

Le raisonnement est analogue pour $M_-$. On remarque que la réalisation est concentrée en degré impair et que, par la règle des signes de Koszul, une puissance symétrique sur un motif ou une variété devient une puissance extérieure sur les groupes de cohomologie.
\item (Applications.) Au delà d'être une conjecture naturelle, la notion de dimension finie s'est avéré être  utile : elle est stable par plusieurs opérations, dont le produit tensoriel et le passage à un facteur direct, elle est vérifiée au moins par les courbes, et elle permet de déduire la conservativité pour tous les foncteurs  tensoriels quotient (notamment $\pi$ et $\pi'$ de \eqref{diagramme motifs}). Cela a permis à Kimura de déduire la conjecture de Bloch pour les surfaces dominées par un produit de courbes \cite{Kim}, voir le Théorème \ref{thm intro}(1). En appliquant ces propriétés de conservativité au Frobenius, Kahn a déduit que l'application classe de cycle est injective pour les produits de courbes elliptiques sur un corps fini \cite{Kahn},  voir le Théorème \ref{thm intro}(2). 
\end{enumerate}

\end{rem}
\section{Exemples}\label{exemple}
Dans cette section on discute des exemples de motifs. Ces exemples sont organisés dans trois sous-sections. Dans la première on présente des cas classiques où   les conjectures de la section précédente sont vérifiées. La deuxième sous-section  montre des subtilités (assez amusantes !) entre les différentes réalisations. La dernière partie étudie les motifs de variétés abéliennes CM. 
Tous les exemples de cette section sont repris à plusieurs endroits dans le texte.

\

On continue à travailler avec le diagramme \eqref{diagramme motifs} :
\[
\xymatrix{
    \CHM(k)  \ar@{->>}[d]^{\pi} \ar[rd]^{R} &   \\
      \Mot(k)  \ar@{^{(}->}[r]^I \ar@{->>}[d]^{\pi'} & \GrVect. \\
    \NUM(k) & 
  }
    \]
    et  $\mathfrak{h}(X)$ indiquera le motif d'une variété $X$ dans les différentes catégories, on précisera la catégorie concernée si cela est important.
    
\subsection*{Quelques évidences des conjectures de la Section \ref{conjectures standard}}

  \begin{prop}
  Soient $\mathfrak{h}(X)\in \Mot(k)$ un motif homologique, $f$ un endomorphisme de $\mathfrak{h}(X)$ et $p_n(f)$ le polynôme caractéristique de $f$ agissant sur $H^n(X)$. Supposons que $p_n(f)$ et $p_m(f)$ soient premiers entre eux pour tous les $n \neq m$. Alors $\mathfrak{h}(X)$ admet la décomposition de Künneth 
  \[\mathfrak{h}(X)=\bigoplus_{n=0}  \mathfrak{h}^n(X), \]
voir la Conjecture  \ref{kunneth}.
  De plus les projecteurs de cette décomposition appartiennent à l'algèbre $\Q[f]$.
  \end{prop}
   
  \begin{rem}\label{rem kunn abel}\begin{enumerate}
  \item La preuve est élémentaire : on applique l'identité de Bezout entre $p_n(f)$ et $\prod_{n\neq m} p_m(f),$ voir \cite{KM}.
  \item On peut appliquer cette proposition à toute variété projective et lisse définie sur un corps fini et à $f$ leur Frobenius : les polynômes   $p_n(f)$  vont bien entre premiers être eux par les conjectures de Weil.
  
 La proposition s'applique aussi aux variétés abéliennes sur un corps quelconque. Dans ce cas $f$ est la multiplication par un entier $N$.
 \item\label{kunn abel} Ces résultats ne s'étendent pas automatiquement à une décomposition de Chow--Künneth (Conjecture \ref{conj CK}), notamment pour les variétés sur un corps fini c'est une question ouverte. On dispose d'une décomposition de Chow--Kunneth pour les variétés abéliennes \cite{DeMu}, mais pour cela il faut utiliser la transformée de Fourier, voir Section \ref{AHP}. 
  \end{enumerate}
  \end{rem}
   \begin{prop}\label{cayley}
  Soient $M$ un motif homologique et $f$ un endomorphisme de $M$. Alors $f$ est inversible si et seulement si son action en cohomologie l'est.
  \end{prop}
    \begin{rem}\begin{enumerate}
  \item La preuve est élémentaire : on applique Cayley--Hamilton.
  \item Cette question donne une réponse affirmative à la Conjecture \ref{conj conserv} de conservativité pour le foncteur
  \[I : \Mot(k) \longrightarrow \GrVect\]
pour les endomorphismes.  On peut partiellement étendre le résultat aux morphismes quelconques de la façon suivante.
  \end{enumerate}
  \end{rem}
    \begin{prop}\label{prop conserv}
  Soient $M$ et $N$ deux motifs homologiques et considérons deux applications $f : M \rightarrow N$ et $g: N \rightarrow M$.
 Supposons que $I(f)$ et $I(g)$ soient des isomorphismes en cohomologie, alors $f$ et $g$ sont des isomorphismes.
    \end{prop}
      \begin{rem}\begin{enumerate}
  \item C'est une conséquence de la Proposition \ref{cayley} appliquée aux endomorphismes $fg$ et $gf$.
  \item On remarquera que ce n'est pas une solution complète de la conservativité pour $I$ : étant donné un morphisme $f$ il faut être capable d'en construire un dans l'autre sens. Cette proposition est tout de même utile   comme on peut le voir dans la proposition ci-dessous, ainsi que dans la Section \ref{section conserv}.
  \end{enumerate}
    \end{rem}

  \begin{prop}\label{mot var abel} (Kleiman \cite{GKL})
Soit $A$ une variété abélienne de dimension $g$. Alors son motif homologique admet une décomposition de Künneth
\[\mathfrak{h}(A)=\bigoplus_{n=0}^{2g} \mathfrak{h}^n(A)\]
et un isomorphisme de motifs en algèbres de Hopf graduées
\[\bigoplus_{n=0}^{2g} \mathfrak{h}^n(A) = \bigoplus_{n=0}^{2g} \Sym^n\mathfrak{h}^1(A)\]
qui donne en cohomologie l'isomorphisme classique\footnote{Le changement de puissance symétrique en alternée est le même que celui dans la Remarque \ref{rem kimura}\eqref{koszul}.} $H^*(A)=\Lambda^*H^1(A)$.
De plus $A$ vérifie la conjecture standard de type Lefschetz (Conjecture \ref{conj autodual} et Conjecture \ref{CSTL})
\[    \mathfrak{h}^{2g-n}(A)(g)  = \mathfrak{h}^n(A)^\vee \cong  \mathfrak{h}^n(A)(n).\]
    \end{prop}
    \begin{proof}
    La décomposition de Künneth a déjà été discutée dans la Remarque \ref{rem kunn abel}\eqref{kunn abel}.

    On considère l'inclusion diagonale $\Delta : A \hookrightarrow A^n$. Elle induit une application
    $\mathfrak{h}(\Delta) : \mathfrak{h}(A)^{\otimes n} \rightarrow \mathfrak{h}(A)$ qui donne   une application $ \Sym^n\mathfrak{h}^1(A)  \rightarrow \mathfrak{h}^n(A).$
    De façon analogue on  construit une application dans l'autre sens en partant du morphisme de somme $s : A^n \rightarrow A$. 
    
    On peut maintenant appliquer la Proposition \ref{prop conserv} et déduire $ \bigoplus_{n=0}^{2g} \mathfrak{h}^n(A) \cong \bigoplus_{n=0}^{2g} \Sym^n\mathfrak{h}^1(A)$. Les applications construites ci-dessus respectent les structures d'algèbre de Hopf par définition de ces dernières, ce sont d'ailleurs les structures d'algèbre de Hopf qui forcent ces applications à être des isomorphismes en cohomologie.
    
    L'égalité  $\mathfrak{h}^{2g-n}(A)(g)  = \mathfrak{h}^n(A)^\vee$  est donnée par la dualité de Poincaré, voir la Remarque \ref{poincare motif}. Pour l'autodualité $\mathfrak{h}^n(A)^\vee \cong  \mathfrak{h}^n(A)(n)$ on se réduit au cas $n=1$, grâce à l'égalité  $ \Sym^n\mathfrak{h}^1(A)   = \mathfrak{h}^n(A)$.
    
    Pour $\mathfrak{h}^1(A)^\vee \cong  \mathfrak{h}^1(A)(1)$ on applique encore la Proposition \ref{prop conserv}. Une application est donnée par l'opérateur de Lefschetz $L^{g-1}: \mathfrak{h}^1(A) \rightarrow  \mathfrak{h}^{2g-1}(A)(g-1)=\mathfrak{h}^1(A)^\vee(-1)$. Pour  l'autre sens, fixons une isogénie entre $A$ et sa duale $A^\vee$, ce qui donne un isomorphisme $\mathfrak{h}^1(A)\cong \mathfrak{h}^1(A^\vee)$. Il s'agit alors de construire une application de $\mathfrak{h}^1(A)^\vee(-1)$ vers $\mathfrak{h}^1(A^\vee)$. C'est la donné d'un diviseur sur $A\times A^\vee$, le diviseur de Poincaré convient.    
    \end{proof} 
    \begin{rem}
    Les résultats de la proposition ci-dessus sont valables même dans $\CHM(k)$ mais leur preuve est plus délicate et repose sur la transformée de Fourier pour les anneaux de Chow de variétés abéliennes, voir aussi la Section \ref{AHP}.
    \end{rem}

         \subsection*{Motifs et théorèmes de comparaison}

\begin{exemple}\label{K3 fermat} (Réduction supersingulière et $\Q$-structures.)
Soit $S$ une surface K3 complexe de rang de Picard maximal et choisissons un premier $p$ de bonne réduction tel que la réduction $S_p$ soit supersingulière.
La surface quartique de Fermat avec $p\equiv -1 \,\,(4)$ est un tel exemple \cite{ShioF}.
Dans ce cas le motif homologique vérifie
\[ \mathfrak{h}(S) =  \mathbbm{1} \oplus \mathfrak{h}^2(S_0)  \oplus \mathbbm{1}(-2) \hspace{0.5cm} \textrm{et} \hspace{0.5cm} \mathfrak{h}^2(S_0) = \mathbbm{1}(-1)^{\oplus 20} \oplus  \mathfrak{h}^{2,\tr}(S).
\]
Le motif $\mathfrak{h}^{2,\tr}(S)$ est appelé motif transcendant. Sa réalisation singulière est une structure de Hodge de dimension $2$ et de type $(0,2)$. Sa réduction   modulo $p$ est isomorphe à $\mathbbm{1}(-1)^{\oplus 2}$.

Notons $M$ le motif $\mathfrak{h}^{2,\tr}(S)(1)$, $V_B$ sa réalisation singulière, $M_p$ sa réduction modulo $p$, et $Z_p$ les classes algébriques en caractéristique $p$, c'est-à-dire $Z_p=\Hom(\mathbbm{1}, M_p)$. 
Remarquons que $V_B$ et $Z_p$ sont des $\Q$-espaces vectoriels de dimension $2$.

Considérons les identifications
\begin{align}\label{identificazione fq}V_B \otimes_\Q \Q_\ell = R_\ell(M_0) = R_\ell(M_p) = Z_p\otimes_\Q \Q_\ell,\end{align}
où la première est donnée par le théorème de comparaison d'Artin entre cohomologie singulière et $\ell$-adique, la deuxième suit du théorème de changement de base propre et lisse et la dernière vient de la propriété de supersingularité.
On peut imaginer l'identification $V_B \otimes_\Q \Q_\ell =Z_p\otimes_\Q \Q_\ell$ comme un isomorphisme de périodes. C'est ce point de vue qui a inspiré le travail présenté dans la Section \ref{section andre}.

Il est naturel de se demander si les deux $\Q$-structures $V_B$ et $Z_p$ sont respectées sous cette identification. Expliquons pourquoi la réponse est non. Si on considère le cup-produit sur $V_B$ et $Z_p$ on déduit deux  $\Q$-formes quadratiques $q_B$ et $q_Z$. Si les espaces étaient égaux on aurait, en particulier, que ces deux formes quadratiques seraient isomorphes. Or $q_B$ est définie positive par les relations de Hodge--Riemann et $q_Z$ est définie négative par le théorème de l'indice de Hodge\footnote{Voir respectivement la Définition \ref{def pol pair} et la Conjecture \ref{CSTH}, qui est connue pour les diviseurs. Les formes quadratiques qu'y apparaissent sont obtenues à partir du cup-produit par un changement de signe.}.
\end{exemple}
\begin{rem} Gardons les notations de l'exemple ci-dessus.
C'est intéressant de regarder les $\Q$-formes quadratiques $q_B$ et $q_Z$ non seulement à la place à l'infini mais aussi aux autres completions de $\Q$. L'identification \eqref{identificazione fq} montre l'égalité $q_B\otimes \Q_\ell=q_Z\otimes \Q_\ell$ pour tout $\ell \neq p$. La seule place qui reste à déterminer est en $p$. Or elle est déterminée par les autres places\footnote{Cela suit de  la formule du produit sur les symboles de Hilbert, l'argument sera détaillé dans la Section \ref{positivite}.} et en particulier $q_B\otimes \Q_p \neq q_Z\otimes \Q_p$ puisque $q_B\otimes \R \neq q_Z\otimes \R$.

Une remarque amusante : observons que la forme quadratique $q_Z$ reconnait le nombre premier $p$, c'est en effet le seul nombre premier pour lequel $q_B\otimes \Q_p \neq q_Z\otimes \Q_p$. En particulier si on fait varier $p$ parmi tous les premiers à réduction supersingulière on trouvera des $\Q$-formes quadratiques toujours différentes.

Plus sérieusement,  remarquons que les $\Q_p$-espaces $V_B \otimes_\Q \Q_p$ et  $ Z_p\otimes_\Q \Q_p$ ont une interprétation cohomologique : $V_B \otimes_\Q \Q_p$ est isomorphe à la réalisation $p$-adique de $M$, encore par le théorème de comparaison   d'Artin, et $ Z_p\otimes_\Q \Q_p$ est la partie Frobenius invariante de la réalisation crystalline de $M_p$. Il est naturel alors de se demander si la relation  $q_B\otimes \Q_p \neq q_Z\otimes \Q_p$ peut s'obtenir par des méthodes purement $p$-adiques, par exemple via la théorie de Hodge $p$-adique. 
La réponse  est oui, elle sera esquissé dans la Section \ref{positivite}. Ces techniques permettront par ailleurs de montrer des cas de la Conjecture de positivité \ref{Conj pos}, voir encore la Section \ref{positivite}.
\end{rem}

\begin{exemple}\label{exemple periode} (Périodes.)
Considérons la catégorie  $\CHM(\Q)$ des motifs définis sur le corps $\Q$. Elle est munie de deux foncteurs vers les $\Q$-espaces vectoriels gradués
\[ R_{B} , R_{\dR}  :  
 \CHM(\Q) 
  \longrightarrow 
   \GrVect_{\Q}
    \]
    la réalisation singulière, ou de Betti, et la réalisation de de Rham algébrique.
    Les théorèmes de comparaison fournissent une identification
    \[R_{B} \otimes_\Q \C = R_{\dR}\otimes_\Q \C \]
     essentiellement induite par l'integration des formes différentielles algébriques sur des simplexes topologiques. Les coefficients complexes qui apparaissent dans cette identification sont appelés périodes et sont attendus être aussi transcendants que possible.
     
     Le formalisme tannakien montre l'existence d'isomorphismes  \[R_{B} \otimes_\Q \overline{\Q} \cong R_{\dR}\otimes_\Q \overline{\Q} \] de foncteurs monoïdaux
     (mais on ne sait pas en construire un explicitement). Montrons qu'en revanche il ne peut pas y avoir d'isomorphisme monoïdal entre $R_{B}  $ et   $R_{\dR}    $ et même que l'on a 
 \[R_{B} \otimes_\Q \R \not\cong R_{\dR}  \otimes_\Q \R.  \] 
          
     Pour le montrer il suffit de trouver un motif   $M$ muni d'une application
     \[q:\Sym^2 M  \longrightarrow \mathbbm{1}\]
    telle que  les réalisations $R_{B}(q)$ et  $R_{\dR}(q)$ sont deux  formes quadratiques avec signature différente.
    
    Par exemple on peut prendre $M$ tel que  $R_{B}(M)$ soit une structure de Hodge de poids $0$ et types $(-1,1)$ et $(0,0)$ et $q$ tel que $R_{B}(q)$ soit une polarisation de $R_{B}(M)$, voir  Définition \ref{def pol pair}.
    Alors la signature de la forme quadratique $R_{B}(q)$ est $(\dim R_{B}(M)^{0,0} , 2 \cdot \dim R_{B}(M)^{-1,1})$.
    
 D'autre part la réalisation $R_{\dR}(q)$ respecte la filtration de de Rham et donc l'espace $\Fil^1R_{\dR}(q)$ est isotrope ce qui force la signature $(s_+,s_-)$ de $R_{\dR}(q)$ à vérifier $s_+\geq \dim\Fil^1R_{\dR}(q)$. Or  $ \dim\Fil^1R_{\dR}(q)=\dim R_{B}(M)^{-1,1}$, il suffit donc de trouver un exemple où  l'on a    $\dim R_{B}(M)^{-1,1} > \dim R_{B}(M)^{0,0}$   pour conclure.
    
    De tels exemples se trouvent   dans la catégorie engendrée par les courbes d'équation $C_n :  y^2=x^n-1$, voir   \cite{Sch}. Plus précisément $M$ sera\footnote{Remarquons que le premier candidat que l'on pourrait imaginer, à savoir $M=\mathfrak{h}^2(S)(1)$ avec $S$ une surface, ne peut pas fonctionner. En effet pour les surface on a l'inégalité opposée : $h^{1,1}>h^{0,2}$, voir \cite[Proposition 22]{Sch}.} la partie $G$-invariante du motif  $\mathfrak{h}^2(C^N_n)(1)$ pour certains entiers $n,N$  et  pour un groupe fini $G$  convenable agissant sur $C^N_n$ et défini sur $\Q$.
    
    Notons par ailleurs que, par le théorème de comparaison de de Rham, $R_{B} \otimes_\Q \R$ est isomorphe à  la cohomologie de de Rham classique du lieu complexe sous-jacent, calculée à l'aide des formes différentielles $\mathcal{C}^{\infty}$ à coefficients réels. En particulier on vient de montrer que, pour les variétés algébriques  définies sur $\R$, il ne peut pas y avoir d'isomorphisme naturel entre la cohomologie de de Rham algébrique   de la variété et  la cohomologie de de Rham classique du lieu complexe. 
         \end{exemple}
         
            \subsection*{Motifs de variétés abéliennes CM}
La Proposition \ref{mot var abel}   montre que le motif d'une variété abélienne $A$ est contrôlé par son $\mathfrak{h}^1(A)$. Si $A$ est une variété abélienne CM on peut utiliser l'action CM pour décomposer $\mathfrak{h}^1(A)$ et donc le motif de $A$ tout entier. Cette décomposition est bien utile : Clozel l'avait déjà utilisée pour un résultat sur la conjecture $\hom=\num$ (voir le Théorème  \ref{thm clozel}) et on l'utilisera également pour les Conjectures de conservativité et de positivité, voir les Section \ref{section conserv} et \ref{positivite}.

D'autre part, bien qu'élémentaire, cette décomposition n'est pas digeste à la première lecture, on encourage à y revenir au fur et à mesure de ses applications.

\begin{exemple}\label{decomp CM}
    Soit $A$ une variété abélienne simple de dimension $g$ définie sur un corps fini. Fixons une polarisation et considérons l'involution de Rosati induite sur $\End(A) \otimes \Q.$
    Par Honda--Tate il existe un corps CM de degré $2g$ et une inclusion $F \subset \End(A) \otimes \Q$ tels que l'involution laisse stable $F$ et agit comme la conjugaison complexe sur $F$. 
    
    Considérons l'action de $F$ sur le motif $\mathfrak{h}^1(A)$ dans la catégorie  $\CHM(k)_{\tilde{F}}$ des motifs de Chow à coefficients dans une clôture galoisienne $\tilde{F}$ de $F$. Elle décompose le motif 
    \begin{align}\label{decomp h1}
    \mathfrak{h}^1(A) =L_1 \oplus \ldots \oplus L_{2g}
    \end{align}
    en une somme de $2g$ facteurs  échangés par l'action du groupe de Galois $\Gal(\tilde{F}/\Q)$ et la réalisation de chaque facteur est une droite propre pour l'action de $F$. De plus le choix d'une polarisation induit un morphisme dans $\CHM(k)$
    \[q_1 : \mathfrak{h}^1(A) \otimes \mathfrak{h}^1(A) \longrightarrow \mathbbm{1}(-1).\]
    Par rapport à cet accouplement, une droite propre est orthogonale à toutes les autres hormis sa conjuguée complexe.

    À l'aide de la formule $\mathfrak{h}^n(A)  = \Sym^n\mathfrak{h}^1(A)$ on déduit une décomposition de $\mathfrak{h}^n(A)$ dans $\CHM(k)_{\tilde{F}}$ en somme de facteurs dont la réalisation a dimension un : chaque facteur correspond au produit tensoriel de $n$ différents $L_i$. De plus l'accouplement $q_n=\Sym^n q_1$ rend la réalisation d'une telle droite orthogonale à toutes les autres hormis sa conjuguée complexe.
    \end{exemple}
    \begin{rem}\label{rem decomp} Gardons les notations de l'exemple ci-dessus.
    \begin{enumerate}
    \item\label{decomp C} Soient $\alpha_1,\ldots, \alpha_{2g}$ les valeurs propres de l'action du Frobenius sur $\mathfrak{h}^1(A)$ comptées avec multiplicité. Quitte à les renuméroter on a $\alpha_i\cdot \alpha_{2g-i}=q$, où $q$ est le cardinal du corps de base. Cette symétrie des valeurs propres provient de l'accouplement parfait $q_1$.
    
    Les valeurs propres de l'action du Frobenius sur $\mathfrak{h}^n(A)$ sont données par tous les  produits possibles de $n$ distincts $\alpha_i$. La dimension de l'espace des classes Galois-invariantes dans $H^{2n}(A)(n)$ est alors donnée par le nombre de collections  $\{\alpha_{i_1},\ldots, \alpha_{i_{2n}} \}$ vérifiant
    \begin{align}\label{disarli}\alpha_{i_1}\cdot\ldots\cdot \alpha_{i_{2n}}=q^n.\end{align}
    La conjecture de Tate prédit que chaque droite propre de $H^{2n}(A)(n)$ correspondant à une telle collection contient une classe algébrique. 
    
    Cette conjecture est connue pour les diviseurs. Une droite propre contient donc une intersection de diviseurs si et seulement si la collection $\alpha_{i_1},\ldots, \alpha_{i_{2n}} $ vérifie
     \begin{align}\label{darienzo}\alpha_{i_j} \cdot \alpha_{i_{2n-j}}=q, \hspace{0.5cm}\forall j,
     \end{align}
     quitte à  renuméroter. 
     \item\label{decomp R}
Soit $F_0$ le plus grand sous-corps totalement réel de $\tilde{F}$. Le groupe de Galois $\Gal(\tilde{F}/F_0)$ est d'ordre deux, engendré par la conjugaison complexe. Son action   recolle la décomposition de l'exemple ci-dessus en une décomposition de  $\mathfrak{h}^n(A)$ dans la catégorie $\CHM(k)_{F_0}$   des motifs de Chow à coefficients dans $F_0$. Cette décomposition est orthogonale par rapport à  l'accouplement $q_n$. Les   facteurs obtenus sont de rang un ou deux.
    \item\label{decomp Q} L'action du groupe de Galois $\Gal(\tilde{F}/\Q)$ recolle la décomposition de l'exemple ci-dessus en une décomposition de  $\mathfrak{h}^n(A)$ dans la catégorie $\CHM(k)$   des motifs de Chow à coefficients rationnels. Cette décomposition est orthogonale par rapport à  l'accouplement $q_n$. Le rang des   facteurs obtenus varie et vaut au plus $2g$.
    
    Par la description du point \eqref{decomp C} on remarque que chaque facteur de $\mathfrak{h}^{2n}(A)(n)$ rentre dans une des trois catégories suivantes : 
          \begin{itemize}
          \item[(a)] La réalisation du facteur est engendrée par des classes qui sont toutes intersections de diviseurs,
          \item[(b)] La réalisation du facteur est Frobenius invariante mais ne contient aucune  intersection de diviseurs,
          \item[(c)] Le  Frobenius agissant sur la réalisation du facteur n'a aucun vecteur fixe.
          \end{itemize}
          Pour les questions de cycles algébriques c'est surtout la   classe (b) qui est intéressante. Si $A$ est de dimension quatre les facteurs de ce type ont toujours rang deux. Pour le montrer il s'agit d'étudier les quadruplets vérifiant \eqref{disarli} mais qui ne vérifient pas \eqref{darienzo}.  C'est une étude élémentaire mais dont la combinatoire est délicate, voir \cite[\S 7]{Anc21} pour les détails.
         \end{enumerate}
         \end{rem}

    \section{Autodualité  et conservativité}\label{section conserv}
Cette section concerne les Conjectures \ref{conj conserv} et \ref{conj autodual} de conservativité et d'autodualité et les résultats que l'on peut obtenir pour les variétés abéliennes. On notera 
\[ \CHM(k)^{\ab},   \hspace{0.5cm} \Mot(k)^{\ab}  \hspace{0.5cm}  \textrm{et}  \hspace{0.5cm}  \NUM(k)^{\ab}\]
les catégories de motifs engendrées par les motifs de variétés abéliennes.

On commence par rappeler les théorèmes fondamentaux de semisimplicité de Jannsen et de nilpotence de Kimura, puis on en déduit les   Conjectures \ref{conj conserv}, et  \ref{conj autodual}  pour $\CHM(k)^{\ab} $, avec $k=\C$.

Dans une deuxième partie on explique le contenu de \cite{Ascona} qui étudie ces conjectures pour $k=\mathbb{F}_q$. Il faudra combiner les théorèmes de Jannsen et Kimura avec les décompositions de l'Exemple \ref{decomp CM} induites par la multiplication complexe. Cette méthode est inspirée par un travail de Clozel \cite{Clozel} que nous rappellons également. 

\begin{thm}\label{thm jannsen}(Jannsen \cite{Jannsen})
La catégorie $\NUM(k)$ des motifs numériques est semisimple.
\end{thm}
\begin{thm}\label{thm kimura}(Kimura--O'Sullivan \cite{Kim,OS})
Le noyau du foncteur de projection
\[ \pi_{\num} : \CHM(k)^{\ab}  \longrightarrow  \NUM(k)^{\ab}\]
est nilpotent. En particulier le foncteur $\pi_{\num}$ est conservatif et toute décomposition dans $\NUM(k)^{\ab}$ se relève en une décomposition dans $\CHM(k)^{\ab}$. L'énoncé reste valable si on remplace  $ \NUM(k)^{\ab} $ par $\Mot(k)^{\ab} $ (ou n'importe quelle catégorie tensorielle quotiente). 
\end{thm}

\begin{prop}
 La Conjecture d'autodualité \ref{conj autodual} et la conjecture $\hom=\num$ sont vraies pour tout motif dans $\CHM(\C)^{\ab} $. De plus le foncteur de réalisation singulière
 \[ R  :  
 \CHM(\C)^{\ab}
  \longrightarrow 
   \GrVect_{\Q}
    \]
est conservatif  (cf. Conjecture \ref{conj conserv}).
\end{prop}
\begin{proof}
La conjecture standard de type Lefschetz est vraie pour les variétés abéliennes (Proposition \ref{mot var abel}).
En caractéristique zéro, on peut en déduire $\hom=\num$   (Remarque \ref{rem HR2}).

De plus, on peut  munir $\mathfrak{h}^n(A)$ d'un accouplement
\[  \mathfrak{h}^n(A) \otimes \mathfrak{h}^n(A)  \rightarrow \mathbbm{1}(-n) \]
dont la réalisation singulière est une polarisation (Proposition \ref{pol mot}).
On en déduit que pour tout facteur direct $M$ du motif homologique $\mathfrak{h}^n(A)$ la restriction de l'accouplement à $M$ induit un isomorphisme $M \cong M^{\vee}(-n)$. (Ce fait est impliqué par la Proposition \ref{HS semisimple} et c'est le point crucial où on l'utilise que les motifs sont définis sur $\C$. )
L'autodualité pour les motifs homologiques se relève aussi dans $\CHM(\C)^{\ab}$ par le Théorème \ref{thm kimura}.

Passons maintenant à la conservativité et considérons le diagramme \eqref{diagramme motifs}. Encore par le Théorème \ref{thm kimura}, il suffira de démontrer la conservativité de la réalisation  des motifs homologiques  $I  : \Mot(\C)^{\ab} \rightarrow  \GrVect_{\Q}$.
D'autre part, la   conjecture $\hom=\num$ déduite au début de la preuve dit que  les catégories $\Mot(\C)^{\ab}$ et $\NUM(\C)^{\ab}$ coincident. En particulier, par le Théorème \ref{thm jannsen}, $I$ est un foncteur entre catégories semisimples, il est donc conservatif.
\end{proof}
 Dans la suite de la section on travaille sur un corps fini  $k=\mathbb{F}_q$. Ce qui remplacera l'utilisation de la polarisation dans la preuve ci-dessus est la multiplication complexe, via les décompositions de l'Exemple \ref{decomp CM}.

\begin{thm}\label{thm clozel}(Clozel \cite{Clozel}) Soit $A$ une variété abélienne sur un corps fini. Alors il existe une infinité de nombres premiers $\ell$ tels que l'équivalence numérique coïncide avec l'équivalence homologique pour la cohomologie $\ell$-adique.
 \end{thm}

 \begin{proof}
 On se ramène au cas où $A$ est simple. Soient $\Q \subset F_0 \subset \tilde{F}$ les corps de nombres introduits dans l'Exemple \ref{decomp CM} et la Remarque \ref{rem decomp}\eqref{decomp R}. (Ces corps dépendent de $A$ et plus précisément du choix d'un corps CM dans ses endomorphismes.) 
 
 Fixons un nombre premier $\ell$ tel qu'il existe une place $\lambda$ de $F_0$ au-dessus de $\ell$ telle que la complétion $(F_0)_\lambda$ ne contienne pas $\tilde{F}$. On va montrer qu'un tel $\ell$ convient. Remarquons qu'il y a une infinité de tels $\ell$ et que l'on peut estimer leur densité avec Chebotareff.  
  
 Considérons les classes algébriques de codimension $n$ que l'on voit comme classes dans  $\mathfrak{h}^{2n}(A)(n)$ et soit $q_{2n}$ l'accouplement construit dans l'Exemple \ref{decomp CM}. Puisqu'il est non dégénéré il suffit de voir que pour chaque classe algébrique non nulle $\gamma$ il existe une classe algébrique $\delta$ telle que $q_{2n}(\gamma, \delta)\neq 0$.
 On vérifie que la question est stable par changement de coefficients et on travaille avec les cycles à coefficients dans  $F_0$ et la réalisation $\lambda$-adique $R_\lambda$ à valeurs dans les $(F_0)_\lambda$-espaces vectoriels. Dans ce cas on utilise la décomposition de la Remarque \ref{rem decomp}\eqref{decomp R} en plans et droites. Il suffit alors de travailler avec un seul de ces facteurs $M$ et supposer que $\gamma$ vit dans  $M$. 
 Dans ce cas, si la forme quadratique $q_{2n}$ est sans vecteur isotrope   sur $M$  le choix $\delta=\gamma$ convient.
 
 Si $M$ a dimension un cela suit du fait que $q_{2n}$ est non dégénérée sur chaque facteur de la décomposition et donc sur $M$. Si $M$ a dimension deux alors il admet au plus deux droites isotropes. Ces deux droites existent au moins sur $M\otimes_{F_0} \tilde{F}$ : il s'agit de la décomposition en droites de l'Exemple \ref{decomp CM}. Montrons que ces droites ne sont pas contenues dans le $(F_0)_\lambda$-espace vectoriel $R_\lambda(M)$. Pour cela il suffira de construire   un endomorphisme $f : M \rightarrow M$   dans la catégorie $\CHM(k)_{F_0}$ dont ces droites sont des droites propres et de valeurs propres appartenant à $\tilde{F}-F_0$. Cela donnera la conclusion voulue puisqu'on a   que $(F_0)_\lambda$ ne contient pas $\tilde{F}$. 
 
 La construction de ce $f$ procède ainsi. Considérons la décomposition \eqref{decomp h1}. Quitte à changer la numérotation,  le motif $M$ est de la forme    
 \[M=(L_1\otimes \ldots \otimes L_{2n}) \oplus (\bar{L}_1\otimes \ldots \otimes \bar{L}_{2n}), \]
 où $\bar{\cdot}$ est la conjugaison complexe.
Fixons un ordre sur les $L_i$, ceci permet de réaliser $M$ comme facteur direct de $\mathfrak{h}^1(A)^{\otimes 2n}$ dans la catégorie  $\CHM(k)_{F_0}$. On peut alors définir $f$ par l'action induite par  un générateur de $F \subset \End(A) \otimes \Q$ sur le premier terme du produit tensoriel  et l'identité sur les autres $2n-1$.
 \end{proof}
 
  \begin{prop}\label{prop auto}
 La conjecture d'autodualité   \ref{conj autodual} est vraie pour les motifs de $\CHM(\mathbb{F}_q)^{\ab}$ de poids pair et dont la réalisation a dimension un.
 \end{prop}

 \begin{proof}
Soient $X$ un tel motif et $n$ son poids (pair). On a une variété abélienne $A$, telle que $X$ est facteur direct de $\mathfrak{h}^n(A)$.

On veut montrer que $X \cong X^{\vee}(-n)$. Par le Théorème \ref{thm kimura} il suffit de montrer $\pi_{\num}(X)  \cong \pi_{\num}(X)^{\vee}(-n)$.  Par la semisimplicité de Jannsen, il suffit alors de montrer que $\Hom(\pi_{\num}(X) , \pi_{\num}(X)^{\vee}(-n))\neq 0$. Cet énoncé peut se démontrer après  extension des scalaires.  On étend les scalaires au corps $F_0$ de la Remarque \ref{rem decomp}\eqref{decomp R}.

On utilise la décomposition de cette même remarque. Par semisimplicité on peut supposer que $X$ soit un facteur direct d'un des facteurs $M$ de cette décomposition. Rappelons que l'on  dispose d'un accouplement $q_n$ non-dégénéré sur $M$ et que  $M$ a dimension un ou deux. 

Si $M$ a dimension un alors $X=M$ et on a terminé. Si $M$ a dimension deux, remontons aux motifs homologiques, via  le Théorème \ref{thm kimura}. On pourra alors utiliser la réalisation et    il suffira de montrer que l'accouplement restreint à la droite qui est la	réalisation de $X$ reste non-dégénéré, autrement dit que la droite n'est pas isotrope. (C'est ici que l'on utilisera que  le poids est pair. Remarquons notamment que si le poids est impair l'accouplement sur $M$ est alterné et donc toute droite est isotrope.)

La subtilité est que la catégorie des motifs homologiques dépend a priori de la cohomologie choisie mais la bonne nouvelle est qu'il suffit d'étudier une seule cohomologie bien choisie. On utilise alors la cohomologie $\lambda$-adique comme dans la preuve du
Théorème \ref{thm clozel} avec le même choix de $\lambda$ : on y avait montré  que la réalisation de $M$ est sans vecteurs isotropes.
\end{proof}

 \begin{thm}\label{thm conservatif}
Les foncteurs de réalisation $\ell$-adique 
 \[ R_{\ell}  :  
 \CHM(\mathbb{F}_q)^{\ab}
  \longrightarrow 
   \GrVect_{\Q_\ell}
    \]
    sont conservatifs.
 \end{thm}
  \begin{proof}
Soit $f : X \rightarrow Y$ une application dans  $\CHM(\mathbb{F}_q)^{\ab}$ telle que $R_{\ell} (f ): R_{\ell} (X )\rightarrow R_{\ell} (Y)$ soit un isomorphisme. On veut montrer que $f$ est un isomorphisme également. 

 Par le Théorème \ref{thm kimura} il suffira de travailler avec l'équivalence homologique. En utilisant la décomposition de Künneth, il suffira de supposer que $R_{\ell} (X )$ et $R_{\ell} (Y)$ sont concentrés en un même degré cohomologique.  

Supposons d'abord que $R_{\ell} (X )$ et $R_{\ell} (Y)$  aient dimension un.
On dispose d'applications
\[  \mathbbm{1}  \longrightarrow Y \otimes X^{\vee} \cong Y^{\vee}\otimes X  \longrightarrow \mathbbm{1}, \]
où la première et la dernière application sont obtenues par adjonction à partir de $f$ et l'isomorphisme central vient de la Proposition \ref{prop auto}.

 Dans ce cas les réalisations des applications ci-dessus sont des isomorphismes et par  la Proposition \ref{prop conserv} on a  $\mathbbm{1}  \cong Y \otimes X^{\vee} $ donc $X \cong Y$.

Travaillons maintenant dans le cas général : soit $d$ la dimension de  $R_{\ell} (X )$ et $R_{\ell} (Y)$. Supposons que leur degré cohomologique soit pair (sinon il faudra remplacer des produits extérieurs par des produits symétriques dans la suite).
L'application 
\[\Lambda^d f : \Lambda^d  X\longrightarrow   \Lambda^d  Y  \]
retombe dans le cas particulier de la dimension un traité au-dessus. C'est donc un isomorphisme et on dispose de l'application $(\Lambda^d f)^{-1}$. 

On peut maintenant construire une application $g:Y \rightarrow X$  via
\[Y \cong  \Lambda^d  Y  \otimes (\Lambda^{d-1}  Y)^{\vee} \longrightarrow \Lambda^d  X  \otimes (\Lambda^{d-1}  X)^{\vee} \cong X \]
où le premier et le dernier isomorphisme viennent du lemme ci-dessous et l'application centrale est $(\Lambda^d f)^{-1} \otimes (\Lambda^{d-1}  f)^{\vee}.$
Par construction, $R_\ell(g)$ est un isomorphisme, on conclut alors par la Proposition \ref{prop conserv}.
\end{proof}
\begin{lem}(\cite[Lemma 3.2]{OS})
Soit $M$ un motif homologique dont la réalisation est concentrée en un degré pair et de dimension $d$, alors 
\[M \cong  \Lambda^d  M  \otimes (\Lambda^{d-1}  M)^{\vee} \]
\end{lem}
\begin{proof}
Le motif $\Lambda^d  M $ est   un facteur direct de $M \otimes \Lambda^{d-1}  M$. Ceci fournit deux applications entre $M$ et $\Lambda^d  M  \otimes (\Lambda^{d-1}  M)^{\vee}$ dans les deux directions. Leur réalisation est un isomorphisme, on conclut par la Proposition \ref{prop conserv}.
\end{proof}

\section{Positivité en caracteristique positive}\label{positivite}
  Dans cette section nous étudions la Conjecture de positivité \ref{Conj pos}. Le résultat principal dit que la conjecture est vérifiée pour les motifs de dimension $2$ et à réduction supersingulière (Théorème \ref{thm positif}). Nous expliquons ensuite comment appliquer ce résultat pour déduire la Conjecture standard de type Hodge \ref{CSTH} pour certaines variétés, par exemples les variétés abéliennes de dimension quatre. Puis nous discutons le rôle de l'hypothèse de dimension $2$.

   \begin{thm}\label{thm positif}  Soient $M$ un motif homologique sur un corps $k$ de caractéristique $p$ et 
 $q: \Sym^2 M \rightarrow \mathbbm{1} $
 un morphisme dans $\Mot(k)$. 
 Supposons que $q$   soit la réduction modulo $p$ d'une application
 $\tilde{q}: \Sym^2 \tilde{M} \rightarrow \mathbbm{1}$
 définie en caractéristique zéro. 
 
 Définissons  $q_Z$ comme la restriction de $q$ à toutes les classes algébriques $Z(M)=\Hom(\mathbbm{1}, M)$ de $M$ et  $q_B$ comme  la réalisation singulière de $\tilde{q}$.
 
 Supposons  que $q_B$  soit une polarisation et supposons  avoir $M\cong  \mathbbm{1}^{\oplus 2}$. Alors $q_Z$ est définie positive.
 \end{thm}

  \begin{rem} (Sur les hypothèses : relevabilité et rang $2$.)
  Il n'est pas rare d'avoir des motifs qui se relèvent à la caractéristique zéro.  Par exemple il est attendu que tout motif sur un corps fini se relève, car la conjecture de Tate prédit qu'un tel motif serait de type abélien. En général, même si un motif se relève, ses classes algébriques ne se relèveront pas à la caractéristique zéro, ce qui rend les résultats de positivité difficiles, puisqu'ils ne peuvent pas se déduire des propriétés de positivité des polarisation, voir Définition \ref{def pol pair}.

  L'hypothèse restrictive dans le théorème ci-dessus est la dimension deux. La façon d'utiliser ce résultat pour déduire la conjecture standard de type Hodge pour certaines variétés est la suivante. On décompose le motif d'une   variété donnée autant que possible. Certains facteurs ne posséderont pas de classes algébriques, d'autres en posséderont uniquement certaines pour lesquels la conjecture standard de type Hodge peut se déduire des cas connus : par exemple ce sont des classes qui se relèvent à la caractéristique zéro ou  qui sont  construites à partir de   diviseurs. Enfin, ils resteront parfois des facteurs qui possèdent des classes algébriques qui ne se ramènent pas à des cas connus. Le point est alors de trouver des variétés pour lesquels ces derniers facteurs sont de dimension deux. Un exemple est donné dans le corollaire ci-dessous.
      \end{rem}
\begin{cor}Soit $A$ une variété abélienne de dimension quatre définie sur un corps de caractéristique $p$. Alors
\begin{enumerate}
\item La Conjecture standard de type Hodge \ref{CSTH} est vraie pour $A$,
\item Le produit d'intersection
\[\CH^2(A)/{\num} \times \CH^2(A)/{\num}  \longrightarrow \Q\]
est de signature $(\rho_2 - \rho_1 +1; \rho_1 - 1)$, où
$\rho_n=\dim_\Q (\CH^2(A)/{\num}),$
\item  Il y a une infinité de nombres premiers $\ell\neq p$ pour lesquels l'équivalence numérique sur $A$ coïncide avec l'équivalence homologique pour la cohomologie $\ell$-adique.
\end{enumerate}
\end{cor}
 \begin{proof}
 Par un argument de spécialisation on peut supposer que le corps de définition est fini. On peut alors utiliser la décomposition de la Remarque \ref{rem decomp}\eqref{decomp Q}. Les facteurs qui sont a priori mystérieux pour la Conjecture standard de type Hodge \ref{CSTH} sont ceux de type (b), dans la notation de la même remarque.   Il se trouve  que tous ces facteurs de toutes les variétés abéliennes de dimension quatre sont bien de dimension deux. Ce fait est  un petit miracle combinatoire, voir \cite[\S 7]{Anc21} pour les détails ou la Remarque \ref{rem decomp}\eqref{decomp Q} pour un aperçu.
    
    Les points (1) et (2) sont en fait équivalents. Cette équivalence n'utilise pas le fait que $A$ est une variété abélienne mais uniquement la dimension quatre. Elle se déduit de la décomposition en parties primitives.
    
    Si le corps de définition est fini le point (3) est un cas particulier du Théorème \ref{thm clozel} de Clozel. Pour se ramener aux corps finis on   spécialise et on utilise la Proposition \ref{pol mot}.
       \end{proof}
    \begin{rem}\begin{enumerate}
    \item (Dimension supérieure.)
Pour les variétés abéliennes  de dimension quelconque on pourra encore utiliser la décomposition de la Remarque \ref{rem decomp}\eqref{decomp Q}. En général les facteurs (b) de la remarque    auront dimension plus grande que deux.  On peut tout de même trouver des exemples sporadiques pour lesquels ces facteurs de type (b) ont rang $2$ et déduire la conjecture standard de type Hodge à l'aide du Théorème \ref{thm positif}. Ceci a été récemment étudié par Koshikawa \cite{Koshi}.
\item (Supersingularité vs Frobenius invariant.)
On remarquera un petit décalage entre l'hypothèse de supersingularité $M\cong  \mathbbm{1}^{\oplus 2}$ du Théorème \ref{thm positif} et la caractérisation des facteurs (b). Tout d'abord remarquons que ces deux descriptions sont équivalentes sous la conjecture de Tate. 

Inconditionnellement, a priori, parmi les facteurs de type (b) certains pourraient ne pas posséder de classe algébrique : ces facteurs pourront être négligés à l'étude de la conjecture standard de type Hodge. 
Pour les autres on a besoin de montrer que dès qu'un facteur a une classe algébrique il est engendré par des classes algébriques. Ceci se montre en utilisant l'action CM mais l'argument nécessite de   travailler avec l'équivalence numérique. Si on travaillait avec l'équivalence homologique on se trouverait devant des problèmes similaires à ceux qui empêchent l'argument de Clozel du  Théorème \ref{thm clozel} de fonctionner pour tout nombre premier $\ell$. 
\end{enumerate}
 \end{rem}
    \begin{defin} (Symbole de Hilbert.)
    Soit $q$ une $\Q$-forme quadratique de dimension deux et $\nu=2,3,5,\ldots,\infty$ une place de $\Q$. On définit le symbole de Hilbert  $\varepsilon_\nu(q) $ de $q$ en $\nu$ comme étant $+1$ si $q(x,y)=z^2$ a une solution non-nulle dans la complétion $\Q_\nu$ et $-1$ sinon.
    \end{defin}
   
    \begin{prop}\label{prop hilb}
    Gardons les notations du Théorème \ref{thm positif} et soit $n$ l'unique entier tel que la réalisation singulière de $\tilde{M}$ soit une structure de Hodge de type $(-n,+n)$. Alors $q_Z$ est définie positive si et seulement si 
    \begin{align}\label{formule hilb}
    \varepsilon_p(q_Z) = (-1)^n \varepsilon_p(q_B),
    \end{align}
    ce qui est encore équivalent au fait que $q_Z\otimes \Q_p$ est isomorphe à  $q_B\otimes \Q_p$ si et seulement si $n$ est pair.
        \end{prop}
\begin{rem}
L'idée de la preuve de cette proposition a déjà été introduite dans l'Exemple \ref{K3 fermat} et la remarque qui le suit. Cette proposition est par ailleurs le point crucial où l'hypothèse de la dimension deux est nécessaire.
    \end{rem}
    \begin{proof}
    Tout d'abord on remarque que, pour tout nombre premier $\ell\neq p$, on a $q_B\otimes \Q_\ell   \cong q_Z\otimes \Q_\ell,$ c'est la combinaison du théorème de comparaison d'Artin et du changement de base propre et lisse en cohomologie $\ell$-adique. Cela implique en particulier que $\varepsilon_\ell(q_Z) = \varepsilon_\ell(q_B)$, mais aussi que le discriminant de $q_B$ et $q_Z$ coïncident dans $\Q^*/(\Q^*)^2$, car un nombre rationnel est un carré s'il l'est dans presque toute complétion.
    
D'autre part, en suivant la Définition \ref{def pol pair}, on a que $q_B$  est $(-1)^n$-définie positive. Cela implique en particulier que $\varepsilon_\infty(q_B) = (-1)^n$, et que le discriminant de $q_B$ est positif. On en déduit  que $q_Z$ a discriminant positif et donc,  puisqu'on est en rang deux, que $q_Z$ est définie positive ou définie négative. La positivité de $q_Z$ est   équivalente alors à $\varepsilon_\infty(q_Z)=+1$.

A l'aide de la formule du produit sur les symboles de Hilbert on a
\begin{align}
    \prod_\nu \varepsilon_\nu (q_Z) = 1 =  \prod_\nu \varepsilon_\nu (q_B).
    \end{align}
En simplifiant les facteurs $\ell$-adiques on  obtient
$\varepsilon_\infty(q_Z)\cdot \varepsilon_p(q_Z)=\varepsilon_\infty(q_B)\cdot \varepsilon_p(q_B)$ et donc
$\varepsilon_\infty(q_Z)\cdot \varepsilon_p(q_Z)=(-1)^n\cdot \varepsilon_p(q_B)$. On conclut que $\varepsilon_\infty(q_Z)=+1$ si et seulement si la formule \eqref{formule hilb} est satisfaite.

La dernière équivalence de la proposition suit du fait que deux $\Q_p$-formes quadratiques non-dégénérées de même rang   sont isomorphes si et seulement si elles ont le même discriminant et le même symbol de Hilbert.
\end{proof}    
 \begin{proof}[Démonstration du Théorème \ref{thm positif}]
   Par la proposition précédente on est ramené à  un problème purement $p$-adique. Ce dernier a en plus une  interprétation cohomologique qui permet de le traduire en  une question de théorie de Hodge $p$-adique. Pour expliquer cette traduction définissons  $V_{B,p}$  comme la réalisation étale $p$-adique de  $\tilde{M}$ et $V_{Z,p}$ comme la partie Frobenius invariante de la réalisation cristalline de $M$. Chacun de ces deux $\Q_p$-espaces vectoriels de dimension deux est muni d'une forme quadratique induite  respectivement par $\tilde{q}$ et $q$. Ces deux formes quadratiques ne sont rien d'autre que 
$q_B\otimes \Q_p$ et $q_Z\otimes \Q_p$. La première identification suit du   théorème de comparaison d'Artin. La deuxième vient du fait que la partie Frobenius invariante contient toujours l'espace engendré par les classes algébriques et dans ce cas cette inclusion est une égalité par dimension.
   
   Le théorème de comparaison $p$-adique montré  par Faltings fournit un isomorphisme
   \begin{align}\label{faltings}
   V_{B,p}  \otimes B_{\cris}  =  V_{Z,p} \otimes B_{\cris}  \end{align}
   fonctoriel, compatible à toutes les structures que l'on pourrait imaginer et en particulier avec les formes quadratiques $q_B$ et $q_Z$.
   On conclut à l'aide des deux phénomènes suivants. 
   
  (a) La matrice de changement de base  est calculable dans $\Mat_{2 \times 2}(B_{\cris} )$. Elle ne dépend que de l'entier $n$ de la filtration de Hodge et de l'algèbre  $\End(V_{B,p})$ des endomorphismes de $V_{B,p} $ comme représentation galoisienne.
     
   (b) La description de cette matrice est suffisante pour  déduire que $q_B\otimes \Q_p$ et $q_Z\otimes \Q_p$ sont isomorphes si et seulement si $n$ est pair.
\end{proof} 
On passe maintenant  à la description des deux phénomènes (a) et (b) de la fin de la preuve ci-dessus. Si le deuxième est élémentaire le premier est une propriété remarquable de la théorie de Hodge $p$-adique qui la différencie de la théorie de Hodge classique.   
\begin{exemple}\label{exem elem}
Supposons avoir deux $\R$-formes quadratiques définies et de dimension deux $q_1(x,y)=a_1 x^2 + b_1 y^2$ et $q_2(x,y)=a_2 x^2 + b_2 y^2$ et une identification $q_1\otimes_\R \C=q_2 \otimes_\R \C$. Supposons   savoir que la matrice de changement de base de l'identification est
 \[  \begin{pmatrix}
		i & 0\\
		0 & i
	\end{pmatrix}. \]
	On pourra alors en déduire que $q_1$ et $q_2$  ne sont pas isomorphes (une est définie négative et l'autre est définie positive). Les arguments au point (b) dans la preuve du Théorème \ref{thm positif} sont tout aussi élémentaires et ressemblent à cet exemple, avec $\R$ et $\C$ qui sont remplacés par $\Q_p$ et $B_{\cris} $.
\end{exemple}

\begin{rem} \label{rem faltings}(Théorème de comparaison $p$-adique vs   classique.)
Le théorème de comparaison $p$-adique 
   \begin{align}\label{compp}R_{p} \otimes  B_{\cris}  = R_{\cris}\otimes  B_{\cris}, \end{align}
dont \eqref{faltings} en est une instance,
 est souvent considéré comme l'analogue $p$-adique du théorème de comparaison entre cohomologie singulière et  cohomologie de de Rham algébrique
   \begin{align}\label{compC} R_{B} \otimes \C = R_{\dR}\otimes  \C \end{align}
   que l'on a discuté dans l'Exemple \ref{exemple periode}. En fait le théorème de comparaison $p$-adique a des avantages par rapport à sa version classique, que l'on liste ci-dessous. (C'est grâce à ces propriétés que l'on peut notamment calculer certaines matrices de changement de base  et déterminer des relations entre leurs entrées, cf. le point (a) de la preuve du Théorème \ref{thm positif}.)
   
   \
   
   Rappelons que les réalisations classiques d'un motif $M$  possèdent des structures supplémentaires. En particulier, $R_{\dR}(M)$ est munie d'une filtration, $R_{B}(M)$  est munie d'une structure de Hodge, $R_{p}(M)$ est munie de l'action du groupe de Galois d'un corps $p$-adique et $R_{\cris}(M)$ est un $\varphi$-module filtré, i.e. elle est munie d'une filtration et d'une action du Frobenius absolu $\varphi$. 
   
   Rappelons aussi que les coefficients des matrices qui apparaissent dans les comparaisons \eqref{compp} ou \eqref{compC} sont appelés périodes.
   
   \begin{enumerate}
   \item 
   L'anneau $B_{\cris} $ est muni des actions du Frobenius absolu $\varphi$ et du groupe de Galois absolu de $\Q_p$ ainsi que  d'une filtration\footnote{Certaines structures sont définies dans un plus gros anneau noté $B_\dR$. On ignore ici ce point pour simplifier l'exposition.}.
   La comparaison \eqref{compp} respecte ces trois structures. Quand on dispose d'objets cohomologiques suffisamment concrets, comme ceux de \eqref{faltings}, on peut explicitement décrire ces structures sur les périodes.
      
   Le corps $\C$ en revanche n'est pas muni de structures qui imiterait la filtration ou la structure de Hodge. Connaître la structure de Hodge sous-jacente à $R_{B}(M)$ n'aide pas à avoir des informations sur les périodes complexes du motif $M$.
   
   \item Dans \cite{Falt}, Faltings montre une équivalence de catégories entre certains $\varphi$-modules filtrés, dits admissibles, et certaines représentations de groupes de Galois  de corps $p$-adiques, dites cristallines. La condition d'admissibilité est toujours vérifiée par les modules d'origine géométrique, i.e. par les réalisations de motifs.   
   
    Cette équivalence est en plus compatible au théorème de comparaison \eqref{compp}. En particulier, des périodes dans $B_{\cris}$ qui ont un certain comportement par rapport à  $\varphi$ et à la filtration doivent correspondre à un unique $\varphi$-module filtré et donc aussi à une unique représentation de Galois.
   Autrement dit, les périodes $p$-adique associées à un motif donné sont  caractérisées par leur comportement par rapport à deux structures : Frobenius et filtration.

   Une structure de Hodge  est beaucoup plus riche qu'une filtration. Par exemple,  pour deux variétés d'une même famille, les espaces vectoriels filtrés correspondant seront  isomorphes alors que les structures de Hodge ne le seront pas, en général\footnote{Si on considère la courbe elliptique $E_t : y^2=x(x-1)(x-t)$ pour $t\in\Q-\{0,1\}$ et le motif $M_t=\mathfrak{h}^1(E)$ alors $R_\dR(M_t)$ est la donnée d'un $\Q$-espace vectoriel de dimension $2$ muni d'une droite à l'intérieur. En revanche  la structure de Hodge $R_B(M_t)$ détermine $E_t$ à isogénie près, notamment il y aura des structures de Hodge CM et d'autres qui ne le sont pas.}. On ne peut pas avoir une équivalence de catégories entre ces deux structures cohomologiques.
   
   (De façon informelle, le passage du cas complexe \eqref{compC} au  cas $p$-adique \eqref{compp} correspond à enrichir la structure de de Rham et à réduire celle de Hodge suffisamment pour avoir deux structures équivalentes.
En effet, la réalisation cristalline hérite la filtration de de Rham mais elle possède en plus l'action du Frobenius absolu $\varphi$.
D'autre part, d'après la conjecture de Mumford-Tate, une structure de Hodge est grosso-modo équivalente à une représentation de Galois d'un corps de nombres, or sur  $R_{p}(M)$ on ne  regarde que l'action d'un certain de ses sous-groupes de décomposition.)

\item  La condition d'admissibilité, discutée au point précédent, se trouve être relativement élémentaire à vérifier, grâce à \cite{CF}. Ceci permet de construire facilement des $\varphi$-modules filtrés admissibles et donc des matrices de périodes avec action de Frobenius et filtration prescrites. 

On ne dispose pas de méthode élémentaire de construction de périodes complexes. Il s'agit d'intégrales de formes différentielles algébriques qui sont souvent difficiles à calculer. Leurs relations sont prédites par la conjectures des périodes de Grothendieck et restent mystérieuses.

   \end{enumerate}
En résumant, le point (a) de la preuve du Théorème \ref{thm positif} est le calcul de la matrice de périodes associée à \eqref{faltings}. Ce calcul procède comme suit : on décrit l'action de $\varphi$ et la filtration sur ces pédiodes (point (1) ci-dessus), puis on montre que cette description  caractérise les périodes en question (point (2)), enfin on construit  de telles périodes (point (3)).
 \end{rem}
 
 \begin{exemple}\label{exem periode} (Un calcul de périodes.)
Soit $M$ un motif comme dans le Théorème \ref{thm positif} pour lequel on veut montrer la relation \eqref{formule hilb}. Cela passe par le calcul de la matrice de périodes de \eqref{faltings}. Ce calcul dépend de l'entier $n$ (défini dans la Proposition \ref{prop hilb}) et de  l'algèbre  $\End(V_{B,p})$ des endomorphismes de $V_{B,p} $ comme représentation galoisienne. 

Supposons $n=1$ : on doit alors montrer   que $q_Z\otimes \Q_p$ et $q_B\otimes \Q_p$ ne sont pas isomorphes, voir la Proposition \ref{prop hilb}. Supposons également que  $\End(V_{B,p})$ soit le corps $\Q_{p^2}$, l'unique extension non ramifiée de degré $2$ de $\Q_p$. Comme $B_{\cris}$ contient toutes les extensions non ramifiées, on peut alors écrire \eqref{faltings} comme
\begin{align}\label{nonso}   (V_{B,p}  \otimes_{\Q_p} \Q_{p^2} )  \otimes_{\Q_{p^2}}  B_{\cris}  =  (V_{Z,p}  \otimes_{\Q_p} \Q_{p^2}) \otimes_{\Q_{p^2}} B_{\cris} .\end{align}
L'action de $\End(V_{B,p})=\Q_{p^2}$ sur $V_{B,p}$ décompose $ V_{B,p}  \otimes_{\Q_p} \Q_{p^2} $ en deux droites échangées par le groupe de Galois $\Gal (\Q_{p^2}/\Q_p). $ En particulier on peut choisir deux vecteurs $v_B$ et $w_B$ échangés par le groupe de Galois et appartenant à   ces droites.

 L'algèbre $\Q_{p^2}$ agit également sur $V_{Z,p}$, grâce à l'équivalence de catégories de Faltings, voir aussi le point (2) de la Remarque \ref{rem faltings}. On peut alors construire $v_Z$ et $w_Z$ de façon analogue. De plus, comme \eqref{faltings} est compatible à cette action, il existe deux périodes $\alpha,\beta\in B_{\cris}$ telles que
\[\alpha v_B =  v_Z \hspace{0.5cm} \textrm{et}  \hspace{0.5cm} \beta v_B =  v_Z .\]
Je prétends qu'elles satisfont aux relations 
\begin{align}\label{rel1}\varphi({\alpha})=\beta \hspace{0.5cm} \textrm{et}  \hspace{0.5cm}\varphi({\beta})=  \alpha  \end{align}
ainsi que
\begin{align}\label{rel2}\alpha\in\Fil^1-\Fil^2  \hspace{0.5cm} \textrm{et}  \hspace{0.5cm} \beta\in\Fil^{-1}-\Fil^0.   \end{align}
En effet $\Gal (\Q_{p^2}/\Q_p)$ est engendré par le Frobenius et d'autre part l'action du Frobenius sur $V_{B,p}$ et $V_{Z,p}$ est triviale par définition.
Quant à la relation sur la filtration, on la déduit du fait que les droites propres que l'on a construites doivent être isotropes, or la droite définie par la filtration sur  la réalisation cristalline de $M$ doit aussi l'être, en particulier elle doit coïncider avec une de ces droites propres (après extension des scalaires). 

Les relations \eqref{rel1} et \eqref{rel2} sont un exemple du principe (1) expliqué dans la Remarque \ref{rem faltings}.

\

Considérons maintenant le $\Q_{p^2}$ espace vectoriel
\[P=\{\gamma \in B_{\cris}, \hspace{0.5cm} \varphi^2({\gamma})=\gamma, \hspace{0.5cm}   
\gamma\in\Fil^1-\Fil^2  \hspace{0.5cm} \textrm{et}  \hspace{0.5cm} \varphi({\gamma})\in\Fil^{-1}-\Fil^0  \}  .\]
Je prétends qu'il a dimension $1$, autrement dit que n'importe quel élément de $P$ est en fait une période de $M$ construite ci-dessus. Ceci est un exemple du principe (2) expliqué dans la Remarque \ref{rem faltings}.

Pour le montrer remarquons d'abord qu'une période  $\alpha$ de $M$ est inversible dans $B_{\cris}$. On peut alors considérer l'espace $P/\alpha\subset B_{\cris}$, il correspondra à
\[\alpha^{-1}P=\{\lambda \in B_{\cris}, \hspace{0.5cm} \varphi^2({\lambda})=\lambda, \hspace{0.5cm}   
\lambda\in\Fil^0-\Fil^{-1}  \hspace{0.5cm} \textrm{et}  \hspace{0.5cm} \varphi({\lambda})\in\Fil^{0}-\Fil^{-1}  \}  .\]
Or cet espace est $\Q_{p^2}\subset B_{\cris}$ par \cite[Théorème 5.3.7]{Fontp}.

\

Pour conclure construisons  une période $t$ avec les propriétés
\[t,\varphi(t)\in B^*_{\cris}, \hspace{0.5cm} \varphi^2(t)=p\cdot t  , \hspace{0.5cm}   
t\in\Fil^1-\Fil^2  \hspace{0.5cm} \textrm{et}  \hspace{0.5cm} \varphi(t)\in\Fil^{0}-\Fil^1   .\]
De ces propriétés on déduit que $\alpha=t/\varphi(t)\in P$, autrement dit $\alpha$ est une période du motif $M$,  et $\alpha\cdot\varphi(\alpha)=1/p$. Cette dernière relation implique   que $q_Z\otimes \Q_p$ et $q_B\otimes \Q_p$ ne sont pas isomorphes par un calcul élémentaire qui est analogue à celui de l'Exemple \ref{exem elem}.

La construction de $t$ suit le principe (3) expliqué dans la Remarque \ref{rem faltings}. Considérons le $\varphi$-module filtré   $N=\Q_p^2$ muni du Frobenius
\[\varphi= \begin{pmatrix}
		0 & 1/p\\
		1 & 0
	\end{pmatrix} \]
et de la filtration 
\[ \Fil^{-1}=N, \Fil^0 = \Q_p\cdot e_2, \Fil^1=0.\]
On vérifie   que c'est un $\varphi$-module admissible,  donc il existe une représentation de Galois $V$ qui lui correspond par l'équivalence de catégorie de Faltings et qui donne une comparaison $V\otimes B_{\cris} = N\otimes B_{\cris} = B_{\cris}^2.$ Puisque cette identification est compatible à toutes les structures, on voit que les vecteurs de $V\subset B_{\cris}^2$ sont exactement de la forme $(t,\varphi(t))$ où $t$ satisfait aux propriétés voulues.

\end{exemple}

\begin{rem} (Généralisations possibles.)
La partie $p$-adique de l'argument présenté se généralise aux motifs de dimension plus grande : les principes généraux expliqués dans la Remarque \ref{rem faltings} restent valables, les calculs de l'Exemple \ref{exem periode} deviennent plus compliqués mais peuvent être traités. Dans un travail avec Adriano Marmora \cite{AM} nous avons pu en déduire une généralisation de la formule \eqref{formule hilb}. En suivant l'argument de la Proposition \ref{prop hilb} cela donne la positivité du symbole de Hilbert à l'infini $\varepsilon_\infty(q_Z)=+1.$ Malheureusement cette information ne suffit pas à déduire que $q_Z$ est défini positive, c'est le point crucial où l'on utilisait l'hypothèse de dimension $2$. 

Pour passer à la dimension supérieure il faudrait trouver un invariant défini en toute place, tel que la place à l'infini soit contrôlée par toutes les places finies et d'autre part tel que  l'invariant à l'infini détermine toute la signature. Une tentative pourrait passer par la cohomologie galoisienne : les $k$-formes quadratiques non dégénérées et de rang donné sont en bijection avec  $H^1(k,O)$ où $O$ est un $k$-groupe orthogonal de rang convenable. L'application  
\[H^1(\Q,O)\longrightarrow \bigoplus_\nu H^1(\Q_\nu,O).\]
  est injective, c'est le théorème d'Hasse--Minkowski. Son défaut de surjectivité est justement contrôlé par la formule du produit des symboles de Hilbert.  Le problème déjà soulevé se reformule alors ainsi : l'application
\[H^1(\Q,O)\longrightarrow \bigoplus_{\nu\neq \infty} H^1(\Q_\nu,O)\]
n'est plus injective.

On peut alors essayer d'exploiter plus d'informations géométriques de notre situation et faire surgir des groupes plus petits. Par exemple les motifs qui apparaissent dans le problème sont munis non seulement d'une forme quadratique mais aussi de l'action d'un corps CM. Ajouter cette donnée au problème correspond à étudier la cohomologie galoisienne d'un  tore maximal $T$ du groupe orthogonal $O$.
On dispose encore d'un principe local-global : l'application
\[H^1(\Q,T)\longrightarrow \bigoplus_\nu H^1(\Q_\nu,T).\]
est injective. Le point crucial serait alors d'avoir l'injectivité aussi de l'application
\[H^1(\Q,T)\longrightarrow \bigoplus_{\nu\neq \infty} H^1(\Q_\nu,T).\]
Malheureusement elle n'est pas injective   : on peut calculer  son noyau  à l'aide de la suite exacte de Poitou--Tate.
   \end{rem}

    \section{Périodes $p$-adiques à la André}\label{section andre}

Dans cette section nous présentons un travail en collaboration avec Dragos Fratila. Les motivations sont d'origine géométrique - l'étude des classes algébriques en caractéristique $p$ - mais le résultat final est plutôt arithmétique : on construit une algèbre de périodes $p$-adiques ainsi qu'un cadre tannakien pour l'étudier.

Des telles périodes devraient avoir des analogies avec les périodes complexes que l'on a rencontré dans  l'Exemple \ref{exemple periode}. Les périodes $p$-adiques de Fontaine possèdent des propriétés cohomologiques analogues à celles des périodes complexes et même plus fortes (Remarque \ref{rem faltings}). Par contre les propriétés arithmétiques des périodes complexes, comme leur transcendence ou leur  lien avec les fonctions spéciales,  n'ont pas de bon analogue dans les périodes $p$-adiques de Fontaine \cite{Betti}.  

Pendant que notre travail avançait nous avons découvert qu'André avait tissé des liens similaires \cite{Andre1,Andrep}. Son travail nous a été utile pour raffiner notre étude et notamment pour formuler la condition de ramification (Définition \ref{def ram}).

\subsection*{Motivation} Le point de départ vient d'une remarque de Tate : la conjecture de Tate prédit l'existence de classes algébriques mais elle ne prédit pas quelle classe cohomologique est algébrique  \cite[Aside 6.5]{MilneTate}. Une façon d'interpréter cette remarque est que l'on a une description du $\Q_\ell$-espace vectoriel engendré par les classes algébrique mais on n'a pas de description du $\Q$-espace vectoriel engendré par ces dernières. C'est un point délicat qui est présent dès le travail de Tate sur la conjecture de Tate pour les diviseurs sur les variétés abéliennes sur un corps fini \cite{Tate}, et plus récemment dans le travail de Charles sur la conjecture de Tate pour les diviseurs sur les surfaces K3 \cite{Charles}.

Un exemple élémentaire qui illustre cette subtilité est le suivant : il existe des $\Q_\ell$-droites dans la  cohomologie $\ell$-adique d'une variété $X$, disons définie sur un corps fini, qui sont Galois invariantes et pourtant elles ne contiennent pas d'élément du $\Q$-espace vectoriel $\Im \cl_X$. Pour construire de tels exemples prenons  $X$   une surface abélienne, ou une K3, et fixons deux classes de diviseurs $\alpha$ et $\beta$ linéairement indépendantes.  Prenons maintenant une constante $c \in \Q_\ell-\Q$. Alors la droite engendrée par $\alpha+c \beta$ convient. Pour le montrer on peut utiliser le produit d'intersection et le fait qu'il est défini à coefficients rationnels.

\subsection*{Un échec : cas $\ell$-adique} Faute de savoir décrire le $\Q$-espace vectoriel des classes algébriques, un premier pas est de le comparer   à un autre $\Q$-espace vectoriel. C'est notamment ce que l'on a fait dans l'Exemple \ref{K3 fermat} et la remarque qui le suit.

Prenons une variété $X_p$ définie sur $\overline{\mathbb{F}}_p$ et supposons qu'elle se relève à une variété $X$ définie sur $\overline{\Q}$. Au moyen d'un plongement $\sigma : \overline{\Q} \hookrightarrow \C$ on dispose d'une identification
\begin{align}\label{compartin}H^*_B(X(\C),\Q)\otimes \Q_\ell = H^*_\ell(X_p) \end{align}
induite par le théorème de comparaison d'Artin et le changement de base propre et lisse. Sous cette identification on peut étudier la position du $\Q$-espace vectoriel $Z_p$ des classes algébriques sur $X_p$ par rapport à $H^*_B(X(\C),\Q)$. Le premier fait que l'on remarque est que l'intersection $Z_p \cap H^*_B(X(\C),\Q)$ contient le $\Q$-espace vectoriel $Z_0$ des classes algébriques sur $X$. Inspiré par les différentes versions de la conjectures des périodes de Grothendieck on peut se demander si cette inclusion est en fait une égalité. 

Après avoir montré que cette question a réponse affirmative dans certains cas, nous avons compris que la réponse est négative en général. Les cas affirmatifs sont les surfaces à rang de Picard maximal - par une méthode similaire celle présentée dans l'Exemple \ref{K3 fermat} - et les variétés abéliennes CM \cite[\S 10]{AF}. Il est possible de construire des contre-exemples avec le carré d'une courbe elliptique non CM. L'argument suit en fait la technique qui a permis à André de montrer que l'analogue de la conjecture des périodes de Grothendieck est fausse pour le théorème de comparaison $p$-adique. En effet \eqref{compartin} dépend du choix de $\sigma$. On peut faire varier $\sigma$ en utilisant le groupe de Galois absolu du corps  de nombre sur lequel $X$ est défini. Dans certains cas on sait  que l'action de ce groupe de Galois sur $H^*_\ell(X_p)$ est hautement non triviale \cite{SerreGal}, ce qui permet faire varier le $\Q$-espace vectoriel $H^*_B(X(\C),\Q)$ et notamment de le faire rencontrer $Z_p $ de façon inattendue.

\subsection*{Cas $p$-adique}
Fixons une fois pour toutes un plongement $\overline{\Q} \subset \overline{\Q}_p$. Prenons comme auparavant une variété $X$ définie sur $\overline{\Q}$ à bonne réduction et  notons $X_p$ sa réduction. On dispose de la comparaison entre cohomologie de de Rham et cohomologie cristalline
\[H^*_{\dR}(X ,\overline{\Q})\otimes \overline{\Q}_p = H^*_{\cris}(X_p,\overline{\Q}_p) \]
due à Berthelot.

Notons  par $Z_p $ le $\Q$-espace vectoriel  des classes algébriques sur $X_p$ et par $Z_0 $ le $\Q$-espace vectoriel  des classes algébriques sur $X.$ On a comme dans le cas $\ell$-adique l'inclusion 
\[Z_0 \subset Z_p \cap H^*_{\dR}(X ,\overline{\Q}) \]   et on peut encore une fois se demander si cette inclusion est en fait une égalité. 
\begin{conj}\label{conj pGPCw}($p$GPCw :  Analogue $p$-adique de la version faible de la conjecture  des périodes de Grothendieck.)

\noindent Est-ce que l'inclusion $Z_0 \subset Z_p \cap H^*_{\dR}(X ,\overline{\Q})$ est une égalité ?
\end{conj}

Pour rendre cette question raisonnable il est nécessaire d'imposer une condition de ramification que l'on discutera plus tard et que l'on ignore pour l'instant  (Définition \ref{def ram}). Cette question apparaît comme l'analogue $p$-adique de la version faible de la conjecture des périodes de Grothendieck  (appelée parfois conjecture de de Rham--Betti \cite[\S 7]{Andmot}). La remarque ci dessous fait le lien entre cette conjecture et  différentes conjectures classiques sur les cycles algébriques.

Tout comme son pendant classique, cette conjecture prédit de la transcendence. En effet elle prédit qu'une classe algébrique en caractéristique $p$ qui n'est pas relevable à la caractéristique zéro ne peut pas être dans $H^*_{\dR}(X ,\overline{\Q}) $, autrement dit, au moins une de ses coordonnées par rapport à une base de $H^*_{\dR}(X ,\overline{\Q}) $ doit être transcendante. 
 La version forte de la conjecture de Grothendieck $p$-adique prédira de façon précise le degré de transcendance de toutes ces coordonnées (Conjecture \ref{CPGf}).
\begin{rem}\label{CPGw} ($p$GPCw vs conjectures classiques.)
Comparons maintenant la question qui a été soulevée au paragraphe précédent, notée  ($p$GPCw), avec trois conjectures classiques que nous rappelons de façon informelle  (voir \cite[\S 7]{Andmot} pour plus de détails).
Ces trois conjectures sont la conjecture de Hodge (HC), la version faible de la conjecture des périodes de Grothendieck (GPCw) et la conjecture de Hodge variationnelle   $p$-adique de Fontaine et Messing   ($p$HC).    
	
\begin{itemize}
	\item[(HC)] Une classe rationnelle en cohomologie singulière est algébrique si et seulement si elle appartient au bon degré de la filtration de de Rham.
		\item[(GPCw)] Une classe rationnelle en cohomologie de de Rham est algébrique si et seulement si elle est rationnelle pour la  cohomologie singulière
			\item[($p$HC)] Une classe algébrique en cohomologie cristalline se relève à la caractéristique zéro si et seulement si elle appartient au bon degré de la filtration de de Rham.
\end{itemize}
De façon informelle, on peut voir  ($p$GPCw) comme    \og le produit fibré de ($p$HC) et (GPCw) au-dessus de (HC) \fg.
\[\begin{matrix}
(p\textrm{GPCw}) & \rightarrow & (\textrm{GPCw}) \\
\downarrow &  & \downarrow \\
($p$ \textrm{HC}) & \rightarrow & (\textrm{HC})
\end{matrix}\]
La partie droite du diagramme concerne la caractéristique zéro, celle de gauche la caractéristique mixte. La   conjectures du  bas  comparent une structure rationnelle et une filtration, celles  du haut comparent deux structures rationnelles.
\end{rem}
\begin{defin}\label{defperp}(Périodes $p$-adiques à la André.)
Soit $M\in \Mot(\overline{\Q} ) $ un motif homologique\footnote{Tout comme dans le cas classique on travaillera qu'avec des motifs vérifiant $\hom=\num$, par exemple les motifs issues de produits de courbes elliptiques : ceci est nécessaire pour avoir des catégories tannakiennes.}   à bonne réduction et soit $M_p \in \Mot(\overline{\mathbb{F}}_p) $  sa réduction. Notons leurs classes algébriques par
\[Z_0(M)=\Hom_{\Mot(\overline{\Q})}(\mathbbm{1},M) \hspace{0.5cm} \textrm{et} \hspace{0.5cm}Z_p(M)=\Hom_{\Mot(\overline{\mathbb{F}}_p)}(\mathbbm{1},M_p).\]
Considérons le théorème de comparaison de Berthelot
\begin{align}\label{compbert}R_{\dR}(M)\otimes_{\overline{\Q}} \overline{\Q}_p = R_{\cris}(M_p).\end{align}

Pour tout choix de base $\mathcal{B}$ de $Z_p(M)$ et $\mathcal{B}'$ de $R_{\dR}(M)$ définissons $\Mat_{\mathcal{B},\mathcal{B}'}(M)$ comme la matrice ayant comme vecteurs colonnes les coordonnées de $\mathcal{B}$ par rapport à  $\mathcal{B}'$. Nous appelons les coefficients de cette matrice les périodes $p$-adique d'André de $M$  et définissons $\mathcal{P}_p(M)\subset  \overline{\Q}_p$ comme la $\overline{\Q}$-algèbre engendrée par ces périodes.
\end{defin}
\begin{rem}\label{rem mattia}(Périodes classiques vs périodes $p$-adiques à la André.)
\begin{enumerate}
\item  Si un élément de  $\mathcal{B}$ est une classe algébrique qui n'est pas relevable à la caractéristique zéro   au moins une de ses  périodes devrait être transcendante par la Conjecture  \ref{conj pGPCw}.
\item La matrice  $\Mat_{\mathcal{B},\mathcal{B}'}(M)$ dépend bien du choix des bases $\mathcal{B}$ et  $\mathcal{B}'$, par contre l'algèbre $\mathcal{P}_p(M)$ n'en dépend pas.
\item Pour que l'espace $Z_p(M)$ ne soit pas réduit à zéro il faut que $M$ contienne des facteurs directs de poids zéro. Il faut typiquement imaginer $M=\mathfrak{h}^{2n}(X)(n)$ pour une variété $X$ à bonne réduction. De plus, pour avoir des périodes intéressantes, il faut que $M$ admette des classes algébriques modulo $p$ qui ne sont pas relevables, sinon toutes les périodes $p$-adiques seraient algébriques.  
\item  Contrairement au cas classique, la matrice de périodes $p$-adique n'est pas carrée, en effet   l'inégalité\footnote{Point technique : cette inégalité a encore besoin de l'hypothèse $\hom=\num$. On utilise le fait que   l'équivalence numérique commute à l'extension des scalaires, voir la Conjecture \ref{conj nonso} et la remarque qui la suit.} $\# \mathcal{B}' \leq \#\mathcal{B}$ est stricte en général.
\item Il est possible de définir les périodes $p$-adiques pour les motifs mixtes. Il faut dans ce cas considérer uniquement les motifs mixtes vérifiant \eqref{compbert} - cette relation n'est automatique pour les variétés ouverte.
\end{enumerate}
\end{rem}  

\begin{exemple}\label{exem periodes p}
\begin{enumerate}
\item (Courbes elliptiques CM et valeurs Gamma.) Considérons $E$ une courbe elliptique CM. Les périodes complexes de son $\mathfrak{h}^1(E)$ sont un produit de certaines valeurs spéciales de la fonction gamma $\Gamma_\C$ en certains rationnels explicites dépendant uniquement du corps CM. Il n'y a pas de période $p$-adique associée à $\mathfrak{h}^1(E)$, puisqu'il n'y a pas de classe algébrique dans le $\mathfrak{h}^1(E)$, par contre on peut considérer le motif  $M=\mathfrak{h}^1(E) \otimes \mathfrak{h}^1(E)^{\vee} $. Ses classes algébriques correspondent aux endomorphismes de $E$. Pour avoir des périodes $p$-adiques intéressantes considérons un premier $p$ à réduction supersingulière, autrement tous les endomorphismes se relèveraient et toutes les périodes seraient algébriques. Il s'agit de décrire l'action de ces endomorphismes par rapport à une base de $H_\dR^1(E,\overline{\Q})$. Ce genre de calculs a  été traités par Coleman et Ogus \cite{Colem,Ogus}. On y voit apparaître des produits de valeurs spéciales de la fonction gamma $p$-adique  $\Gamma_p$ en des rationnels.

\item (Motifs de Kummer et logarithme.) Considérons le motif de Tate mixte de type Kummer $K_a=\mathfrak{h}^1(\mathbb{G}_m,\{1,a\})^{\vee}$, avec $a\in \Q$. Ce motif s'insère dans une suite exacte
\[0\longrightarrow \mathbbm{1}(+1) \longrightarrow  K_a   \longrightarrow   \mathbbm{1}    \longrightarrow   0\]
qui est non scindée pour $a\neq 0,1,-1$. En particulier ce motif n'a pas de classe algébrique non nulle, qui est la raison d'avoir considéré $K_a $ et non pas son dual $ \mathfrak{h}^1(\mathbb{G}_m,\{1,a\}) $.
Sa matrice de périodes complexes est
\[  \begin{pmatrix}
		2\pi i & \log(a)\\
		0 & 1
	\end{pmatrix}.\]

  Fixons un nombre premier $p$.  Pour   $a\not\equiv 0,1 \hspace{0.2cm}[p]$,  le motif $K_a$ a bonne réduction. De plus les motifs de Tate sur un corps fini forment une catégorie semisimple, en particulier la suite exacte ci-dessus se scinde modulo $p$. On en déduit que le motif possède une classe algébrique non nulle modulo $p$ qui est donc non relevable. Sa matrice de périodes $p$-adiques est 
  \[  \begin{pmatrix}
		  \log_p(a)\\
		  1
	\end{pmatrix} \]
		dont l'analogie avec son pendant complexe est encore une fois frappante. Cette matrice s'obtient à partir du calcul de la matrice du Frobenius agissant sur $R_\dR(K_a)$ qui est dû à Deligne \cite[\S 2.9]{Delignelog}. On y voit apparaître le  $\log_p(a^{1-p})$ et on trouve curieux que le passage de la matrice de Frobenius à la matrice de périodes corrige cet exposant.  (La correction de $\log_p(a^{1-p})$ à $\log_p(a)$  aurait pu s'obtenir en changeant de base, or pour ces motifs on dispose de bases canoniques et toutes les matrices décrites ci-dessus utilisent uniquement ces bases).
\item  (Fonctions hypergéométriques.) Soient $M$ et $N$ deux motifs non isomorphes mais dont les réductions modulo $p$ le sont. Alors le motif $M\otimes N^{\vee}$ a une classe algébrique modulo $p$ non relevable qui est justement associée à cet isomorphisme. Ses périodes $p$-adiques sont les coordonées de $R_\dR(N)$ par rapport à $R_\dR(M)$. 

Par exemple on peut considérer $M=\mathfrak{h}^1(E)$ et $N=\mathfrak{h}^1(E')$ où $E$ et $E'$ sont deux relèvements non isogènes d'une courbe elliptique ordinaire sur un corps fini. On peut notamment choisir $E$ comme le relèvement canonique de Serre--Tate et $E'$ comme une courbe elliptique non CM. Les périodes $p$-adiques qui apparaissent dans ce cas là sont décrites par Katz \cite{Katz}. On y voit notamment apparaître des valeurs spéciales de fonctions hypergéométriques.

\item (Matrice du Frobenius.) Soit $f\in \End(M_p)$ un endomorphisme modulo $p$. On peut considérer la matrice de son action par rapport à une base de $R_\dR(M)$. Ses coefficients pourront s'interpréter comme périodes $p$-adiques à la André en regardant $f$ comme une classe algébrique modulo $p$ du motif $M\otimes M^{\vee}$. 
D'intérêt particulier est le cas où $f$ est le Frobenius : certains auteurs \cite{Furusho,Brownp} ont définis les périodes $p$-adique associés à $M$ comme ses coefficients. 

Le point de vue des périodes $p$-adiques à la André est meilleure pour plusieurs raisons. Entre autres,  il donne des   bornes plus fine à la transcendance ainsi qu'une interprétation motivique de certaines relations naturelles, comme celles provenant du polynôme caractéristique du Frobenius, voir \cite[Remark 9.7]{AF}.
 \end{enumerate}
 \end{exemple}  
 \subsection*{Transcendance} Comme expliqué dans la Remarque \ref{rem mattia}, la Conjecture \ref{conj pGPCw}    prédit la transcendance de certaines périodes $p$-adiques. Le prochain but est de donner une borne au degré de transcendance de ces périodes (Théorème \ref{thm AF1}) ainsi qu'une conjecture  qui prédira ce degré (Conjecture \ref{CPGf}).
On montrera que cette dernière conjecture implique en fait la Conjecture \ref{conj pGPCw}, voir la Proposition \ref{prop mattia}.

\

Gardons les notations de la Définition \ref{defperp}. 
Considérons les catégories tannakiennes $\langle M\rangle $ et $\langle M_p\rangle $ engendrées par $M$ et $M_p$ ainsi que les groupes tannakiens $G_\dR(M)$ et $G_{\cris}(M_p)$ associés aux foncteurs fibres $R_\dR$ et $R_{\cris}$. 

Dans le cadre des périodes classiques,  Grothendieck démontre que leur degré de transcendence est borné par la dimension de $G_\dR(M)$. Notre résultat principal en est l'analogue $p$-adique.
\begin{thm}\label{thm AF1}
Le degré de transcendance des périodes $p$-adiques vérifie l'inégalité
\[  \degtr \mathcal{P}_p(M)\leq \dim G_\dR(M) - \dim G_{\cris}(M_p).\]
\end{thm}
\begin{rem}\label{troilo}(Transcendance classique vs transcendance $p$-adique.) 
\begin{enumerate}
\item
Cette inégalité pourrait sembler plus forte que celle du cas complexe mais la matrice rectangulaire des périodes $p$-adiques est en général plus petite que celle des périodes complexes. Elles ont la même taille uniquement dans le cas de réduction supersingulière ce qui revient à $G_{\cris}(M_p)=\{1\}.$
\item Dans le cadre complexe, le point crucial  est d'interpréter les périodes comme les coordonnées d'un $\C$-point du foncteur $T(M)=\Isom_{\langle M\rangle }^{\otimes}(R_B,R_\dR)$. Par la théorie tannakienne ce foncteur est représentable   par une variété affine sur $\overline{\Q}$. L'action naturelle de $G_\dR(M)=\Aut_{\langle M\rangle }^{\otimes}( R_\dR)$ sur $T(M)$ rend cette variété un torseur.

Dans le cas $p$-adique   le foncteur $Z_p$ n'est pas un foncteur fibre, pour des questions de dimension, ce qui le fait sortir du cadre tannakien. C'est tout de même un foncteur lax-monoïdal (ce qui revient à dire que le produit de classes algébriques est une classe algébrique). 

Le foncteur $\Isom_{\langle M\rangle }^{\otimes}(Z_p,R_\dR)$  est   vide en général, encore pour des raisons de dimension. On peut en revanche considérer les transformations naturelles tensorielles ou les plongements.  
\end{enumerate}
\end{rem}
\begin{thm}\label{thm AF2}
L'inclusion de foncteurs $\Emb_{\langle M\rangle }^{\otimes}(Z_p,R_\dR) \subseteq  \Nat_{\langle M\rangle }^{\otimes}(Z_p,R_\dR)$ est une égalité. Ces foncteurs sont représentables par une variété $H(M)$ affine sur $\overline{\Q}$.  L'action naturelle de $G_\dR(M)$ sur $H(M)$ est transitive. De plus on a un isomorphisme $H(M)_{\overline{\Q}_p}    
=G_\dR(M)_{\overline{\Q}_p}/ G_{\cris}(M_p).$
\end{thm}

\begin{rem} (Théorème \ref{thm AF1} implique Théorème \ref{thm AF2}.)
En analogie avec le cas classique on peut  interpréter
\begin{align}\label{bertpt}Z_p \otimes \overline{\Q}_p \hookrightarrow R_{\cris}=R_\dR  \otimes \overline{\Q}_p\end{align}
comme un  $ \overline{\Q}_p$-point de  $H(M)$. L'évaluation en ce point donne un morphisme d'algèbres
\begin{align}\label{eval} \eval  : \mathcal{O}(H(M)) \longrightarrow\overline{\Q}_p.\end{align}
Par construction, l'image de ce morphisme est l'algèbre $\mathcal{P}_p(\langle M\rangle )$ engendrée par toutes les périodes $p$-adiques de tous les motifs appartenant à la catégorie $\langle M\rangle $.
Contrairement au cas classique, ces périodes contiennent strictement celles de $M$, en général. Cela vient du fait que l'inclusion $Z_p(M)^{\otimes n}  \subset  Z_p(M^{\otimes n})$ est stricte en général : c'est le défaut d'une formule de Künneth pour les classes algébriques.

Le Théorème \ref{thm AF2} et les relations
\begin{align}   \mathcal{P}_p(M)   \subset  \mathcal{P}_p(\langle M\rangle )=   \Im \eval   \subset \overline{\Q}_p.\end{align}
prouvent le Théorème \ref{thm AF1} et même l'inégalité plus forte
\begin{align}   \degtr\mathcal{P}_p(M)   \leq \degtr\mathcal{P}_p(\langle M\rangle ) \leq   \dim H(M)=G_\dR(M) - G_{\cris}(M_p) .\end{align}
\end{rem}
La définition suivante est inspirée de travaux d'André \cite{Andre1,Andrep}.
\begin{defin}\label{def ram} (Condition de ramification.) 
On dit  qu'un motif $N$ est CM si  $\End(N)$ est un corps de nombres tel que $\dim_\Q\End(N)=\dim R_\dR(N)$. 

On dit que le nombre premier $p$  ne ramifie pas dans $\langle M\rangle $ si, pour tout $N$ dans $\langle M\rangle $ qui est CM, le nombre premier $p$ ne ramifie pas dans le corps de nombre $\End(N)$.
\end{defin}
\begin{exemple}
Soit $A$ une variété abélienne. Dans le cas où $N=\mathfrak{h}^1(A)$, imposer  que $\End(N)$ soit un corps de nombres  tel que $\dim_\Q\End(N)=\dim R_\dR(N)$ revient à demander que $A$ soit simple et CM. Dans ce cas demander que $p$ ne ramifie pas dans $\End(N)$ correspond à demander que $p$ ne ramifie pas dans son corps CM.
\end{exemple}
\begin{prop}\label{prop mattia}
Pour un motif $M'$ fixé, les nombres premiers qui ramifient dans  $\langle M'\rangle $ sont en nombre fini.
\end{prop}
\begin{conj}\label{CPGf}($p$-GPCs : Analogue $p$-adique de la version forte de la conjecture  des périodes de Grothendieck.)

Si $p$ ne ramifie pas dans $\langle M\rangle $ alors l'application d'évaluation \eqref{eval} est injective. De façon équivalente, l'espace homogène $H(M)$ est connexe et le $ \overline{\Q}_p$-point de  $H(M)$ induit par \eqref{bertpt} vit au-dessus du point générique de  $H(M)$ (ou encore l'inégalité 
$\degtr\mathcal{P}_p(\langle M\rangle ) \leq   \dim H(M)=G_\dR(M) - G_{\cris}(M_p)$ est en fait une égalité.)
\end{conj}
\begin{prop}
Si $M$ vérifie la Conjecture \ref{CPGf} alors pour tout $N$ dans $ \langle M\rangle $ on a une \[Z_0(N)=Z_p(N)\cap R_\dR(N), \]  c'est-à-dire la Conjecture \ref{conj pGPCw}  ($p$GPCw)   a réponse affirmative pour $N$.
\end{prop}
\begin{rem}
\begin{enumerate}
\item  Les preuves des résultats de cette section utilisent des techniques tannakiennes. Comme déjà mentionné, notamment dans la Remarque \ref{troilo}(2), on ne peut pas utiliser les résultats classiques tel quels mais il faut plutôt adapter leurs preuves.
\item   La condition de ramification (Définition \ref{def ram}) est inspirée d'une condition qu'André a imposée dans l'étude de ce genre de questions pour les variétés abéliennes à réduction supersingulière. Il avait remarqué que ces variétés pouvaient avoir des  périodes $p$-adiques vérifiant des relations algébriques non motiviques. Comme mentionné dans l'Exemple \ref{exem periodes p}(1), les périodes qui apparaissent pour de tels  motifs sont liées aux valeurs spéciales de la fonction Gamma $p$-adique. Ces dernières se trouvent être  plus souvent algébriques que leurs analogues complexes.
\item Assez peu est connu   sur la conjecture classique des périodes de Grothendieck. La version forte a été démontrée pour les courbes elliptiques CM par Chudnovsky \cite{Chud}. Le cas particulier de la courbe elliptique de Fermat   implique notamment la transcendance de $\Gamma_\C(1/3)$. La version faible (voir (GPCw) de la Remarque \ref{CPGw}) a été montrée pour les diviseurs sur les variétés abéliennes et sur les surfaces K3 par Bost et Charles \cite{BostCha}, en utilisant entre autre le théorème du sous-groupe analytique de Wüstholtz \cite{Wus}. Ce sont des résultats difficiles et on peut s'attendre à ce que leurs analogues $p$-adiques le soient aussi.

Le seul cas où la Conjecture \ref{CPGf} est vérifiée est pour le motif de Kummer $K_a$ de l'Exemple  \ref{exem periodes p}(2). Cela suit de la transcendence des valeurs spéciales du logarithme $p$-adique \cite{Bert}. Le premier cas ouvert intéressant serait celui des courbes elliptiques à réduction supersingulière, ce qui donnerait notamment la transcendence de $\Gamma_p(1/3)$, pour $p\equiv 2 [3]$. La version faible semble aussi difficile. Une petit résultat dans cette direction a été donné dans le cas des courbes elliptique non CM à réduction supersingulière \cite[Proposition 3.5]{AF}.
\end{enumerate}
\end{rem}

    \section{Motifs des schémas en groupes commutatifs}\label{AHP}
    Cette section résume deux travaux en collaboration avec Stephan Enright-Ward, Annette Huber et Simon Pepin Lehalleur \cite{AEH,AHP}. Ils portent sur l'anneau de Chow et le motif d'un schéma en groupes commutatifs et généralisent les théorèmes de Beauville \cite{Beau} et Deninger--Murre \cite{DeMu} qui traitent le cas des schémas abéliens. Nous expliquons quelles sont les subtilités qui apparaissent quand on quitte le cadre des schémas en groupes projectifs. Les motifs de Voevodsky deviennent essentiels, non seulement leur existence mais aussi la nature de leur construction : c'est un exemple du principe expliqué au \S \ref{complexes motiviques}.
    
    Dans ce qui suit $S$ est une variété de type fini et lisse sur un corps $k$ qui  jouera le rôle d'une base fixée. Tout $S$-schéma en groupe que l'on considérera sera lisse et de type fini à fibres connexes. Certains énoncés sont valables dans des meilleures généralités. Pour un $S$-schéma lisse $f:X\rightarrow S$, $\CH(X)$ indique l'anneau de Chow de l'espace total $X$. On continue à travailler avec les coefficients rationnels.

    \begin{thm}\label{thm beauville}(Beauville \cite{Beau}, Deninger--Murre \cite{DeMu})
    Soient $A$ un $S$-schéma abélien de dimension relative $g$ et $n_A:A \rightarrow A$ le morphisme de multiplication par $n$. Alors on a une décomposition
    \[\CH^i(A)=\bigoplus_{r=i}^{g+i}  \CH_{(r)}^i(A)  \]
    où $\CH_{(r)}^i(A) = \{\alpha \in \CH^i(A), n^*_A \alpha = n^r \alpha, \forall n\in \Z \}$.
   \end{thm}
   \begin{rem}
   \begin{enumerate}
   \item
   Ce théorème a été démontré par Beauville dans le cas $S=\Spec(k)$ et Deninger--Murre dans le cas général. Les deux résultats utilisent de façon cruciale la transformée de Fourier que l'on rappelle plus loin.
   \item Quand $S=\Spec(k)$, on a l'inclusion  \begin{align}\label{BBM}\ker \cl^i_A\supseteq\bigoplus_{r\neq 2i}^{g+i}  \CH_{(r)}^i(A) \end{align}
   qui vient du fait que $n^*_A $ agit sur la cohomologie de degré $s$ comme $n^s\cdot \Id.$
\item
   La décomposition du théorème induit une bigraduation sur l'anneau $\CH_{(\bullet)}^*(A)$. Quand $S=\Spec(k)$, la nouvelle graduation scinde la filtration de Bloch--Beilinson qui est conjecturée avoir certaine propriétés par rapport à l'application classe de cycle, notamment l'inclusion \eqref{BBM} devrait être une égalité. La filtration de Bloch--Beilinson est conjecturée exister pour tous les anneaux de Chow de toutes les variétés projectives et lisses sur $k$, mais en général cette filtration ne se scinde pas.   \end{enumerate}
   \end{rem}
   \begin{proof}
   Soient $A^\vee$ le schéma abélien dual, $\mathcal{P}\in \CH^1(A \times A^\vee) $ le diviseur associé au fibré de Poincaré et $\pi_1,\pi_2$ les projections de $A \times A^\vee$ sur les deux facteurs. On définit la transformée de Fourier
   \[\mathcal{F}_A :  \CH^*(A) \longrightarrow    \CH^*(A^\vee), \hspace{0.5cm} \alpha \mapsto (\pi_2)_* (\exp(\mathcal{P}) \cdot \pi_1^*\alpha). \]
   On vérifie que c'est un isomorphisme, dont l'inverse est essentiellement $\mathcal{F}_{A^\vee}$. On en déduit 
   la décomposition
    \begin{align}\label{percal}
    \CH^i(A)=\bigoplus_{s}\{\alpha \in \CH^i(A),  \hspace{0.5cm} \mathcal{F}_A (\alpha) \in  \CH^{s}(A)\}. 
    \end{align}
   Ensuite un calcul direct permet de voir comment $n^*_A $ agit sur chaque facteur de la décomposition \eqref{percal}. On retrouve ainsi la décomposition de l'énoncé et l'identification
    $  \CH_{(r)}^i(A) = \{\alpha \in \CH^i(A),  \hspace{0.5cm} \mathcal{F}_A (\alpha) \in  \CH^{g+i-r}(A)\}.$
  \end{proof}
  
   Nous montrons la généralisation suivante.
    \begin{thm}\label{thm AHP}
        Soient $G$ un $S$-schéma en groupes commutatifs  de dimension relative $d$ et $n_G:G \rightarrow G$ le morphisme de multiplication par $n$. Alors on a une décomposition
    \[\CH^*(G)=\bigoplus_{r=0}^{2 d}  \CH^*_{(r)}(G)  \]
    où $\CH_{(r)}^*(G) = \{\alpha \in \CH^*(G), n^*_G \alpha = n^r \alpha, \forall n\in \Z\}$. 
    \end{thm}
    \begin{rem}
    \begin{enumerate}
    \item On ne sait pas définir une transformée de Fourier pour un tel $G$ : l'application $\pi_2$ n'est pas propre, donc $(\pi_2)_*$  n'existe pas, et de plus le dual d'un tel $G$ est un $1$-motif en général et non pas une variété. 
    \item Le premier cas non trivial pour les groupes non projectifs est donné par $S=\Spec(k)$ et $G$  qui admet une suite exacte
    \[0 \longrightarrow \mathbb{G}_m \longrightarrow G  \longrightarrow  A  \longrightarrow 0 \]
    où $A$ est une variété abélienne. Dans ce cas le Théorème \ref{thm AHP} est  facile pour $\mathbb{G}_m$ et connu pour $A$ mais on ne peut pas le déduire directement pour $G$. Le problème est que cet énoncé se comporte bien pour les sommes directes mais mal pour les suites exactes. L'énoncé qui suivra sera plus adapté à ce genre de dévissage.
    \end{enumerate}
    \end{rem}
    \begin{defin}\label{def AHP}
Soient $\Sm/S$   la catégorie des $S$-schémas lisses de type fini et $\PSh(S)$ la catégorie des préfaisceaux  sur $\Sm/S$ à valeur dans les $\Q$-espaces vectoriels. Soient $\Q(G)\in  \PSh(S)$ le préfaisceau qui associe à chaque $Y\in  \Sm/S$ le $\Q$-espace vectoriel ayant comme base l'ensemble $\Hom_S(Y,G)$ et   $\underline{G}\in  \PSh(S)$ celui qui associe  à chaque $Y\in  \Sm/S$ le $\Q$-espace vectoriel  $\Hom_S(Y,G)\otimes_\Z \Q.$ La loi de groupe de $G$ induit une transformation naturelle
\[s_G : \Q(G)   \longrightarrow   \underline{G}.\]
Par construction de   $\DM(S)$, la catégorie des motifs relatifs\footnote{On considère ici uniquement la version stable de cette catégorie, c'est-à-dire la catégorie obtenue après $\otimes$-inversion du motif de Lefschetz. Il y a plusieurs descriptions de la catégorie  stable $\DM(S)$. Il se trouve qu'elles sont   équivalentes sous des hypothèses assez générales qui sont notamment satisfaites pour les  bases $S$ que l'on considère. La version qui est adaptée à la Définition \ref{def AHP} est celle des motifs étales étudiés par Ayoub. Cette  catégorie $\DM(S)$ est obtenue à partir de $D(\PSh(S))$ en localisant pour imposer la descente étale et l'invariance par $\mathbb{A}^1$-homotopie, puis en stabilisant.}, les préfaisceaux ci-dessus induisent des motifs et la transformation naturelle un morphisme entre eux que l'on notera \[\alpha_{G/S} : M(G/S) \longrightarrow   M_1(G/S).\]

Le motif $M(G/S) $ est appelé le motif de $G$ et le motif $M_1(G/S)$ est appelé le $1$-motif de $G$.
\end{defin}
    \begin{thm}\label{thm AHP2}
       Gardons les notations de la définition ci-dessus. Alors le motif $\Sym^r M_1(G/S)$ est nul   pour $r$ assez grand et  le morphisme $\alpha_{G/S}$ se prolonge en un unique morphisme de motifs en algèbres de Hopf 
    \begin{align}\label{iso AHP}\varphi_{G/S} : M(G/S)\longrightarrow \bigoplus_{r=0}  \Sym^r M_1(G/S)  \end{align}
    qui est de plus un isomorphisme. 
    \end{thm}

    \begin{rem}\label{garufa}
\begin{enumerate}    
\item (Motifs vs faisceaux.) Ce théorème n'est pas une conséquence formelle d'un énoncé sur les préfaisceaux. Remarquons par exemple  que le faisceau $\Sym^r \underline{G}$ n'est pas nul, puisque $\underline{G}$ est un faisceau en espaces vectoriels. 

Un phénomène plus subtile est le suivant : le théorème montre en particulier l'existence d'applications non nulles de $ M_1(G/S)$ vers $ M(G/S)$. En revanche, il n'y a pas d'application non nulle du faisceau  $\underline{G}$ vers $\Q(G)$. Esquissons l'argument. Soient $\alpha : \underline{G} \rightarrow \Q(G)$ une telle transformation naturelle et $\id_G\in  \underline{G} (G)$ l'application identité. La naturalité de $\alpha$ implique qu'il suffit de voir que $\alpha(\id_G)=0$. 
Posons $\alpha(\id_G)  = a_i\cdot f_i$ où $a_i$ sont des nombres rationnels  et les $f_i$ sont des endomorphismes de $G$. 

Pour montrer $a_i\cdot f_i=0$ considérons la naturalité par rapport aux morphismes $n_G$ de multiplication par $n$ :
 \[
\xymatrix{
    \id_G \ar@{|->}[ddd]   \ar@{|->}[rrr]&  &   &  \sum a_i\cdot f_i  \ar@{|->}[ddd] \\
     & \underline{G} \ar@{->}^{n_G}[d]   \ar@{->}^{\!\!\!\!\!\!\!\alpha}[r]  & \Q(G)(G) \ar@{->}^{n_G}[d] & \\
          & \underline{G}   \ar@{->}^{\!\!\!\!\!\!\!\alpha}[r]  & \Q(G)(G) & \\
    n_G=n\cdot  \id_G  \ar@{|->}[rr]  &   &  \sum na_i\cdot f_i \ar@{=}[r] & \sum a_i\cdot (n_G \circ  f_i ).
  }
    \]
On prétend que l'égalité  $\sum na_i\cdot f_i  = \sum a_i\cdot (n_G \circ  f_i )$ force tous les $f_i$ à être nuls. tout d'abord supposons par l'absurde qu'il y avait un $f_i$, disons $f_1$, qui n'était pas de torsion. Alors la liste des morphismes $n_G \circ  f_1 $ serait infinie et on pourrait choisir un $n$ tel que $n_G \circ  f_1$ n'apparaisse pas dans la liste des $f_i$. Ceci contredirait l'égalité $\sum na_i\cdot f_i  = \sum a_i\cdot (n_G \circ  f_i )$.

On peut alors supposer que les $f_i$ soient tous de torsion et  on peut donc  choisir un $n$ tel que  les $n_G \circ  f_i $ soient tous nuls.  L'égalité $\sum na_i\cdot f_i  = \sum a_i\cdot (n_G \circ  f_i )=0$  implique alors $\sum a_i\cdot f_i =0$. 
\item (Décomposition de Chow--Künneth.) On continue à noter par $n_G:G \rightarrow G$ le morphisme de multiplication par $n$. Remarquons que son action sur $ M_1(G/S)  $ vaut $n \cdot \Id$.  En particulier 
 l'isomorphisme \eqref{iso AHP} donne  une décomposition de $ M(G/S)  $ en espaces propres par rapport à l'action de $n_G$. On en déduit par ailleurs que cette décomposition est une décomposition de Chow--Künneth relative (voir la Conjecture \ref{conj CK} pour le cas absolu).

 \item   (Décomposition de l'anneau de Chow.)  La décomposition du point (2) donne une décomposition de \begin{align}\label{maipo}\Hom_{DM(S)}(M(G/S),\mathbbm{1}(p)[q])\end{align} en espaces propres par rapport à l'action de $n_G$. Comme ces $\Hom$ calculent\footnote{Les constructions de Grothendieck et de Voevodsky ont une convention de covariance différente, notamment les motifs de Chow se plongent dans les motifs de Voevodsky par un foncteur contravariant. C'est la raison pour laquelle l'objet $\mathbbm{1}(p)[q]$ apparaît à droite dans la formule \eqref{maipo}.} les groupes de Chow   supérieurs  \cite[Corollary 14.2.14]{CD} on  déduit le Théorème \ref{thm AHP}.

\item (Décomposition de motifs vs décomposition d'anneaux de Chow.) Pour les schémas abéliens $G=A$ les décompositions des anneaux de Chow au point (3) permettent de retrouver celle du motif au point (2) comme l'ont remarqué  Deninger et Murre. Le point est de considérer $A\times_S A$ comme schéma abélien sur $A$ et d'appliquer le Théorème  \ref{thm beauville} à ce schéma abélien : la diagonale se décomposera alors en somme de vecteurs propres. D'autre part la diagonale s'interprète comme l'identité du motif $M(A/S)$ et on peut verifier que cette décomposition  de $\Id \in \End (M(A/S))$ est  une décomposition en somme de projecteurs orthogonaux.

Pour un $G$ général, la formule \[ \End_{DM(S)}(M(G/S))= \Hom_{DM(S)}(M(G/S)  \otimes M(G/S)^{\vee},\mathbbm{1} ) \]
ne permet pas de relier ce groupe à l'anneau de Chow de $G\times_S G $. En effet la dualité de Poincaré  
identifie, à un twist et shift près,  $M(G/S)^{\vee}$ avec $M_c(G/S) $ qui n'est pas, en général, $M(G/S)$.
Dans ce cas la décomposition   des anneaux de Chow du Théorème \ref{thm AHP} ne permet pas de retrouver celle des motifs au point (2).

(L'argument ci-dessus montre qu'en général  les endomorphismes d'un motif sont reliés aux groupes de Chow uniquement dans le cas propre et lisse. Cela a déjà été signalé au \S \ref{motifs purs} et c'est le point qui limite la construction classique de Grothendieck au cadre propre et lisse.)

\item (Voevodsky vs Chow.) Dans le cas $G=A$ d'un schéma abélien, une formule \begin{align}\label{kunnemann} M(A/S)\cong \bigoplus_{r=0}  \Sym^r M_1(A/S)\end{align} a été montré par Künnemann encore à l'aide de la transformée de Fourier \cite{Ku1}. Une version plus faible, valable pour les motifs homologiques, a été discutée dans la   Proposition \ref{mot var abel}.

A l'époque de \cite{Ku1} on ne disposait pas des motifs de Voevodsky et le travail a été fait dans les motifs de Chow. Dans ce cas l'existence du motif $M_1(A/S)$ présent dans la formule \eqref{kunnemann} n'est pas du tout triviale : il faut construire un projecteur convenable de   $\End (M(A/S))$. Ce motif est en revanche facile à définir dans le cadre de Voevodsky (Définition \ref{def AHP}). Un des avantage des motifs de Voevodsky sur les motifs de Chow est notamment cette possibilité de disposer facilement de motifs par des constructions faisceautiques : c'est le principe que nous avons mentionné au \S \ref{complexes motiviques}.

 \end{enumerate}    

    \end{rem}
    \begin{proof}
    La preuve se base sur deux dévissages qui font chacun   l'objet d'un article. Un premier dévissage sert à se réduire au cas d'un corps algébriquement clos $S= \Spec(K) $, \cite{AHP}. Le deuxième \cite{AEH} est une réduction aux cas des variétés abéliennes, ce qui nous ramène essentiellement au résultat de Künnemann \cite{Ku1}.

    \subsection*{Réduction à $S= \Spec(K) $} Pour le premier dévissage on utilise le théorème suivant d'Ayoub \cite[Proposition 3.24]{Ayoubet}. Si $f:M \rightarrow N$ est un morphisme de motifs dans $\DM(S)$ alors pour voir si $f$ est un isomorphisme il suffit de voir si son tiré en arrière en tout point géométrique l'est. 
    
    Soit $i:  \Spec(K) \rightarrow S$ un point géométrique. Pour compléter le premier  dévissage il suffit alors de montrer que
    \begin{align}\label{chgbase}i^*\varphi_{G/S} =\varphi _{G\times_S  \Spec(K) / \Spec(K) }.\end{align}
    La formule \eqref{chgbase} est en fait le point technique du travail. Pour tout morphisme $g:T\rightarrow S$, le foncteur $g^*$ est caractérisé par la propriété que, pour tout $S$-schéma lisse $X$, on ait le   changement de base
        \begin{align}g^*M(X/S)=M(X\times_S T/T)\end{align}
comme pour la cohomologie à support compact. Or parmi les deux motifs qui interviennent dans le morphisme $\varphi_{G/S}$, seulement $M(G/S)$ est de cette forme. Le point est alors de trouver une résolution de $M_1(G)$ par un complexe dont tous les termes sont des sommes d'objets de la forme $M(X/S)$ en tous les degrés et toutes les flèches de connections sont induites par des morphismes de $S$-schémas. 

Nous construisons une telle résolution  inspirée par des construction de \cite{Breen}. On utilise uniquement des $X$ qui sont des puissances de $G$. Avec les notations de la Définition \ref{def AHP} on peut écrire son début sous la forme
\[ \cdots  \longrightarrow  \Q(G\times G)   \stackrel{t_G}{\longrightarrow} \Q(G)    \stackrel{s_G}{\longrightarrow}   \underline{G},\]
où
$ t_G([g_1,g_2] ) =     [g_1]+[g_2]-[g_1+g_2] .$

\subsection*{Réduction à $G=A$} Pour le deuxième dévissage fixons un corps algébriquement  clos $K$ et considérerons $S=\Spec(K)$. Dans la suite on allégera la notation en enlevant les $/S$.

 Le théorème de Chevalley nous dit qu'un groupe algébrique sur un corps $K$ s'insère dans une suite exacte
\[0 \longrightarrow  L \longrightarrow  G  \longrightarrow  A \longrightarrow  0,\]
où $A$ est une variété abélienne et $L$ est un groupe linéaire. On peut alors raisonner par récurrence sur la dimension de $L$ et se ramener à 
    \begin{align}\label{cas1} 0 \longrightarrow  \mathbb{G}_a \longrightarrow  G  \longrightarrow  H \longrightarrow  0\end{align}
ou
    \begin{align}\label{cas2} 0 \longrightarrow  \mathbb{G}_m \longrightarrow  G  \longrightarrow  H \longrightarrow  0,\end{align}
    où dans les deux cas $H$ est un groupe pour lequel l'énoncé est connu.
    Le cas \eqref{cas1} est facile : par homotopie l'énoncé pour $G$ est équivalent à l'énoncé pour $H$.
    Pour le  cas \eqref{cas2} l'argument est plus délicat. 
    
    Premièrement, la suite exacte \eqref{cas2} donne une suite exacte au niveau des foncteurs des points. Si on applique $\Sym^n$ a cette dernière on obtient un triangle
    \[ \Sym^{r-1} M_1(H)  \otimes  M_1(\mathbb{G}_m)  \longrightarrow  \Sym^r M_1(G)  \longrightarrow \Sym^r M_1(H),\]
    où on a utilisé $\Sym^2 M_1(\mathbb{G}_m) =0$ \cite[Corollary 2.1.5]{TMF}. Cette suite exacte permet de déduire par récurrence que $ \Sym^r M_1(G) $ s'annule pour $r$ assez grand.
    
    Deuxièmement, on complète  le $\mathbb{G}_m$-fibré $G \longrightarrow H$ en un $\mathbb{A}^1$-fibré $\bar{G} \longrightarrow H$ avec une section nulle. Le triangle de localisation par rapport à la section donne
    \[     M(G) \longrightarrow   M(\bar{G}) \longrightarrow M(H)(1)[2],\]
    or par homotopie $M(\bar{G})=M(H)$. On utilise maintenant l'hypothèse de récurrence et on en déduit le diagramme
    \begin{align} 
\xymatrix@C=14pt{   
M(G) \ar@{->}[rr]^(0.5){ } \ar@{-->}[dd]_(0.5){\psi} & & 
M(H) \ar@{->}[rr]^(0.45){ } \ar@{->}[dd]^(0.5){\varphi_{H}} & & 
M(H)(1)[2] \ar@{->}[dd]^(0.5){\varphi_{H}(1)[2]}   & & 
  \\
& & & & & & \\
\bigoplus_{r}  \Sym^r M_1(G) \ar@{->}[rr] & & 
\bigoplus_{r}  \Sym^r M_1(H) \ar@{->}[rr] & & 
\bigoplus_{r}  \Sym^r M_1(H)(1)[2].  & &    
} 
\end{align} 
L'existence de $\psi$ suit de la commutativité du carré de droite. Cette commutativité n'est pas gratuite, elle utilise la définition de l'application $\varphi_{H}$.

Par hypothèse de récurrence on déduit que  l'application $\psi$ est un isomorphisme. Malheureusement la récurrence n'est pas terminée puisqu'on ne sait pas lier $\psi$ à  $\varphi_{G}$. L'existence de $\psi$ a tout de même  une  conséquence importante :  $M(G)$ est un motif de dimension finie, voir la Remarque \ref{rem kimura}(2).

On montre que la réalisation de $\varphi_{G}$ est  un isomorphisme. On peut alors appliquer la Proposition \ref{prop conserv} à $\varphi_{G}$ et  $\psi^{-1}$ pour déduire que $\varphi_{G}$  est un isomorphisme après passage au quotient par l'équivalence homologique. On peut ensuite utiliser les propriétés des motifs de dimension finie pour conclure que $\varphi_{G}$  est un isomorphisme même avant passage au quotient (voir la Remarque \ref{rem kimura}(2), cf. le Théorème \ref{thm kimura} dans le cas pur).
\end{proof}

\begin{rem}
\begin{enumerate}
\item La notion de dimension finie dans les motifs a toujours été appliquée pour les motifs purs : elle ne se comporte pas bien par suite exacte et on connait des exemples de motifs mixtes qui ne sont pas de dimension finie.  À notre connaissance cette preuve est le premier exemple d'application de ces idées aux motifs mixtes.
\item Le Théorème \ref{thm AHP2} a permis à Huber et Kings de construire le polylogarithme  pour tous les schémas en groupes commutatifs.
\item Le Théorème \ref{thm AHP2} donne une description du motif d'un groupe algébrique commutatif $G$ en terme d'un motif assez simple, $M_1(G)$. On pourrait espérer que cela puisse aider à une meilleure compréhension des anneaux de Chow de $G$.
\end{enumerate}

\end{rem}
          \section{Construction de classes algébriques}\label{section lef}

    Cette section concerne la conjecture standard de type Lefschetz, introduite dans la    Conjecture \ref{CSTL}. Dans  un travail en cours en collaboration avec Mattia Cavicchi, Robert Laterveer et Giulia Saccà, nous étudions cette conjecture pour les variétés hyper-kähler  qui admettent une fibration lagrangienne. Le point de départ est une idée récente de Voisin \cite{Voisin} que l'on regarde dans la perspective du théorème de décomposition.
    
    Nous nous concentrerons dans la suite sur les variétés complexes, la conjecture standard de type Lefschetz est alors une instance particulière de la conjecture de Hodge. De façon assez surprenante elle en est aussi le pilier principal : André démontre que sous la conjecture standard de type Lefschetz  le transport parallèle de classes algébriques est encore algébrique \cite{AndIHES}. Il en déduit que cette conjecture impliquerait la conjecture de Hodge pour les variétés abéliennes. 
    
    La conjecture standard de type Lefschetz est connue pour les variétés abéliennes, voir la Proposition \ref{mot var abel}, et on sait en déduire le cas des surfaces. Plus récemment Charles et Markmann l'ont montrée pour les variétés hyper-kähler de type $KS^{[n]}$, \cite{ChaMark}.
    
La dimension d'une variété $X$ sera notée $d_X$ et la dimension de la fibre générique d'un morphisme $f$ sera notée $d_f$. Même en présence de faisceaux pervers on utilisera la convention classique pour les degrés cohomologiques.
\subsection{Une approche naïve}\label{approche}
Considérons une variété projective et lisse $X$ et supposons qu'elle admette un morphisme $f : X \rightarrow B$ vers une base $B$ qui est aussi projective et lisse. Nous nous demandons jusqu'à quel point connaître la conjecture standard de type Lefschetz pour $B$ et pour les fibres lisses de $f$ peut aider pour montrer la conjecture standard de type Lefschetz pour $X$. 

Le   théorème de décomposition  implique en particulier une décomposition\footnote{Une telle décomposition n'est pas unique en général, seulement la filtration perverse associée l'est. Deligne, puis De Cataldo, ont proposé des décompositions qui se comportent mieux que les autres \cite{Deldec,Decdec}.} de structures de Hodge
\begin{align}\label{BBD}  \HW^n(X)=\bigoplus_{a+b=n} \HW^a(B, ^p\!\!R^{b}f_* \Q  ).\end{align} 
Soient $\eta$ un diviseur sur $X$ qui soit relativement ample, $L_B$ un diviseur ample sur $B$ et $\beta$ son tiré en arrière sur $X$. Le théorème de décomposition fournit aussi les isomorphismes
\begin{align}\label{BBDeta} \cup \eta^{b} :  \HW^a(B, ^p\!\!R^{d_f-b}f_* \Q  ) \isocan \HW^a(B, ^p\!\!R^{d_f+b}f_* \Q  ), \end{align} 
\begin{align}\label{BBDbeta}  \cup \beta^{a} :  \HW^{d_B -a}(B, ^p\!\!R^{b}f_* \Q  ) \isocan \HW^{d_B +a}(B, ^p\!\!R^{b}f_* \Q  ).\end{align} 

La combination de \eqref{BBD}, \eqref{BBDeta} et \eqref{BBDbeta} peut suggérer que la conjecture standard de type Lefschetz se ramène à montrer que les inverses de ces isomorphismes sont algébriques  et à première vue on pourrait penser que ces derniers se ramènent uniquement à l'étude des fibres de $f$ et de  la base $B$. Mais il faut en fait faire attention à un certain nombre de subtilités.
\begin{enumerate}
\item Les faisceaux pervers $^p\!R^{b}f_* \Q$ dépendent aussi des fibres singulières du morphisme $f$.
\item Inverser l'action de $L_B$ sur la cohomologie de $B$ ne suffit pas à inverser $\beta$. En effet $f^* \HW(B)\subset \HW(X)$ ne représente que le facteur $\HW^*(B, ^p\!\!R^0f_* \Q  )$ de la décomposition $\eqref{BBD}$.
\item Même si on était capables de construire des correspondances algébriques $\Lambda_{ \eta^{b}}$ et $\Lambda_{\beta^{a}}$ qui agiraient comme les inverses de  \eqref{BBDeta} et \eqref{BBDbeta} il n'est pas clair de pouvoir les mettre ensemble pour  déduire une correspondance algébrique $\Lambda_{n}: \HW^{d_X+n}(X ) \isocan \HW^{d_X -n}(X  ).$ Il faudrait par exemple contrôler comment  $\Lambda_{ \eta^{b}}$ agit sur les  degrés pervers différents de $b$. C'est délicat,   notamment parce  que les opérateurs $\eta$ et $\beta$ ne sont pas bigradués en général.
\end{enumerate}
  
  Dans le  cas des variétés hyper-kähler  qui admettent une fibration lagrangienne on peut espérer contourner ces problèmes :  le théorème du support de Ngô donne une description explicite des faisceaux pervers $^p\!R^{b}f_* \Q$, un argument de Voisin permet grosso-modo  de construire une deuxième fibration lagrangienne qui inverse le rôle de $\eta$ et $\beta$, enfin Shen et Yin ont montré que les opérateurs $\eta$ et $\beta$  sont en fait   bigradués pour les fibrations lagrangiennes.
  
  \begin{prop}\label{lef rel}
  Soit $f:X\rightarrow B$ une application entre variétés projectives, lisses et connexes et soit $U$ l'ouvert de $B$ sur lequel l'application est lisse.
  \begin{enumerate}
  \item Supposons que la fibre générique vérifie la conjecture standard de type Lefschetz et que \begin{align}\label{hyp supp}^p\!R^{b}f_* \Q= IC ((R^{b}f_* \Q)_{\vert U}).  \end{align}
  Alors il existe des correspondances dans $X \times_B X$ dont l'action sur la cohomologie relative induit des isomorphismes $^p\!R^{d_f+b}f_* \Q \isocan  {}^p\!R^{d_f-b}f_* \Q$.
  \item  Supposons qu'il existe des correspondances dans $X \times_B X$ dont l'action sur la cohomologie relative induit des isomorphismes $^p\!R^{d_f+b}f_* \Q \isocan {}^p\!R^{d_f-b}f_* \Q$. Alors il existe une décomposition du motif \begin{align}\label{motBBD}   \mathfrak{h}(X)=\bigoplus_{ b} \mathfrak{h}(B, ^p\!\!R^{b}f_* \Q  )\end{align}
  où $R(\mathfrak{h}(B, ^p\!\!R^{b}f_* \Q  )) = \HW^*(B, ^p\!\!R^{b}f_* \Q  )$. De plus chaque facteur de la décomposition est autodual à un twist près.
\end{enumerate}
  \end{prop}
  \begin{rem} (Autour de la preuve.)
      La preuve de (1) est facile. Soit $Y$ la fibre générique de $f$. L'hypothèse sur $Y$ dans (1) fournit des correspondances dans $Y\times Y$,  dont les adhérences donnent des correspondances dans $X \times_B X$. Leur action sur la cohomologie relative   $ IC ((R^{b}f_* \Q)_{\vert U})$ est contrôlée par leur action sur le système local $(R^{b}f_* \Q)_{\vert U}$ et donc par l'action sur la cohomologie de $Y$.
  
   Pour (2) on suit l'argument classique qui montre que Lefschetz implique Künneth \cite{GKL}. Par rapport au cas absolu on prendra garde au fait  
que la décomposition \eqref{motBBD} n'est pas unique. L'action des correspondances relatives ne respecte pas cette graduation mais seulement la filtration associée.  Par ailleurs on ne sait pas démontrer que toute décomposition cohomologique \eqref{BBD} est la réalisation d'une décomposition motivique \eqref{motBBD}  mais seulement qu'il en existe au moins une qui est d'origine motivique.
   \end{rem}
   \begin{cor}\label{cor lef}
   Soit $f:X\rightarrow \mathbb{P}^1$ une fibration de Lefschetz dont la fibre générique vérifie la conjecture standard de type Lefschetz (par exemple $f$ est une fibration en surfaces). Alors il existe une décomposition du motif \begin{align}\label{motBBDLef}   \mathfrak{h}(X)=\bigoplus_{ b} \mathfrak{h}(\mathbb{P}^1, ^p\!\!R^{b}f_* \Q  )\end{align}
  où $R(\mathfrak{h}(\mathbb{P}^1, ^p\!\!R^{b}f_* \Q  )) = \HW^*(\mathbb{P}^1, ^p\!\!R^{b}f_* \Q  )$. De plus chaque facteur de la décomposition est autodual à un twist près.
   \end{cor}
 \begin{rem}\label{rem solide} (Du relatif à l'absolu.)
    Toute variété de dimension trois admet une fibration de Lefschetz après éclatement le long d'une courbe $C$. D'autre part la formule de  l'éclatement 
    \[\mathfrak{h}(Bl_C(X)) = \mathfrak{h}(X) \oplus \mathfrak{h}(C)(-1),\]
    due à Manin \cite{Manin},
    montre que la conjecture standard de type  Lefschetz pour $X$ se ramène à celle pour $Bl_C(X)$. Si l'on veut étudier la conjecture standard de type  Lefschetz pour une  variété $X$  de dimension trois on peut donc supposer que $X$ admet une telle fibration. Ce qui manque au corollaire ci-dessus pour déduire cette conjecture est l'existence d'une décomposition 
    \[ \mathfrak{h}(\mathbb{P}^1, ^p\!\!R^{b}f_* \Q  ) = \bigoplus_{ a=0}^2 \mathfrak{h}^a(\mathbb{P}^1, ^p\!\!R^{b}f_* \Q  ),\]
    où $R(\mathfrak{h}^a(\mathbb{P}^1, ^p\!\!R^{b}f_* \Q  )) = \HW^a(\mathbb{P}^1, ^p\!\!R^{b}f_* \Q  )$, telle que chaque facteur de la décomposition soit autodual à un twist près.
    
Pour construire cette décomposition on pourrait vouloir utiliser l'isomorphisme \eqref{BBDbeta} et essayer de construire une correspondance algébrique qui induise l'inverse. C'est une question qui ressemble au problème original de construction de l'inverse de l'opérateur de Lefschetz. On ne sait pas si c'est un problème plus simple, voir aussi la Remarque \ref{rem final}.
   \end{rem}
  \begin{thm}
  Soit $X$ une variété hyper-kähler de dimension $2n$ et $f:X\rightarrow \mathbb{P}^n$ une fibration lagrangienne dont toutes les fibres sont irréductibles. Supposons que le schéma en groupes $\Aut^0(f)$ soit polarisable (au sens de Ngô).
  Alors 
  \begin{align}\label{motBBDHK}   \mathfrak{h}(X)=\bigoplus_{ b=0}^{2n} \mathfrak{h}(\mathbb{P}^n, ^p\!\!R^{b}f_* \Q  )\end{align}
  où $R(\mathfrak{h}(\mathbb{P}^n, ^p\!\!R^{b}f_* \Q  )) = \HW^*(\mathbb{P}^n, ^p\!\!R^{b}f_* \Q  )$. De plus chaque facteur de la décomposition est autodual à un twist près.
  \end{thm}
  \begin{proof}
  On utilise le théorème du support de Ngô, \cite{Ngo} : quand les fibres sont toutes irréductibles on a bien l'hypothèse \eqref{hyp supp}. D'autre part la fibre générique est une variété abélienne, donc elle vérifie la conjecture standard de type Lefschetz. On peut alors utiliser la Proposition \ref{lef rel} pour conclure. 
  \end{proof}

     \begin{cor}
  Soit $X$ une variété hyper-kähler de dimension $2n$  de rang de Picard $2$ qui admet une fibration lagrangienne. Supposons que pour toute  fibration lagrangienne $g$ de n'importe quel  variétés hyper-kähler  birationnelle à $X$, les fibres de $g$ soient irréductibles et   le schéma en groupes $\Aut^0(g)$ soit polarisable.  Alors la conjecture standard de type Lefschetz est vraie pour $X$.
  \end{cor}
    \begin{proof}
    Fixons $f  : X \rightarrow \mathbb{P}^n$ une fibration lagrangienne et soit $\eta$ et $\beta$ les diviseurs introduits dans \S \ref{approche}.
Un argument de Voisin \cite{Voisin}  montre grosso-modo l'existence d'une fibration lagrangienne $g: X \rightarrow \mathbb{P}^n$ où le rôle de $\eta$ et $\beta$ est  inversé. En fait $g$ existe seulement sur une variété  hyper-kähler qui est birationnelle à $X$, mais elles ont le même motif \cite{Riess}.

On peut alors appliquer le théorème ci-dessus à $f$ et $g$ pour déduire deux décompositions du motif  $\mathfrak{h}(X)$ dont tous les facteurs sont autoduaux.
Il n'est pas clair que ces deux décompositions soient compatibles, voir aussi le point (3) dans \S \ref{approche}.  Mais un théorème de Shen et Yin montre que les opérateurs $\eta$ et $\beta$ sont bigradués pour les fibrations lagrangiennes : on peut voir que ceci    force les deux décompositions à être   compatibles. La bidécomposition que l'on en déduit   est plus fine que \eqref{BBD} et a tous les facteurs autoduaux, ce qui implique la conjecture standard de type Lefschetz.
\end{proof}
Le corollaire ci-dessus s'applique par exemple aux variétés hyper-kähler construites par Laza--Saccà--Voisin \cite{LSV}. 
\begin{rem} (Généralisations.)
On souhaite se débarrasser de l'hypothèse d'irréductibilité des fibres dans le théorème ci-dessus, ce qui permettrait de l'enlever aussi dans son corollaire et de pouvoir l'appliquer à beaucoup plus de variétés. 

Si les fibres ne sont pas irréductible les faisceaux pervers $^p\!R^{b}f_* \Q $ ne vérifient plus \eqref{hyp supp}. Cependant le théorème du support de Ngô décrit aussi la nature des faisceaux pervers supportés sur les sous-variétés strictes de la base. Il montre  l'existence de certaines variétés abéliennes contenues dans certaines fibres singulières dont la cohomologie contrôle les faisceaux pervers. On pense qu'une variante stratifiée de la Proposition \ref{lef rel} devrait s'appliquer à ce contexte général. On utilisera encore que les variétés abéliennes vérifient la conjecture standard de type Lefschetz.
\end{rem}
   
   \begin{rem}\label{rem final} (Du relatif à l'absolu, suite.)
    Reprenons les notations de la Remarque \ref{rem solide}. On souhaiterait montrer  l'existence d'une décomposition 
    \begin{align}\label{gardel} \mathfrak{h}(\mathbb{P}^1, ^p\!\!R^{b}f_* \Q  ) = \bigoplus_{ a=0}^2 \mathfrak{h}^a(\mathbb{P}^1, ^p\!\!R^{b}f_* \Q  ),\end{align}
    où $R(\mathfrak{h}^a(\mathbb{P}^1, ^p\!\!R^{b}f_* \Q  )) = \HW^a(\mathbb{P}^1, ^p\!\!R^{b}f_* \Q  )$, telle que chaque facteur de la décomposition soit autodual à un twist près : ceci montrerait la conjecture standard de type Lefschetz pour les variétés   de dimension trois.
    
Pour construire cette décomposition on pourrait vouloir utiliser l'isomorphisme \eqref{BBDbeta} induit par la classe $\beta$ et essayer de construire une correspondance algébrique qui induise l'inverse. Inspirés par l'idée de Voisin esquissé dans la preuve du corollaire ci-dessus on peut essayer de construire une nouvelle fibration $g: X \rightarrow \mathbb{P}^2$ telle que $\beta$ soit relativement ample (ces constructions sont toujours possibles après éclatement). Une telle fibration induit   une décomposition
  \begin{align}\label{varela}   \mathfrak{h}(X)=\bigoplus_{ a=0}^2 \mathfrak{h}(\mathbb{P}^2, ^p\!\!R^{a}g_* \Q  )\end{align}
  où $\cup \beta  : \mathfrak{h}(\mathbb{P}^2, ^p\!\!R^{0}g_* \Q ) \isocan \mathfrak{h}(\mathbb{P}^2, ^p\!\!R^{2}g_* \Q)$ admet un inverse algébrique. Le problème est qu'en général la décomposition \eqref{varela} n'induit pas la décomposition \eqref{gardel}. C'est encore lié au fait que les actions de $\eta$ et $\beta$ ne sont pas bigraduées en général.
  \end{rem}
  \newpage

\bibliographystyle{alpha}	
\bibliography{masterbib}

\def\cprime{$'$}
\begin{thebibliography}{AEWH15}

\bibitem[ACLS22]{ACLS}
Giuseppe Ancona, Mattia Cavicchi, Robert Laterveer, and Giulia Saccà.
\newblock Standard conjectures for some lagrangian fibrations.
\newblock {\em prépublication}, 2022.

\bibitem[AEWH15]{AEH}
Giuseppe Ancona, Stephen Enright-Ward, and Annette Huber.
\newblock On the motive of a commutative algebraic group.
\newblock {\em Documenta Math.}, 20:807--858, 2015.

\bibitem[AF22]{AF}
Giuseppe Ancona and Dragos Fratila.
\newblock Algebraic classes in mixed characteristic and {A}ndré's $p$-adic
  periods.
\newblock {\em prépublication}, 2022.

\bibitem[AHPL16]{AHP}
Giuseppe Ancona, Annette Huber, and Simon Pepin~Lehalleur.
\newblock On the relative motive of a commutative group scheme.
\newblock {\em Algebr. Geom.}, 3(2):150--178, 2016.

\bibitem[AM22]{AM}
Giuseppe Ancona and Adriano Marmora.
\newblock The {H}ilbert symbol in the {H}odge standard conjecture.
\newblock {\em prépublication}, 2022.

\bibitem[Anc21]{Anc21}
Giuseppe Ancona.
\newblock Standard conjectures for abelian fourfolds.
\newblock {\em Invent. Math.}, 223(1):149--212, 2021.

\bibitem[Anc22]{Ascona}
Giuseppe Ancona.
\newblock Conservativity of realization functors on motives of abelian type
  over finite fields.
\newblock {\em Proceedings of "Motives and Complex Multiplication"}, Fres\'an,
  Jetchev, Jossen and Pink editors, 2022.

\bibitem[And90]{Betti}
Yves Andr\'{e}.
\newblock {$p$}-adic {B}etti lattices.
\newblock In {\em {$p$}-adic analysis ({T}rento, 1989)}, volume 1454 of {\em
  Lecture Notes in Math.}, pages 23--63. Springer, Berlin, 1990.

\bibitem[And95]{Andre1}
Yves Andr\'{e}.
\newblock Th\'{e}orie des motifs et interpr\'{e}tation g\'{e}om\'{e}trique des
  valeurs {$p$}-adiques de {$G$}-functions (une introduction).
\newblock In {\em Number theory ({P}aris, 1992--1993)}, volume 215 of {\em
  London Math. Soc. Lecture Note Ser.}, pages 37--60. Cambridge Univ. Press,
  Cambridge, 1995.

\bibitem[And96]{AndIHES}
Yves Andr\'{e}.
\newblock Pour une th\'{e}orie inconditionnelle des motifs.
\newblock {\em Inst. Hautes \'{E}tudes Sci. Publ. Math.}, (83):5--49, 1996.

\bibitem[{And}03]{Andrep}
Yves {Andr\'e}.
\newblock {\em {Period mappings and differential equations. From \({\mathbb
  C}\) to \({\mathbb C}_p\). T\^ohoku-Hokkaid\^o lectures in arithmetic
  geometry. With appendices: A: Rapid course in \(p\)-adic analysis by F. Kato,
  B: An overview of the theory of \(p\)-adic unifomization by F. Kato, C:
  \(p\)-adic symmetric domains and Totaro's theorem by N. Tsuzuki}}, volume~12.
\newblock Tokyo: Mathematical Society of Japan, 2003.

\bibitem[And04]{Andmot}
Yves Andr{\'e}.
\newblock {\em Une introduction aux motifs (motifs purs, motifs mixtes,
  p\'eriodes)}, volume~17 of {\em Panoramas et Synth\`eses [Panoramas and
  Syntheses]}.
\newblock Soci\'et\'e Math\'ematique de France, Paris, 2004.

\bibitem[Ayo07]{Ayoub_these_1}
Joseph Ayoub.
\newblock Les six op\'erations de {G}rothendieck et le formalisme des cycles
  \'evanescents dans le monde motivique. {I}.
\newblock {\em Ast\'erisque}, (314):x+466 pp. (2008), 2007.

\bibitem[Ayo10]{Ayoubbetti}
Joseph Ayoub.
\newblock Note sur les op\'erations de {G}rothendieck et la r\'ealisation de
  {B}etti.
\newblock {\em J. Inst. Math. Jussieu}, 9(2):225--263, 2010.

\bibitem[Ayo14]{Ayoubet}
Joseph Ayoub.
\newblock La r\'ealisation \'etale et les op\'erations de {G}rothendieck.
\newblock {\em Ann. Sci. \'Ecole Norm. Sup.}, 47:1--141, 2014.

\bibitem[Ayo15]{AyoubKZ}
Joseph Ayoub.
\newblock Une version relative de la conjecture des p\'{e}riodes de
  {K}ontsevich-{Z}agier.
\newblock {\em Ann. of Math. (2)}, 181(3):905--992, 2015.

\bibitem[BC16]{BostCha}
Jean-Beno\^{\i}t Bost and Fran\c{c}ois Charles.
\newblock Some remarks concerning the {G}rothendieck period conjecture.
\newblock {\em J. Reine Angew. Math.}, 714:175--208, 2016.

\bibitem[Bea86]{Beau}
Arnaud Beauville.
\newblock Sur l'anneau de {C}how d'une vari\'et\'e ab\'elienne.
\newblock {\em Math. Ann.}, 273(4):647--651, 1986.

\bibitem[Ber77]{Bert}
Daniel Bertrand.
\newblock Sous-groupes \`a un param\`etre {$p$}-adique de vari\'{e}t\'{e}s de
  groupe.
\newblock {\em Invent. Math.}, 40(2):171--193, 1977.

\bibitem[Bon10]{Bon}
Mikhail~V. Bondarko.
\newblock Weight structures vs. {$t$}-structures; weight filtrations, spectral
  sequences, and complexes (for motives and in general).
\newblock {\em J. K-Theory}, 6(3):387--504, 2010.

\bibitem[Bre70]{Breen}
Lawrence Breen.
\newblock Extensions of abelian sheaves and {E}ilenberg-{M}ac{L}ane algebras.
\newblock {\em Invent. Math.}, 9:15--44, 1969/1970.

\bibitem[Bro12]{Brown}
Francis Brown.
\newblock Mixed {T}ate motives over {$\Bbb Z$}.
\newblock {\em Ann. of Math. (2)}, 175(2):949--976, 2012.

\bibitem[Bro17]{Brownp}
Francis Brown.
\newblock Notes on motivic periods.
\newblock {\em Commun. Number Theory Phys.}, 11(3):557--655, 2017.

\bibitem[CD19]{CD}
Denis-Charles Cisinski and Fr\'{e}d\'{e}ric D\'{e}glise.
\newblock {\em Triangulated categories of mixed motives}.
\newblock Springer Monographs in Mathematics. Springer, Cham, 2019.

\bibitem[CDN22]{CDN}
Mattia Cavicchi, Frédéric Déglise, and Jan Nagel.
\newblock Motivic decompositions of families with tate fibers: smooth and
  singular cases.
\newblock {\em prépublication}, 3(2), 2022.

\bibitem[CF00]{CF}
Pierre Colmez and Jean-Marc Fontaine.
\newblock Construction des repr\'esentations {$p$}-adiques semi-stables.
\newblock {\em Invent. Math.}, 140(1):1--43, 2000.

\bibitem[Cha13]{Charles}
Fran\c{c}ois Charles.
\newblock The {T}ate conjecture for {$K3$} surfaces over finite fields.
\newblock {\em Invent. Math.}, 194(1):119--145, 2013.

\bibitem[Chu80]{Chud}
G.~V. Chudnovsky.
\newblock Algebraic independence of values of exponential and elliptic
  functions.
\newblock In {\em Proceedings of the {I}nternational {C}ongress of
  {M}athematicians ({H}elsinki, 1978)}, pages 339--350. Acad. Sci. Fennica,
  Helsinki, 1980.

\bibitem[Clo99]{Clozel}
Laurent Clozel.
\newblock Equivalence num\'erique et \'equivalence cohomologique pour les
  vari\'et\'es ab\'eliennes sur les corps finis.
\newblock {\em Ann. of Math. (2)}, 150(1):151--163, 1999.

\bibitem[CM13]{ChaMark}
Fran\c{c}ois Charles and Eyal Markman.
\newblock The standard conjectures for holomorphic symplectic varieties
  deformation equivalent to {H}ilbert schemes of {$K3$} surfaces.
\newblock {\em Compos. Math.}, 149(3):481--494, 2013.

\bibitem[Col90]{Colem}
Robert~F. Coleman.
\newblock On the {F}robenius matrices of {F}ermat curves.
\newblock In {\em {$p$}-adic analysis ({T}rento, 1989)}, volume 1454 of {\em
  Lecture Notes in Math.}, pages 173--193. Springer, Berlin, 1990.

\bibitem[dC13]{Decdec}
Mark Andrea~A. de~Cataldo.
\newblock Hodge-theoretic splitting mechanisms for projective maps.
\newblock {\em J. Singul.}, 7:134--156, 2013.
\newblock With an appendix containing a letter from P. Deligne.

\bibitem[Del89]{Delignelog}
P.~Deligne.
\newblock Le groupe fondamental de la droite projective moins trois points.
\newblock In {\em Galois groups over {${\bf Q}$} ({B}erkeley, {CA}, 1987)},
  volume~16 of {\em Math. Sci. Res. Inst. Publ.}, pages 79--297. Springer, New
  York, 1989.

\bibitem[Del94a]{DelMot}
Pierre Deligne.
\newblock \`a quoi servent les motifs?
\newblock In {\em Motives ({S}eattle, {WA}, 1991)}, volume~55 of {\em Proc.
  Sympos. Pure Math.}, pages 143--161. Amer. Math. Soc., Providence, RI, 1994.

\bibitem[Del94b]{Deldec}
Pierre Deligne.
\newblock D\'{e}compositions dans la cat\'{e}gorie d\'{e}riv\'{e}e.
\newblock In {\em Motives ({S}eattle, {WA}, 1991)}, volume~55 of {\em Proc.
  Sympos. Pure Math.}, pages 115--128. Amer. Math. Soc., Providence, RI, 1994.

\bibitem[Del13]{DelBrown}
Pierre Deligne.
\newblock Multiz\^{e}tas, d'apr\`es {F}rancis {B}rown.
\newblock Number 352, pages Exp. No. 1048, viii, 161--185. 2013.
\newblock S\'{e}minaire Bourbaki. Vol. 2011/2012. Expos\'{e}s 1043--1058.

\bibitem[DM91]{DeMu}
Christopher Deninger and Jacob Murre.
\newblock Motivic decomposition of abelian schemes and the {F}ourier transform.
\newblock {\em J. Reine Angew. Math.}, 422:201--219, 1991.

\bibitem[Fal89]{Falt}
Gerd Faltings.
\newblock Crystalline cohomology and {$p$}-adic {G}alois-representations.
\newblock In {\em Algebraic analysis, geometry, and number theory ({B}altimore,
  {MD}, 1988)}, pages 25--80. Johns Hopkins Univ. Press, Baltimore, MD, 1989.

\bibitem[Fon94]{Fontp}
Jean-Marc Fontaine.
\newblock Le corps des p\'eriodes {$p$}-adiques.
\newblock {\em Ast\'erisque}, (223):59--111, 1994.
\newblock With an appendix by Pierre Colmez, P\'eriodes $p$-adiques
  (Bures-sur-Yvette, 1988).

\bibitem[Fur07]{Furusho}
Hidekazu Furusho.
\newblock {$p$}-adic multiple zeta values. {II}. {T}annakian interpretations.
\newblock {\em Amer. J. Math.}, 129(4):1105--1144, 2007.

\bibitem[Jan07]{Jannsen}
Uwe Jannsen.
\newblock On finite-dimensional motives and {M}urre's conjecture.
\newblock In {\em Algebraic cycles and motives. {V}ol. 2}, volume 344 of {\em
  London Math. Soc. Lecture Note Ser.}, pages 112--142. Cambridge Univ. Press,
  Cambridge, 2007.

\bibitem[Kah03]{Kahn}
Bruno Kahn.
\newblock \'{E}quivalences rationnelle et num\'erique sur certaines
  vari\'et\'es de type ab\'elien sur un corps fini.
\newblock {\em Ann. Sci. \'Ecole Norm. Sup. (4)}, 36(6):977--1002 (2004), 2003.

\bibitem[Kat80]{Katz}
Nicholas~M. Katz.
\newblock {$p$}-adic {$L$}-functions, {S}erre-{T}ate local moduli, and ratios
  of solutions of differential equations.
\newblock In {\em Proceedings of the {I}nternational {C}ongress of
  {M}athematicians ({H}elsinki, 1978)}, pages 365--371. Acad. Sci. Fennica,
  Helsinki, 1980.

\bibitem[Kim05]{Kim}
Shun-Ichi Kimura.
\newblock Chow groups are finite dimensional, in some sense.
\newblock {\em Math. Ann.}, 331(1):173--201, 2005.

\bibitem[Kle68]{GKL}
Steven Kleiman.
\newblock Algebraic cycles and the {W}eil conjectures.
\newblock In {\em Dix espos\'es sur la cohomologie des sch\'emas}, pages
  359--386. North-Holland, Amsterdam, 1968.

\bibitem[KM74]{KM}
Nicholas~M. Katz and William Messing.
\newblock Some consequences of the {R}iemann hypothesis for varieties over
  finite fields.
\newblock {\em Invent. Math.}, 23:73--77, 1974.

\bibitem[Kos22]{Koshi}
Teruhisa Koshikawa.
\newblock The numerical {H}odge standard conjecture for the square of a simple
  abelian variety of prime dimension.
\newblock {\em prépublication}, 3(2), 2022.

\bibitem[K{\"u}n94]{Ku1}
Klaus K{\"u}nnemann.
\newblock On the {C}how motive of an abelian scheme.
\newblock In {\em Motives}, volume 55.1 of {\em Proceedings of Symposia in pure
  mathematics}, pages 189--205. American mathematical society, 1994.

\bibitem[LSV17]{LSV}
Radu Laza, Giulia Sacc\`a, and Claire Voisin.
\newblock A hyper-{K}\"{a}hler compactification of the intermediate {J}acobian
  fibration associated with a cubic 4-fold.
\newblock {\em Acta Math.}, 218(1):55--135, 2017.

\bibitem[Man68]{Manin}
Ju.~I. Manin.
\newblock Correspondences, motifs and monoidal transformations.
\newblock {\em Mat. Sb. (N.S.)}, 77 (119):475--507, 1968.

\bibitem[Mil07]{MilneTate}
J.~Milne.
\newblock The tate conjecture over finite fields.
\newblock pages 1--27. 2007.
\newblock Available on the author's webpage.

\bibitem[Mum68]{Mumzero}
D.~Mumford.
\newblock Rational equivalence of {$0$}-cycles on surfaces.
\newblock {\em J. Math. Kyoto Univ.}, 9:195--204, 1968.

\bibitem[Ngô10]{Ngo}
Bao~Ch\^{a}u Ngô.
\newblock Le lemme fondamental pour les alg\`ebres de {L}ie.
\newblock {\em Publ. Math. Inst. Hautes \'{E}tudes Sci.}, (111):1--169, 2010.

\bibitem[Ogu90]{Ogus}
Arthur Ogus.
\newblock A {$p$}-adic analogue of the {C}howla-{S}elberg formula.
\newblock In {\em {$p$}-adic analysis ({T}rento, 1989)}, volume 1454 of {\em
  Lecture Notes in Math.}, pages 319--341. Springer, Berlin, 1990.

\bibitem[O'S05]{OS}
Peter O'Sullivan.
\newblock The structure of certain rigid tensor categories.
\newblock {\em C. R. Math. Acad. Sci. Paris}, 340(8):557--562, 2005.

\bibitem[Rie16]{Riess}
Ulrike Rie\ss.
\newblock On {B}eauville's conjectural weak splitting property.
\newblock {\em Int. Math. Res. Not. IMRN}, (20):6133--6150, 2016.

\bibitem[Sch15]{Sch}
Stefan Schreieder.
\newblock On the construction problem for {H}odge numbers.
\newblock {\em Geom. Topol.}, 19(1):295--342, 2015.

\bibitem[Ser72]{SerreGal}
Jean-Pierre Serre.
\newblock Propri\'{e}t\'{e}s galoisiennes des points d'ordre fini des courbes
  elliptiques.
\newblock {\em Invent. Math.}, 15(4):259--331, 1972.

\bibitem[SK79]{ShioF}
Tetsuji Shioda and Toshiyuki Katsura.
\newblock On {F}ermat varieties.
\newblock {\em T\^ohoku Math. J. (2)}, 31(1):97--115, 1979.

\bibitem[Tat66]{Tate}
John Tate.
\newblock Endomorphisms of abelian varieties over finite fields.
\newblock {\em Invent. Math.}, 2:134--144, 1966.

\bibitem[Voe00]{TMF}
Vladimir Voevodsky.
\newblock Triangulated categories of motives over a field.
\newblock In {\em Cycles, transfers, and motivic homology theories}, volume 143
  of {\em Ann. of Math. Stud.}, pages 188--238. Princeton Univ. Press,
  Princeton, NJ, 2000.

\bibitem[Voi22]{Voisin}
Claire Voisin.
\newblock On the {L}efschetz standard conjecture for {L}agrangian covered
  hyper-{K}\"{a}hler varieties.
\newblock {\em Adv. Math.}, 396:Paper No. 108108, 29, 2022.

\bibitem[Wil17]{WildShi}
J\"{o}rg Wildeshaus.
\newblock Intermediate extension of {C}how motives of {A}belian type.
\newblock {\em Adv. Math.}, 305:515--600, 2017.

\bibitem[Wüs89]{Wus}
G.~Wüstholz.
\newblock Algebraische {P}unkte auf analytischen {U}ntergruppen algebraischer
  {G}ruppen.
\newblock {\em Ann. of Math. (2)}, 129(3):501--517, 1989.

\end{thebibliography}
\end{document}